\documentclass[10pt,openany]{bookkk}
\usepackage{amsmath,amsthm,amsfonts,amssymb,amscd}
\usepackage[verbose,includemp=false, paperwidth=15.50cm, paperheight=21cm, body={10.50cm,16cm}, rmargin= 1.5cm, twosideshift=0pt]{geometry}
\usepackage[cam,letter,center]{crop}
\usepackage[english]{babel}
\usepackage[applemac]{inputenc}
\usepackage[all]{xy}
\usepackage{cite}
\usepackage{graphicx}

\usepackage{calc}

\theoremstyle{plain}% default
\newtheorem{thm}{Theorem}[section]
\newtheorem{lem}[thm]{Lemma}
\newtheorem{prop}[thm]{Proposition}
\newtheorem{cor}[thm]{Corollary}

\theoremstyle{definition}
\newtheorem{defn}{Definition}[section]
\newtheorem{conj}{Conjecture}[section]
\newtheorem{exmp}{Example}[section]

\theoremstyle{remark}
\newtheorem{rem}{Remark}

\DeclareMathOperator{\rank}{\text{\rm rank}}

\makeatletter %eliminate caption on withe pages
\def\cleardoublepage{\clearpage\if@twoside \ifodd\c@page\else
\hbox{}
\thispagestyle{empty}
\newpage
\if@twocolumn\hbox{}\newpage\fi\fi\fi}
\makeatother

\begin{document}

\title{A Survey on Hyperbolicity of Projective Hypersurfaces}
\author{Simone Diverio and Erwan Rousseau \\ IMPA, Rio de Janeiro, September 2010}
\date{}

\maketitle

\frontmatter

\chapter{Introduction}

These are lecture notes of a course held at IMPA, Rio de Janiero, in september 2010: the purpose was to present recent results on Kobayashi hyperbolicity in complex geometry. 

This area of research is very active, in particular because of the fascinating relations between analytic and arithmetic geometry. After Lang and Vojta, we have precise conjectures between analytic and arithmetic hyperbolicity, \textsl{e. g.} existence of Zariski dense entire curves should correspond to the density of rational points. Our ultimate goal is to describe the results obtained in \cite{D-M-R10} on questions related to the geometry of entire curves traced in generic complex projective hypersurfaces of high degree. 

For the convenience of the reader, this survey tries to be as self contained as possible. Thus, we start by recalling the basic definitions and concepts of complex hyperbolic geometry. Our presentation will focus later on the concepts of jet bundles and jet differentials which turn out to be the crucial tools that have been applied successfully in the last decades. These ideas date back to the work of Bloch \cite{Blo} and have been developed later by many other people (let use cite for example, Green and Griffiths \cite{G-G80}, Siu and Yeung \cite{S-Y97}, Demailly \cite{Dem97}...).

The presentation of the main techniques is certainly inspired by the notes \cite{Dem97} but many progress have been achieved since these notes were written and it seemed to us quite useful to update them.

Let us now describe the contents of this survey. In chapter one we introduce the classical Poincar distance on the complex unit disc and, following Kobayashi, we use it to construct an invariant pseudodistance on any complex space $X$ by means of chains of holomorphic discs. The complex space $X$ will be said (Kobayashi) hyperbolic if this pseudodistance is actually a true distance. We then present an infinitesimal form of this pseudodistance which reveals to be very useful in order to characterize the hyperbolicity of compact complex manifold in terms of the existence or not of entire curves (non-constant holomorphic maps from the entire complex plane) in it: this is the content of the famous Brody's criterion for hyperbolicity. We then end the chapter with a couple of applications of Brody's criterion to deformations of compact complex manifolds and to hyperbolicity of complex tori and with a general discussion about uniformization and hyperbolicity in complex dimension one in order to put in perspective the new difficulties which come up in higher dimension.

Chapter 2 deals with the notion of algebraic hyperbolicity. In the case of projective varieties, people have looked for a characterization of hyperbolicity depending only on algebraic subvarieties. Here, we focus on the so called algebraic hyperbolicity as in Demailly \cite{Dem97}, which is by definition a uniform lower bound on the genus of algebraic curves in terms of their degree once a polarization is fixed. We first discuss a nowadays classical result of Bogomolov about the finiteness of rational and elliptic curves on algebraic surfaces of general type with positive Segre class. Then, motivated by the Kobayashi conjecture which predicts the hyperbolicity of generic projective hypersurfaces of high degree, we explain an algebraic analogue of this conjecture which has been proved in the works of Clemens \cite{Cle86}, Ein \cite{Ein88}, Voisin \cite{Voi96} and Pacienza \cite{Pac04}. We focus here on the approach coming from ideas of Voisin which makes an essential use of the universal family $\mathcal X \subset \mathbb P^{n+1}\times\mathbb P(H^0(\mathbb P^{n+1},\mathcal O_{\mathbb P^{n+1}}(d)))$ of projective hypersurfaces in $\mathbb P^{n+1}$ of a given degree $d>0$. This object turns out to be very useful because of the positivity properties of its tangent bundle. The existence of sufficiently many vector fields with controlled pole order is used to prove that generic projective hypersurfaces satisfy the conjecture of Lang claiming that a projective manifold is hyperbolic if and only if all its subvarieties are of general type.

Starting from Chapter 3, we enter in the core of this survey, turning to the study of transcendental objects. We describe in detail the constructions of jet bundles (introduced in this formalism by \cite{G-G80}) following closely the presentation of projectivized jet bundles of Demailly \cite{Dem97} as an inductive procedure in the category of directed manifolds $(X,V)$ where $V$ is a holomorphic subbundle of the tangent bundle $T_X$. This tower of projectivized bundles is naturally endowed with tautological line bundles at each stage. Considering the sheaf of sections of the direct images of these line bundles leads to the concepts of (invariant) jet differentials which are more concretely interpreted as algebraic differential operators $Q(f',f'',\dots,f^{(k)})$ acting on jets of germs of holomorphic curves. The algebraic structure of these vector bundles $E_{k,m}T_X^*$ of invariant jet differentials leads to interesting (and difficult) questions in invariant theory which were intensively investigated recently \cite{Rou06a}, \cite{Mer08}, \cite{B-K10}.

In Chapter 4, we begin by recalling classical notions of hermitian geometry, such as curvature and positivity of hermitian line bundles on complex manifolds. A basic idea is that Kobayashi hyperbolicity is somehow related with suitable properties of negativity of the curvature of the manifold even in dimension greater than one. We formalize this heuristic concept by means of the Ahlfors-Schwarz lemma in connection with invariant jet differentials:  we illustrate the general philosophy whose key point is that global jet differentials vanishing along an ample divisor provide algebraic differential equations which every entire curve must satisfy.

It is then possible to state a general strategy which leads to sufficient conditions in order to have algebraic degeneracy of entire curves in a given compact complex manifold. The first step consists in finding a global section of the bundle of jet differentials vanishing on an ample divisor. The second step should produce much more differential equations, enough to impose sufficiently many conditions on the entire curves to force their algebraic degeneracy. One way to do this is to generalize the ideas described in Chapter 2 about vector fields. Following the strategy of Siu \cite{Siu04}, one should now consider vector fields tangent to the jet space. As a jet differential is after all a function on the jet space, one can differentiate it with vector fields and obtain new jet differentials. Of course, one has to guarantee that these new differential operators still vanish on an ample divisor. So, one is forced to have a precise control of the pole order of the vector fields constructed (which, for example in the case of projective hypersurface, should not depend on the degree of the hypersurface itself). 

The general strategy presented in Chapter 4 is not directly applicable to deal with projective hypersurfaces. To illustrate the modification needed in order to be able to run it, we present in Chapter 5 the solution of the Kobayashi conjecture for generic surfaces in projective $3$-space,  after \cite{McQ99}, \cite{D-EG00} and \cite{Pau08}.
In particular, we show how to find global invariant jet differentials vanishing along an ample divisor on a projective surface of general type by means of Riemann-Roch-type computations together with a vanishing theorem for the higher cohomolgy groups by Bogomolov. Then we explain in great details how to produce meromorphic vector fields of controlled pole order on the universal family of degree $d$ surfaces in $\mathbb P^3$. Finally, with these two ingredients available, we adapt the aforesaid general strategy to obtain the conclusion that very generic projective surfaces of degree greater than or equal to $90$ in projective $3$-space are Kobayashi hyperbolic.
This is far from being an optimal bound and it is even far from the bound obtained independently by Mc Quillan, Demailly-El Goul and P{\u{a}}un, but the strategy presented here is the only one which we were able to generalize in higher dimension.

The last chapter is devoted to the recent result on algebraic degeneracy of entire curves in generic projective hypersurfaces of high degree obtained in \cite{D-M-R10}.
In the higher dimensional case, the non-vanishing of the higher cohomology groups creates new conceptual difficulties. On the other hand, the extension to all dimensions of the existence of lots of meromorphic vector fields with controlled pole order presents \lq\lq only\rq\rq{} new technical difficulties while the conceptual nature of the construction remains the same of the one described in Chapter 5. Therefore, we have decided to concentrate ourself more on the general proof of the existence of global invariant jet differentials --- first in dimension three, then in the general case.

One way to control the cohomology is to use the holomorphic Morse inequalities of Demailly. If one can compute the Euler characteristic of the bundle of jet differentials $E_{k,m}T_X^*$ and then find upper bounds for the higher even cohomology groups $H^{2i}(X,E_{k,m}T_X^*)$ using the weak Morse inequalities, the first step is achieved as in dimension $3$ \cite{Rou06a}. 
Unfortunately, in general the control of the cohomology is quite involved, thus one try to apply directly the strong Morse inequalities to twisted tautological bundles on the projectivized jet bundles. This permits to obtain global jet differentials on hypersurfaces of sufficiently high degree \cite{Div09} in every dimension. 

Then Siu's strategy of exhibiting vector fields is realized on the jet spaces of the universal hypersurfaces \cite{Rou07}, \cite{Mer09}. Finally, the full strategy is used to obtain the algebraic degeneracy of entire curves in generic projective hypersurfaces of degree larger than $2^{n^5}$ \cite{D-M-R10}.

\bigskip

Last but not least, we would like to warmly thank Alcides Lins Neto, Jorge Vitrio Pereira, Paulo Sad and all the people of the IMPA for having organized this course and our stay in Rio. These have been very stimulating, interesting and, why not, funny days.

\tableofcontents

\mainmatter

\chapter{Kobayashi hyperbolicity: basic theory}

{\small\textsc{Abstract}. In this first chapter we state and describe the basic definitions of complex hyperbolic geometry, basically following \cite{Kob98} and \cite{Dem97}. Then, we state and prove the classical Brody's lemma and Picard's theorem. We conclude by giving a brief account of elementary examples and describing the case of Riemann surfaces.}

\section{The Kobayashi distance}

Let $\Delta\subset\mathbb C$ be the unit disc in the complex plane, with complex coordinate $\zeta$. On $\Delta$, there exists a particular (non-euclidean) metric, whose infinitesimal form is given by
$$
ds^2=\frac{d\zeta\otimes d\overline\zeta}{(1-|\zeta|^2)^2},
$$
which enjoys several interesting properties: its name is \emph{Poincar metric}. This is the starting point of the theory of complex hyperbolicity: the idea is to give to each complex space an intrinsic metric built by means of holomorphic mapping from the unit complex disc together with the Poincar metric to the given space.

More precisely, call $\rho$ the integrated form of the Poincar metric; we write here its explicit form even if we shall rarely use it:
$$
\rho(a,b)=\tanh^{-1}\left|\frac{a-b}{1-a\overline b}\right|,\quad a,b\in\Delta.
$$
The distance $\rho$ is complete on $\Delta$. Next, let $X$ be a complex space. We call a \emph{holomorphic disc} in $X$ a holomorphic map from $\Delta$ to $X$. Given two points $p,q\in X$, consider a \emph{chain of holomorphic discs} from $p$ to $q$, that is a chain of points $p=p_0,p_1,\dots,p_k=q$ of $X$, pairs of point $a_1,b_1,\dots,a_k,b_k$ of $\Delta$ and holomorphic maps $f_1,\dots,f_k\colon\Delta\to X$ such that
$$
f_i(a_i)=p_{i-1},\quad f_i(b_i)=p_i,\quad i=1,\dots,k.
$$
Denoting this chain by $\alpha$, define its length $\ell(\alpha)$ by
$$
\ell(\alpha)=\rho(a_1,b_1)+\cdots+\rho(a_k,b_k)
$$ 
and a pseudodistance $d_X$ on $X$ by 
$$
d_X(p,q)=\inf_{\alpha}\ell(\alpha).
$$
This is the \emph{Kobayashi pseudodistance} of $X$.

\begin{defn}
The complex space $X$ is said to be \emph{Kobayashi hyperbolic} if the pseudodistance $d_X$ is actually a distance.
\end{defn}

For $\Delta$ the complex unit disc, it is easy to see using the usual Schwarz-Pick lemma\footnote{We recall here that the usual Schwarz-Pick lemma says that for $f\colon\Delta\to\Delta$ a holomorphic map, one has the following inequality:
$$
\frac{|f'(\zeta)|}{1-|f(\zeta)|^2}\le\frac{1}{1-|\zeta|^2}.
$$
This means exactly that holomorphic maps contract the Poincar metric.} in one direction and the identity transformation in the other that $d_X=\rho$. Then $\Delta$ is hyperbolic. The entire complex plane is not hyperbolic: indeed the Kobayashi pseudodistance is identically zero. To see this from the very definition, take any two point $z_1,z_2\in\mathbb C$ and consider a sequence of holomorphic discs 
$$
\begin{aligned}
f_j\colon & \Delta\to\mathbb C \\
& \zeta\to z_1+j\zeta(z_2-z_1).
\end{aligned}
$$
It is important to remark here that the non hyperbolicity of the complex plane is connected to the possibility of taking larger and larger discs in $\mathbb C$. 

It is immediate to check that the Kobayashi pseudodistance has the fundamental property of being contracted by holomorphic maps: given two complex spaces $X$ and $Y$ and a holomorphic map $f\colon X\to Y$ one has for every pair of point $x,y$ in $X$
$$
d_Y(f(x),f(y))\le d_X(x,y).
$$ 
In particular, biholomorphisms are isometry for the Kobayashi metric.

\subsection{Infinitesimal form}

Let us now come at the infinitesimal analogue of the Koboyashi pseudodistance introduced above. For simplicity, we shall suppose that $X$ is a \emph{smooth} complex manifold but most of the things would work on an arbitrary singular complex space.

So, fix an arbitrary holomorphic tangent vector $v\in T_{X,x_0}$, $x_0\in X$: we want to give it an intrinsic length. Thus, define
$$
\mathbf{k}_{X}(v)=\inf\{\lambda>0\mid\exists f\colon\Delta\to X,\,f(0)=x_0,\,\lambda f'(0)=v\},
$$
where $f\colon\Delta\to X$ is holomorphic. Even with this infinitesimal form, it is straightforward to check that holomorphic maps between complex manifolds contract it and that in the case of the complex unit disc, it agrees with the Poincar metric.

One can give a similar definition in the setting of complex directed manifolds, that is pairs $(X,V)$ where $X$ is a complex manifold and $V\subset T_X$ a holomorphic subbundle of the tangent bundle: in this case one only considers vectors and maps which are tangent to $V$.

\begin{defn}
Let $(X,V)$ be a complex directed manifold and $\omega$ an arbitrary hermitian metric on $V$. We say that $(X,V)$ is \emph{infinitesimally Kobayashi hyperbolic} if $\mathbf{k}_{(X,V)}$ is positive definite on each fiber and satisfies a uniform lower bound
$$
\mathbf{k}_{(X,V)}(v)\ge\varepsilon ||v||_{\omega}
$$ 
when $v\in V_x$ and $x\in X$ describes a compact subset of $X$.
\end{defn}

The Kobayashi pseudodistance is the integrated form of the corrseponding infinitesimal pseudometric (this is due to Royden).

\begin{thm}
Let $X$ be a complex manifold. Then
$$
d_X(p,q)=\inf_\gamma\int_\gamma\mathbf{k}_{X}(\gamma'(t))\,dt,
$$
where the infimum is taken over all piecewise smooth curves joining $p$ to $q$.
\end{thm}

In particular, if $X$ is infinitesimally hyperbolic, then it is hyperbolic.

\section{Brody's criterion for hyperbolicity}

The distance decreasing property together with the fact that the Kobayashi pseudodistance is identically zero on $\mathbb C$, implies immediately

\begin{prop}
If $X$ is a hyperbolic complex space, then every holomorphic map $f\colon\mathbb C\to X$ is constant.
\end{prop} 

\begin{proof}
For $a,b\in\mathbb C$ one has
$$
d_X(f(a),f(b))\le d_{\mathbb C}(a,b)=0.
$$
Hence $f(a)=f(b)$.
\end{proof}

Next theorem, which is due to Brody, is the simplest and most useful criterion for hyperbolicity. It is a converse of the preceding proposition in the case where the target $X$ is compact. Fix any hermitian metric $\omega$ on the compact complex manifold $X$; we say that a holomorphic map $f\colon\mathbb C\to X$ is an \emph{entire curve} if it is non constant and that it is a \emph{Brody curve} if it is an entire curve with bounded derivative with respect to $\omega$ (or, of course, any other hermitian metric).

\begin{thm}[Brody]\label{Brody}
Let $X$ be a compact complex manifold. If $X$ is not (infinitesimally) hyperbolic then there exists a Brody curve in $X$.
\end{thm}

A first direct consequence of this theorem is that in the compact case, hyperbolicity and infinitesimal hyperbolicity are equivalent, since if $X$ is not infinitesimally hyperbolic then there exists an entire curve in $X$ and then two distinct points on this curve will have zero distance with respect to $d_X$. 
For more information on the localization of such a curve, we refer the reader to the very recent and remarkable results of \cite{Duv08}.

We shall start with the so called Brody reparametrization lemma.

\begin{lem}[Reparametrization lemma]
Let $X$ be a hermitian manifold with hermitian metric $\omega$ and $f\colon\Delta\to X$ a holomorphic map. Then, for every $\varepsilon>0$ there exists a radius $R\ge(1-\varepsilon)||f'(0)||_\omega$ and a homographic transformation $\psi$ of the disc $\Delta_R$ of radius $R$ onto $(1-\varepsilon)\Delta$ such that
$$
||(f\circ\psi)'(0)||_{\omega}=1\quad\text{and}\quad ||(f\circ\psi)'(t)||_\omega\le\frac 1{1-|t|^2/R^2},
$$
for every $t\in\Delta_R$.
\end{lem}

\begin{proof}
Consider the norm of the differential
$$
f'((1-\varepsilon)\zeta)\colon T_\Delta\to T_X
$$
with respect to the Poincar metric $|d\zeta|^2/(1-|\zeta|^2)^2$ on the unit disc, which is conformally invariant under $\operatorname{Aut}(\Delta)$: it is given by $(1-|\zeta|^2)||f'((1-\varepsilon)\zeta)||_\omega$. Pick a $\zeta_0\in\Delta$ such that this norm is maximum. We let for the moment $R$ to be determined and we chose an automorphism $\phi$ of the unit disc such that $\phi(0)=\zeta_0$; finally, set $\psi=(1-\varepsilon)\phi(t/R)$. Thus, $\psi'(0)=(1-\varepsilon)\phi'(0)/R=(1-\varepsilon)(1-|\zeta_0|^2)/R$, since $\phi$ is an isometry of the unit disc for the Poincar metric. On the other hand, we want $||(f\circ\psi)'(0)||_{\omega}=1$, which imposes $\psi'(0)=1/||f'(\psi(0))||_{\omega}$. This gives 
$$
\begin{aligned}
R & =(1-\varepsilon)(1-|\zeta_0|^2)||f'(\psi(0))||_{\omega} \\
& =(1-\varepsilon)(1-|\zeta_0|^2)||f'((1-\varepsilon)\zeta_0)||_{\omega} \\
& \ge (1-\varepsilon)||f'(0)||_\omega,
\end{aligned}
$$
since $\zeta_0$ is the maximum for $(1-|\zeta|^2)||f'((1-\varepsilon)\zeta)||_\omega$. Finally, we have
$$
\begin{aligned}
||(f\circ\psi)'(t)||_\omega & = |\psi'(t)|\,||f'(\psi(t))||_\omega \\
& = |\psi'(t)|\,||f'((1-\varepsilon)\phi(t/R))||_\omega \\
& \le |\psi'(t)|\frac{(1-|\zeta_0|^2)||f'((1-\varepsilon)\zeta_0)||_\omega}{1-|\phi(t/R)|^2} \\
& =\frac{1-\varepsilon}R|\phi'(t/R)|\frac{(1-|\zeta_0|^2)||f'((1-\varepsilon)\zeta_0)||_\omega}{1-|\phi(t/R)|^2} \\
& = \frac{|\phi'(t/R)|}{1-|\phi(t/R)|^2}\le\frac 1{1-|t|^2/|R|^2},
\end{aligned}
$$
by the choice of $R$ and the (rescaled) Schwarz-Pick lemma.
\end{proof}

We can now prove Brody's theorem.

\begin{proof}[Proof of Theorem \ref{Brody}]
By hypothesis, there exists a sequence of holomorphic discs $f_j\colon\Delta\to X$ such that $f_j(0)=x_0\in X$ and $||f'_j(0)||_\omega$ tends to infinity. We shall now apply the reparametrization lemma to this sequence in order to construct an entire curve. Fix for instance $\varepsilon=1/2$ and for each $j$, select a $\psi_j$ with the corresponding $R_j\ge||f'_j(0)||_\omega/2$ and call the composition $g_j=f_j\circ\psi_j\colon\Delta_{R_j}\to X$. So, we have that 
$$
||g_j'(0)||_\omega=1,\quad ||g_j'(t)||_\omega\le\frac 1{1-|t|^2/R_j^2}\quad\text{and}\quad R_j\to\infty.
$$
The conclusion will follow from the Ascoli-Arzel theorem: the family $\{g_j\}$ is with values in a compact space (hence pointwise bounded) and the estimate on its derivatives shows that it is equiLipschitz on any given compact subset of a fixed $\Delta_{R_j}$, hence equicontinuous in there. Thus, by a diagonal process, we have a subsequence which converges to an entire map $g\colon\mathbb C\to X$; moreover, $||g'(0)||_\omega=\lim ||g_j'(0)||_\omega=1$, so that $g$ is non constant and also $||g'(t)||_\omega=\lim ||g'_j(t)||_\omega\le 1$.
\end{proof}

The absence of entire holomorphic curves in a given complex manifold is often referred as \emph{Brody hyperbolicity}. Thus, in the compact case, Brody hyperbolicity and Kobayashi hyperbolicity coincide.  

The following example shows that in the non compact case one may have Brody hyperbolic domains which are not Kobayashi hyperbolic.

\begin{exmp}
Consider the domain in $\mathbb C^2$ defined by
$$
D=\{(z,w)\in\mathbb C^2\mid |z|<1,|zw|<1\}\setminus\{(0,w)\mid |w|\ge 1\}.
$$
The mapping $h\colon D\to\mathbb C^2$ which sends $(z,w)\mapsto (z,zw)$ has as image the unit bidisc and is one-to-one except on the set $\{z=0\}$. If $f\colon\mathbb C\to D$ is holomorphic, then $h\circ f$ is constant by Liouville's theorem. Thus, either $f$ is constant or $f$ maps $\mathbb C$ into the set $\{(0,w)\in D\}$. But this set is equivalent to the unit disc, hence $f$ is constant in any case.  

Since $h$, being holomorphic, is distance decreasing with respect to the Kobayashi pseudodistances, we have that $d_D(p,q)>0$ for $p\ne q$ unless both $p$ and $q$ lie in the subset $\{(0,w)\in D\}$. Suppose then that we are in this case and let $p=(0,b)$, $q=(0,0)$ and $p_n=(1/n,b)$. Since, as it is straightforward to see, the Kobayashi pseudodistance is always continuous as a map from a complex space times itself to the real numbers, we have $d_D(p,q)=\lim d_D(p_n,q)$. Call $a_n=\min\{n,\sqrt{n/|b|}\}$. Then the mapping $\Delta\ni \zeta\mapsto (a_n/n\,\zeta,a_nb\,\zeta)\in D$ maps $1/a_n$ to $p_n$. Hence
$$
\lim d_D(p_n,q)\le\lim\rho(1/a_n,0)=0.
$$
\end{exmp}

\subsection{Applications}

Let us show two immediate consequences of Brody's criterion for hyperbolicity: the openness property of hyperbolicity and the study of entire curves in complex tori.

A holomorphic family of compact complex manifolds is a holomorphic proper submersion $\mathcal X\to S$ between two complex manifolds. 

\begin{prop}
Let $\pi\colon\mathcal X\to S$ a holomorphic family of compact complex manifolds. Then the set of $s\in S$ such that the fiber $X_s=\pi^{-1}(s)$ is hyperbolic is open in the Euclidean topology.
\end{prop}

\begin{proof}
Let $\omega$ be an arbitrary hermitian metric on $\mathcal X$ and $\{X_{s_n}\}_{s_n\in S}$ a sequence of non hyperbolic fibers, and let $s=\lim s_n$. By Brody's criterion one obtains a sequence of entire maps $f_n\colon\mathbb C\to X_{s_n}$ such that $||f'_n(0)||_\omega=1$ and $||f'_n||_\omega\le 1$. Ascoli's theorem shows that there is a subsequence of $f_n$ converging uniformly to a limit $f\colon\mathbb C\to X_s$, with $||f'(0)||_\omega=1$. Hence $X_s$ is not hyperbolic and the collection of non hyperbolic fibers is closed in $S$.
\end{proof}

Consider now an $n$-dimensional complex torus $T$, that is the additive quotient of $\mathbb C^n$ modulo a real lattice $\Lambda$ of rank $2n$. In particular, $\mathbb C^n$ is the universal Riemannian cover of $T$ with the flat metric and by the projection we obtain plenty of entire curves in $T$.

\begin{thm}
Let $X\subset T$ be a compact complex submanifold of a complex torus. Then $X$ is hyperbolic if and only if it does not contain any translated of a subtorus.
\end{thm}

\begin{proof}
If $X$ contains some translated of a subtorus, then it contains lots of entire curves and it is not hyperbolic.

Conversely, suppose $X$ is not hyperbolic. Then by Brody's criterion there exists an entire map $f\colon\mathbb C\to X$ such that $||f'||_\omega\le ||f'(0)||_\omega=1$, where $\omega$ is the flat metric on $T$ inherited from $\mathbb C^n$. This means that any lifting $\widetilde f=(\widetilde f,\dots,\widetilde f_n)\colon\mathbb C\to\mathbb C^n$ is such that
$$
\sum_{j=1}^n|f_j'|^2\le 1.
$$
Then, by Liouville's theorem, $\widetilde f'$ is constant and therefore $\widetilde f$ is affine. But then, up to translation $\widetilde f$ is linear and so the image of $f$ and its Zariski closure are subgroup of $T$.
\end{proof}

Another interesting application of Brody's criterion in connection with the algebraic properties of complex manifolds is the hyperbolicity of varieties with ample cotangent bundle: this will be treated in the next chapters.

\section{Riemann surfaces and uniformization}

If a rational function $f(z)$ is not a constant, then for any $\alpha\in\mathbb C$ the equation $f(z)=\alpha$ has a solution $z\in\mathbb C$ and if $f$ has a pole, then $f$ takes the value $\infty$ there. However, this is not the case in general, if $f$ is a holomorphic or meromorphic function on $\mathbb C$. For instance, the entire function $f(z)=e^z$ does not take the value $0$ nor $\infty$. Thus, the number of points which $f$ misses is two. In general, if a meromorphic function $f$, regarded as a holomorphic mapping into $\mathbb P^1$ misses a point of $\mathbb P^1$, this point is called an \emph{exceptional value}. The next very classical theorem shows that the above \lq\lq two\rq\rq{} is the maximum.

\begin{thm}[Little Picard's theorem]
The number of exceptional values of a non constant meromorphic function on $\mathbb C$ is at most two.
\end{thm}

One possible proof relies on the fact that the universal covering of, say $\mathbb C\setminus\{0,1\}$, is the unit disc and on the Liouville theorem; for the former property, the theory of modular curves, easily deduced from basic Riemann surface theory, tells us that the quotient of the upper half plane by the group $\Gamma(2)$ (i.e., the kernel of the map $SL_2(\mathbb Z)\to SL_2(\mathbb Z/2\mathbb Z)$) under the induced action of the modular group $SL_2(\mathbb Z)$ is analytically isomorphic to the complex plane minus two points. In particular, the universal cover of the plane minus two points is the upper plane. As the upper half plane is conformally equivalent to the unit disc, we could just as well have taken the disc. 

In other words, the theorem says that $\mathbb P^1$ minus three points or, equivalently, $\mathbb C$ minus two points is Brody hyperbolic. In the following example, we describe a nice consequence of the Little Picard's theorem in dimension two.

\begin{figure}[htbp]
\begin{center}
\includegraphics[height=6cm]{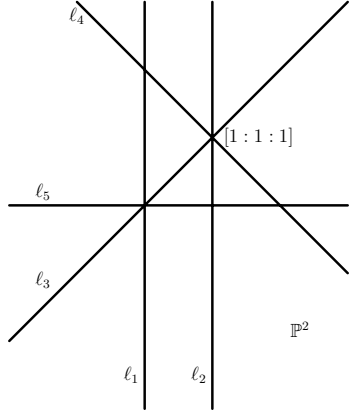}
\caption{A configuration of $5$ lines in $\mathbb P^2$ whose complement is hyperbolic}
\label{fig:p2minus5lines}
\end{center}
\end{figure}

\begin{exmp}
In the projective plane $\mathbb P^2$ with homogeneous coordinates $[Z_0:Z_1:Z_2]$, consider the following five lines as in \figurename~\ref{fig:p2minus5lines}:
$$
\begin{aligned}
& \ell_1\colon Z_1=0, \\
& \ell_2\colon Z_1-Z_0=0, \\
& \ell_3\colon Z_1-Z_2=0, \\
& \ell_4\colon Z_1+Z_2-2Z_0=0, \\
& \ell_5\colon Z_2=0
\end{aligned}
$$
and call $X=\mathbb P^2\setminus\bigcup_{j=1}^5\ell_j$. We claim that $X$ is Brody hyperbolic. To see this, let $\mathcal P$ be the pencil of lines $[\alpha:\beta]\mapsto \alpha\,Z_1+\beta\,Z_2-(\alpha+\beta)Z_0=0$ through the point $[1:1:1]$, restrict it to $X$ and consider the associated projection $\pi\colon X\to\mathbb P^1$. Then $\pi(X)=\mathbb P^1\setminus\{[1:0],[1:-1],[1:1]\}$, so that by the Little Picard's theorem $\pi(X)$ is Brody hyperbolic. Thus, if $f\colon\mathbb C\to X$ is holomorphic the composition $\pi\circ f$ must be constant and hence $f(\mathbb C)$ is contained in a fiber of $\pi$. Now, if a line in the pencil $\mathcal P$ is not contained in one of the three lines of the configuration passing through $[1:1:1]$, its intersection with $X$ is $\mathbb P^1$ minus three points (the point $[1:1:1]$ and the other two intersections with $\ell_1$ and $\ell_5$) so that each fiber of $\pi$ is again Brody hyperbolic by the Little Picard's theorem. The constancy of $f$ follows.

This example is not so special, since any configuration of five lines in $\mathbb P^2$ in generic position leads to the same result, see next chapter.
\end{exmp}

From a hyperbolic point of view, the case of one dimensional complex manifold is, in principle, completely settled by the following

\begin{thm}[Uniformization]
Let $X$ be a connected and simply connected one dimensional complex manifold. Then $X$ is biholomorphic either to the Riemann sphere, the complex plane or the complex unit disc.
\end{thm}

Consider a one dimensional complex manifold $X$ and its universal cover $\widetilde X$. Being $\widetilde X$ simply connected, the uniformization theorem gives us just three possibility. According to that, $X$ is said to be
\begin{itemize}
\item of \emph{elliptic} type, if $\widetilde X$ is $\mathbb P^1$,
\item of \emph{parabolic} type, if $\widetilde X$ is $\mathbb C$,
\item of \emph{hyperbolic} type, if $\widetilde X$ is $\Delta$.
\end{itemize} 
If $f\colon\mathbb C\to X$ is an entire curve then, being $\mathbb C$ simply connected, there exists a well defined holomorphic lifting $\widetilde f\colon\mathbb C\to\widetilde X$. Thus, if $X$ is of hyperbolic type, then $\widetilde f\colon\mathbb C\to\Delta$ must be constant by Liouville's theorem. So, $X$ is Brody hyperbolic and, if it is compact, it is Kobayashi hyperbolic: this is the case if and only if the geometric genus of $X$ is grater than or equal to two (the other possibility being the genus zero case $\mathbb P^1$ and the genus one case $\mathbb C/\Lambda$, the quotient of the complex plane by a lattice of rank two, that is elliptic curves).

In higher dimensions, the situation is much more involved: there is no longer any general uniformization property available.

Anyway, from an algebraic point of view the aforesaid trichotomy is governed by the asymptotic behavior of the (pluri)canonical system of the curve: it is a classical result that the growth rate of the space of global sections of the pluricanonical line bundle $h^0(X,K_X^{\otimes m})$ is linear with $m$ if and only if $X$ has genus greater than or equal to two.

In general, a manifold $X$ is said to be of \emph{general type} if its canonical line bundle $K_X$ is \emph{big}. We recall 

\begin{defn} 
A line bundle $L$ on a projective manifold $X$ of dimension $n$ is said to be
\begin{itemize}
\item \emph{big} if $h^0(X,L^{\otimes m}) \sim  m^n$ for $m \gg 0$,
\item \emph{very ample} if there is an embedding $\iota\colon X \hookrightarrow \mathbb{P}^N$ such that $L\simeq\iota^*\mathcal{O}_{\mathbb{P}^N}(1)$,
\item \emph{ample} if for some $m>0$, $L^{\otimes m}$ is very ample,
\item \emph{nef} if it has nonnegative degree $\int_Cc_1(L)\ge 0$ on any algebraic curve $C \subset X$.
\end{itemize}
\end{defn}

In particular, by Hirzebruch-Riemann-Roch formula and the Kodaira vanishing, manifolds with ample canonical bundle are of general type. These will be the main objects of study in the understanding of the links between hyperbolicity and algebraic properties of projective manifold.

\chapter{Algebraic hyperbolicity}

{\small\textsc{Abstract}. This chapter deals with an algebraic analogue of the Kobayashi hyperbolicity, introduced in \cite{Dem97}. We shall explain how the Kobayashi hyperbolicity implies restrictions on the ratio between the genus and the degree of algebraic curves contained in a hyperbolic projective algebraic manifold, and shall take this property as a definition. Then, we will discuss some general conjecture and related results, in particular Bogomolov's proof of the finiteness of rational and elliptic curves on surfaces whose Chern classes satisfy a certain inequality. In the second part, an account of known results about algebraic hyperbolicity of generic projective hypersurfaces of high degree will be done.}

\section{Hyperbolicity and genus of curves}

We shall make things here in the absolute case, but everything still works in the more general framework of directed manifolds.

Let $X$ be a compact Kobayashi hyperbolic manifold. Then $X$ is Brody hyperbolic and thus it cannot contain any holomorphic image of $\mathbb C$. In particular, from the algebraic point of view, $X$ cannot contain any rational nor elliptic curve (and, more generally, any complex torus). Hence, curves of genus $0$ and $1$ are prohibited by hyperbolicity. In fact, one can say something stronger.

\begin{prop}[\cite{Dem97}]
Let $X$ be a compact hermitian manifold, with hermitian metric $\omega$. If $X$ is (infinitesimally) hyperbolic then there exists a $\varepsilon_0> 0$ such that for every curve $C\subset X$ one has
$$
-\chi(\widehat C)=2\,g(\widehat C)-2\ge\varepsilon_0\deg_\omega C,
$$
where $\widehat C$ is the normalization of $C$ and $\deg_\omega C=\int_C\omega$.
\end{prop}  

\begin{proof}
Let $C\subset X$ be a curve in $X$ and $\nu\colon\widehat C\to C\subset X$ its normalization. Since $X$ is (infinitesimally) hyperbolic and compact, there is an absolute constant $\varepsilon>0$ such that
the infinitesimal Kobayashi pseudometric satisfies a uniform lower bound
$$
\mathbf{k}_X(v)\ge\varepsilon ||v||_\omega
$$
for every $v\in T_X$. Now, the universal Riemannian cover of $\widehat C$ is necessarily the complex unit disc, by the hyperbolicity of $X$: let it be $\pi\colon\Delta\to\widehat C$. We shall endow $\widehat C$ by the induced metric of constant negative Gaussian curvature $-4$ such that
$$
\pi^*\mathbf{k}_{\widehat C}=\mathbf{k}_\Delta=\frac{|d\zeta|}{1-|\zeta|^2}.
$$
Call $\sigma_\Delta=\frac i2\,d\zeta\wedge d\overline\zeta /(1-|\zeta|^2)^2$ and $\sigma_{\widehat C}$ the corresponding area measures. Then the classical Gauss-Bonnet formula yields
$$
-4\int_{\widehat C}\sigma_{\widehat C}=\int_{\widehat C}\Theta(T_{\widehat C},\mathbf{k}_{\widehat C})=2\pi\,\chi(\widehat C),
$$
where $\Theta(T_{\widehat C},\mathbf{k}_{\widehat C})$ is the curvature of $T_{\widehat C}$ with respect to the metric $\mathbf{k}_{\widehat C}$.

Next, if $\iota\colon C\to X$ is the inclusion, the distance decreasing property of the Kobayashi pseudometric applied to the holomorphic map $\iota\circ\nu\colon\widehat C	\to X$ gives
$$
\mathbf{k}_{\widehat C}(\xi)\ge\mathbf{k}_{X}((\iota\circ\nu)_*\,\xi)\ge\varepsilon\,||(\iota\circ\nu)_*\,\xi||_\omega,
$$
for all $\xi\in T_{\widehat C}$. From this, we infer that $\sigma_{\widehat C}\ge\varepsilon^2(\iota\circ\nu)^*\omega$, hence
$$
-\frac\pi 2\,\chi(\widehat C)=\int_{\widehat C}\sigma_{\widehat C}\ge\int\varepsilon^2(\iota\circ\nu)^*\omega=\varepsilon^2\int_C\omega.
$$
The assertion follows by putting $\varepsilon_0=2\,\varepsilon^2/\pi$.
\end{proof}

In other words, for $X$ a hyperbolic manifold, the ratio between the genus of curves and their degrees with respect to any hermitian metric (or any ample divisor) is bounded away from zero: this, following \cite{Dem97}, can be taken as a definition of \lq\lq algebraic\rq\rq{} hyperbolicity.

\begin{defn}
Let $X$ be a projective algebraic manifold endowed with any hermitian metric $\omega$ (for instance $\omega$ can be taken to be the curvature of any ample line bundle on $X$). We say that it is \emph{algebraically hyperbolic} if there exists a constant $\varepsilon_0 > 0$ such that for every algebraic curve $C\subset X$ one has
$$
2\,g(\widehat C)-2\ge\varepsilon_0\,\deg_\omega C.
$$
\end{defn}
When $\omega=i\,\Theta(A)$, where $A$ is any hermitian ample line bundle and $i\,\Theta(A)$ its Chern curvature, the right hand side of the inequality is just the usual degree of a curve in terms of its intersection product $C\cdot A$: in this case the inequality is purely algebraic.

By Riemann-Hurwitz formula, one can take, in the previous inequality of the definition of algebraic hyperbolicity, any finite morphism $f: C \to X$ from a smooth projective curve.

This algebraic counterpart of hyperbolicity satisfies an analogue of the openness property of the Kobayashi hyperbolicity, this time with respect to the Zariski topology.

\begin{prop}
Let $\mathcal X\to S$ an algebraic family of projective algebraic manifolds, given by a projective morphism. Then the set of $s\in S$ such that the fiber $X_s$ is algebraically hyperbolic is open with respect to the countable Zariski topology of $S$ (by definition, this is the topology for which closed sets are countable unions of algebraic sets).
\end{prop}

\begin{proof}
Without loss of generality, we can suppose that the total space $\mathcal X$ is quasi-projective. Let $\omega$ be the K\"ahler metric on $\mathcal X$ obtained by pulling-back the Fubini-Study metric via an embedding in a projective space. Fix integers $d>0$ and $g\ge 0$ and call $A_{d,g}$ the set of $s\in S$ such that $X_s$ contains an algebraic $1$-cycle $C=\sum m_j\,C_j$ with $\deg_\omega C=d$ and $g(\widehat C)=\sum m_j\,g(\widehat C_j)\le g$.

This set is closed in $S$, by the existence of a relative cycle space of curves of given degree and the lower semicontinuity with respect to the Zariski topology of the geometric genus.
But then, the set of non hyperbolic fibers is by definition
$$
\bigcap_{k>0}\quad\bigcup_{2g-2<d/k}A_{d,g}.
$$ 
\end{proof}

An interesting property of algebraically hyperbolic manifolds is
\begin{prop}
Let $X$ be an algebraically hyperbolic projective manifold and $V$ be an abelian variety. Then any holomorphic map $f\colon V \to X$ is constant.
\end{prop}

\begin{proof}
Let $m$ be a positive integer and $m_V\colon V \to V$, $s \mapsto m\cdot s$. Consider $f_m:= f \circ m_V$ and $A$ an ample line bundle on $X$. Let $C$ be a smooth curve in $V$ and ${f_m}|_{C}\colon C \to X$. Then
$$2g(C)-2 \geq \varepsilon \,C\cdot f_m^*A=\varepsilon m^2\,C\cdot f^*A.$$ Letting $m$ go to infinity, we obtain that necessarily $C\cdot f^*A=0$. Thus $f$ is constant on all curves in $V$ and therefore $f$ is constant on $V$.
\end{proof}

It is worthwhile here to mention that in the projective algebraic case, Kobayashi hyperbolicity and algebraic hyperbolicity are expected to be equivalent, but not much is known about it. Both of these properties should be equivalent to the following algebraic property.

\begin{conj}[Lang]
Let $X$ be a projective manifold. Then $X$ is hyperbolic if and only if there are no nontrivial holomorphic maps $V \to X$ where $V=\mathbb C^p/\Lambda$ is a compact complex torus.
\end{conj}

One may be tempted to extend the conjecture to non-projective manifolds but then it becomes false, as shown by the following

\begin{exmp}[\cite{Can}]
Let $X$ be a non-projective $K3$ surface\footnote{A $K3$ surface is a simply connected surface $X$ with irregularity $q(X)=h^1(X,\mathcal O_X)=0$ and trivial canonical bundle $K_X\simeq\mathcal O_X$. } with no algebraic curves (the existence of such a surface is a classical result on K3 surfaces). Then there exists a non-constant entire curve $f\colon \mathbb C \to X$. On the other hand, if $V$ is a compact torus every holomorphic map $F\colon V \to X$ is constant.

Let us justify briefly the claims of the example. The existence of non-constant entire curves is a consequence of the density of Kummer surfaces\footnote{Let $T$ be a two dimensional complex torus with a base point chosen. The involution $\iota\colon T\to T$ has exactly 16 fixed point, namely the points of order 2 on $T$, so that the quotient $T/\langle 1,\iota\rangle$ has sixteen ordinary double points. Resolving the double points we obtain a smooth surface $X$, the Kummer surface $Km(T)$ of $T$. Kummer surfaces are special case of $K3$ surfaces.} in the moduli space of $K3$ surfaces. Since Kummer surfaces contain lots of entire curves (inherited from the starting torus), one just has to apply Brody's theorem. The second claim follows from the non-existence of surjective maps $F\colon V \to X$. Indeed, considering $\Omega$ a non-vanishing holomorphic $2$-form on $X$, if $F$ is surjective, then $F^*\Omega$ is a non-zero section of the trivial bundle $\Lambda^2 T^*_V$: the rank of this $2$-form is therefore constant, equal to $2$. Then, one obtains that $F$ factors through a $2$-dimensional compact torus and induces a covering $V \to X$ which contradicts the fact that $X$ is simply connected.
\end{exmp}

Another characterization of hyperbolicity should be the following.

\begin{conj}[Lang]
Let $X$ be a smooth projective algebraic manifold. Then $X$ is hyperbolic if and only if all subvarieties of $X$ including $X$ itself are of general type.
\end{conj}

In the next section we shall see some partial result in this direction.
The latter conjecture should be put in perspective with this other celebrated one.

\begin{conj}[Green-Griffiths \cite{G-G80}, Lang]
Let $X$ be a smooth projective algebraic manifold of general type. Then there should exist a proper algebraic subvariety $Y\subsetneq X$ such that all entire curves $f\colon\mathbb C\to X$ have image $f(\mathbb C)$ contained in $Y$.
\end{conj}

This conjecture is largely open, too. Nevertheless, related to algebraic hyperbolicity we have the following.

\begin{thm}[Bogomolov \cite{Bog78}]
Let $X$ be  a smooth projective surface of general type with $c_1(X)^2>c_2(X)$. Then there are only finitely many rational or elliptic curves in $X$.
\end{thm}

\begin{proof}
We will see later in some details that the hypothesis on the second Segre number  $c_1(X)^2>c_2(X)$ implies that $h^0(X,S^mT_X^*)\sim m^3$.
A non-trivial symmetric differential $\omega \in H^0(X,S^mT_X^*)$ defines a multifoliation on $X$. Recall that there is an isomorphism (we will come back later on this, too)
$$
H^0(X,S^mT_X^*) \cong H^0(P(T_X),\mathcal{O}_{P(T_X)}(m)).
$$
If $\sigma \in H^0(X,S^mT_X^*)$ and $x \in X$ then $\sigma(x)$ defines naturally a polynomial of degree $m$ on $P(T_{X,x})\simeq \mathbb P^1$. The zeroes of $\sigma(x)$ determines the directions of the multifoliation. Let $C$ be a smooth projective curve and $f\colon C \to X$. The curve $f(C)$ is a leaf of the multifoliation defined by $\sigma$ if $f^*\sigma \in H^0(C,T_C^{*\otimes m})$ is trivial. Equivalently if $t_f\colon C \to P(T_X)$ is the lifting of $f$ then $f(C)$ is a leaf if $t_f(C)$ lies in the zero locus of $\sigma\in H^0(P(T_X),\mathcal{O}_{P(T_X)}(m))$.

The sections of $\mathcal{O}_{P(T_X)} (m)$ for $m$ large enough provide a rational map $\varphi\colon P(T_X) \to \mathbb P^N$ generically $1-1$ onto its image. Let us denote $Z_m \subset P(T_X)$ the union of the positive dimensional fibers of $\varphi$ and of the base locus of  $\mathcal{O}_{P(T_X)}(m)$.

Let $f\colon C \to X$ be a curve. Then, $f(C)$ is said to be irregular if $t_f(C) \subset Z_m$, otherwise it is regular. The set of irregular curves can be broken into $2$ sets: the curves that are leaves of multifoliations and the curves whose lifts lie on the positive dimensional fibers of $\varphi$.

Let $C'$ be a regular curve with normalization $f\colon C \to C' \subset X$. There is a symmetric differential $\sigma \in H^0(X,S^mT_X^*)$ such that  $f^*\sigma \in H^0(C,(T_C^*)^{\otimes m})$ is non-trivial but vanishes somewhere. Hence $\deg_C T_C^{*\otimes m}=m\deg K_C >0$ and $C$ cannot be rational or elliptic.

Let $C'$ be an irregular curve and write $Z_m= Z_m^1 \cup Z_m^2$ where $Z_m^1$ is the union of components not dominating $X$, $Z_m^2$ is the union of components dominating $X$. The number of curves that lift in $Z_m^1$ is clearly finite. The components of $Z_m^2$ have a naturally defined foliation on them. Curves whose lifts lie in $Z_m^2$ are leaves of these foliations. By Jouanolou's theorem on compact leaves of foliations, either there are finitely many compact leaves or they are fibers of a fibration. Thus there are finitely many such elliptic or rational curves: $X$ being of general type, the second situation is not possible since a surface of general type cannot  be ruled or elliptic.
\end{proof}

In the transcendental case, the only result for a quite general case has been obtained Mc Quillan in \cite{McQ98}, for $\dim X=2$ and the second Segre number $c_1(X)^2-c_2(X)$ of $X$ positive. The heart of his proof is 
\begin{thm}
Consider a (possibly singular) holomorphic foliation on a surface of general type. Then any parabolic leaf of this foliation is algebraically degenerate.
\end{thm}

An immediate corollary of the two previous results is a confirmation of the Green-Griffiths conjecture in this situation.

\begin{cor}
Let $X$ be a smooth projective surface of general type with $c_1(X)^2 >c_2(X)$. Then there are finitely many curves $C \subset X$ such that any non-constant entire curve takes value in one of these curves.
\end{cor}

Unfortunately, these \lq\lq order one\rq\rq{} techniques are insufficient to work with surfaces of degree $d$ in projective $3$-space. In this case in fact
$$
c_1(X)^2=d(d-4)^2<d(d^2-4d+6)=c_2(X),\quad d\ge 3.
$$

In higher dimensions, there are few results. For the algebraic version, let us mention the following result of Lu and Miyaoka.

\begin{thm}[\cite{Lu-Mi}]
Let $X$ be a projective manifold of general type. Then $X$ has only a finite number of nonsingular codimension-one subvarieties having pseudo-effective anticanonical divisor. In particular, $X$ has only a finite number of non singular codimension one Fano, Abelian and Calabi-Yau subvarieties.
\end{thm}

For some partial result in all dimensions for the transcendental case, we refer to next chapters.

\section[Algebraic hyperbolicity of projective hypersurfaces]{Algebraic hyperbolicity of generic projective hypersurfaces of high degree}

Consider the Grassmannian $\mathbb G(1,n+1)$ of projective lines in $\mathbb P^{n+1}$ which is canonically identified with the Grassmannian $Gr(2,n+2)$ of $2$-planes in $\mathbb C^{n+2}$: its complex dimension is $2n$. We are interested in understanding when a generic projective hypersurfaces $X\subset\mathbb P^{n+1}$ contains a line. Fix an integer $d>0$. Then a projective hypersurface of degree $d$ is an element of the linear system $|\mathcal O_{\mathbb P^{n+1}}(d)|$ or, equivalently, can be identified with a point in the projectivization $\mathbb P(H^0(\mathbb P^{n+1},\mathcal O_{\mathbb P^{n+1}}(d)))$. One has $\dim\mathbb P(H^0(\mathbb P^{n+1},\mathcal O_{\mathbb P^{n+1}}(d)))=N_d-1$, where $N_d=\binom{n+d+1}{n+1}=h^0(\mathbb P^{n+1},\mathcal O_{\mathbb P^{n+1}}(d))$ is the dimension of homogeneous polynomials of degree $d$ in $n+2$ variables.

Now, consider the incidence variety 
$$
\mathcal L=\{(\ell,X)\in\mathbb G(1,n+1)\times\mathbb P^{N_d-1}\mid\text{the line $\ell$ is contained in $X$}\}.
$$
By construction, the image of $\mathcal L$ in $\mathbb P^{N_d-1}$ by the second projection is the set of projective hypersurfaces of degree $d$ which contain at least one line. Of course, if $\dim\mathcal L$ is less than $N_d-1$, then a generic projective hypersurface of degree $d$ does not contain lines, since the second projection cannot be dominant.
On the other hand, $\mathcal L$ is always mapped onto $\mathbb G(1,n+1)$ by the first projection, since every line is always contained in some degree $d$ hypersurface. Next, an easy parameter computation shows that generically a homogeneous polynomial of degree $d$ in $n+2$ variables must satisfy $d+1$ condition in order to contain a line.
Therefore, the fiber of the first projection has dimension $N_d-d-2$ and thus $\dim\mathcal L=N_d+2n-d-2$.

After all, the second projection maps a variety of dimension $N_d+2n-d-2$ to a variety of dimension $N_d-1$ and so we have proved the following.

\begin{prop}
If $d\ge 2n$, then a generic projective hypersurface of degree $d$ in $\mathbb P^{n+1}$ cannot contain any line.
\end{prop}

This digression shows that if we are interested in hyperbolicity of generic projective hypersurfaces, we surely have to exclude low degree ones. On the other hand, by the Euler short exact sequence
$$
0\to\mathcal O_{\mathbb P^{n+1}}\to\bigoplus_{j=1}^{n+2}\mathcal O_{\mathbb P^{n+1}}(1)\to T_{\mathbb P^{n+1}}\to 0,
$$
combined with the classical adjunction formula 
$$
K_D\simeq (K_Y\otimes\mathcal O_Y(D))|_D
$$
for smooth divisors $D\subset Y$ in a smooth manifold $Y$, one finds straightforwardly, by taking determinants, that the canonical bundle of a smooth hypersurface $X$ of degree $d$ in projective $(n+1)$-space is given by
$$
K_X=\mathcal O_X(d-n-2).
$$
So, the higher the degree of the hypersurface $X$ is, the more positive its canonical bundle is. This is somehow consistent with the picture presented at the end of Chapter 1, where hyperbolicity was heuristically linked to the positivity properties of the canonical bundle.

More precisely, Kobayashi made the following.

\begin{conj}[\cite{Kob70}]
Le $X\subset\mathbb P^{n+1}$ be a generic projective hypersurfaces of degree $d$, $n\ge 2$. Then $X$ is Kobayashi hyperbolic if its degree is sufficiently high, say $d\ge 2n+1$.  
\end{conj}

This conjecture and the bound on the degree are closely related to the conjecture in the case of complements of hypersurfaces.

\begin{conj}[\cite{Kob70}]
Le $X\subset\mathbb P^{n}$ be a generic projective hypersurfaces of degree $d$. Then $\mathbb P^{n} \setminus X$ is Kobayashi hyperbolic if its degree is sufficiently high, say $d\ge 2n+1$.  
\end{conj}

One possible explanation for the bounds on the degrees comes, as far as we know, from the following facts. 
Consider in $\mathbb P^{n}$ with homogeneous coordinates $[Z_1:\cdots:Z_n]$ the divisor $D$ of degree $d$ defined by the homogeneous equation $P(Z)=0$. Then, one can construct a cyclic $d:1$ cover of $\mathbb P^n$ by taking in $\mathbb P^{n+1}$ with homogeneous coordinates $[Z_0:\cdots:Z_n]$ the divisor $X$ defined by $Z_0^d=P(Z_1,\dots,Z_n)$ together with its projection onto $\mathbb P^n$. This covering ramifies exactly along $D$ and thus all holomorphic map $f\colon\mathbb C\to\mathbb P^n\setminus D$ lift to $X$. It is then clear that the hyperbolicity of $\mathbb P^n\setminus D$ is intimately correlated with the hyperbolicity of $X$. On the other hand, if a holomorphic map $f\colon\mathbb C\to\mathbb P^{n}$ misses $2n+1$ or more hyperplanes in general position, then it is a constant map; this is the by now classical result of Dufresnoy \cite{Duf44} and Green \cite{Gre72}. Now, just remark that a configuration of $d$ hyperplanes in general position can be seen as a generic completely reducible divisor of degree $d$.

One has to notice anyway, that if one believes to the equivalence of Kobayashi and algebraic hyperbolicity in the projective algebraic setting then, as we shall see in the next section, this bound should probably be $d\ge 2n$, at least for $n\ge 6$ \cite{Pac04}. Anyway the state of the art on the subject is for the moment very far from this optimal bounds, no matter in which one we want to believe.  

The rest of this chapter will be devoted to prove several algebraic properties of generic projective hypersurfaces of high degree, such as their algebraic hyperbolicity and the property of their subvarieties of being of general type. 

\subsection{Global generation of the twisted tangent bundle of the universal family}
 
First, given a holomorphic vector bundle $E\to X$ over a compact complex manifold $X$, we say that $E$ is \emph{globally generated}, if the global sections evaluation maps
$$
H^0(X,E)\to E_x
$$
are surjective for all $x\in X$, where $E_x$ is the fiber of $E$ over the point $x$. If a vector bundle is globally generated, so are all its exterior powers, in particular its determinant, as it is easy to verify.

Now, consider the universal family of projective hypersurfaces in $\mathbb P^{n+1}$ of a given degree $d>0$. It is the subvariety $\mathcal X$ of the product $\mathbb P^{n+1}\times\mathbb P(H^0(\mathbb P^{n+1},\mathcal O_{\mathbb P^{n+1}}(d)))$  defined by the pairs $([x],X)$ such that $[x]\in X$. The starting point is the following global generation statement.

\begin{prop}[See \cite{Voi96},\cite{Siu04}]
The twisted tangent bundle 
$$
T_{\mathcal X}\otimes p^*\mathcal O_{\mathbb P^{n+1}}(1)
$$ 
is globally generated, where $p\colon\mathcal X\to\mathbb P^{n+1}$ is the first projection.
\end{prop}

\begin{proof}
We shall exhibit on an affine open set of $\mathcal X$ a set of generating holomorphic vector fields and then show that when extended to the whole space, the pole order of such vector fields in the $\mathbb P^{n+1}$-variables is one. 

Consider homogeneous coordinates $(Z_j)_{j=0,\dots,n+1}$ and $(A_{\alpha})_{|\alpha|=d}$ respectively on $\mathbb P^{n+1}$ and $\mathbb P^{N_d-1}$, where $\alpha=(\alpha_0,\dots,\alpha_{n+1})\in\mathbb N^{n+2}$ is a multi-index and $|\alpha|=\sum\alpha_j$. The equation of the universal hypersurface is then given by
$$
\sum_{|\alpha|=d}A_\alpha\,Z^\alpha=0,\quad Z^\alpha=Z_0^{\alpha_0}\cdots Z_{n+1}^{\alpha_{n+1}}.
$$
Next, we fix the affine open set $U=\{Z_0\ne 0\}\times\{A_{d0\cdots 0}\ne 0\}\simeq\mathbb C^{n+1}\times\mathbb C^{N_d-1}$ in $\mathbb P^{n+1}\times\mathbb P^{N_d-1}$ with the corresponding inhomogeneous coordinates $(z_j)_{j=1,\dots,n+1}$ and $(a_{\alpha})_{|\alpha|=d,\alpha_0<d}$. On this affine open set we have
$$
\mathcal X\cap U=\left\{\sum_{|\alpha|=d}a_{\alpha}\,z_1^{\alpha_1}\cdots z_{n+1}^{\alpha_{n+1}}=0\right\},\quad a_{d0\cdots 0}=1.
$$
Its tangent space in $\mathbb C^{n+1}\times\mathbb C^{N_d-1}\times\mathbb C^{n+1}\times\mathbb C^{N_d-1}$ with affine coordinates $(z_j,a_{\alpha},z_j',a_\alpha')$ is then given by the two equations
$$
\begin{cases}
\sum_{|\alpha|=d}a_{\alpha}\,z_1^{\alpha_1}\cdots z_{n+1}^{\alpha_{n+1}}=0,\quad a_{d0\cdots 0}=1 \\
\sum_{|\alpha|=d,\alpha_0<d}\sum_{j=1}^{n+1}\alpha_j\, a_\alpha\,z_1^{\alpha_1}\cdots z_j^{\alpha_j-1}\cdots z_{n+1}^{\alpha_{n+1}}\,z_j' \\ 
\quad +\sum_{|\alpha|=d,\alpha_0<d}z_1^{\alpha_1}\cdots z_{n+1}^{\alpha_{n+1}}\,a_{\alpha}'=0,
\end{cases}
$$
the second of which is obtained by formal derivation.
For any multi-index $\alpha$ with $\alpha_j\ge 1$, set
$$
V_\alpha^j=\frac{\partial}{\partial a_{\alpha}}-z_j\,\frac{\partial}{\partial a_\alpha^j},
$$
where $a_\alpha^j$ is obtained by the multi-index $a_{\alpha}$ lowering the $j$-th entry  by one.
It is immediate to verify that these vector fields are tangent to $\mathcal X_0$ and, by an affine change of coordinates, that once extended to the whole $\mathcal X$ it becomes rational with pole order equal to one in the $z$-variables.

Now consider a vector field on $\mathbb C^{n+1}$ of the form
$$
V_0=\sum_{j=1}^{n+1}v_j\,\frac{\partial}{\partial z_j},
$$
where $v_j=\sum_{k=1}^{n+1}v_{j,k}\,z_k+v_{j,0}$ is a polynomial of degree at most one in the $z$-variables. We can then modify it by adding some \lq\lq slanted\rq\rq{} direction in order to obtain a vector field tangent to $\mathcal X_0$ as follows. Let
$$
V=\sum_{|\alpha|=d,\alpha_{0}<d}v_\alpha\,\frac{\partial}{\partial a_{\alpha}}+V_0,
$$
where the $v_{\alpha}$'s have to be determined. The condition to be satisfied in order to be tangent to $\mathcal X_0$ clearly is
$$
\sum_{\alpha} v_\alpha\, z^\alpha+\sum_{\alpha,j}a_\alpha\,v_j\,\frac{\partial z^\alpha}{\partial z_j}\equiv 0
$$
and thus it suffices to select the $v_\alpha$ to be constants such that the coefficient in each monomial $z^\alpha$ is zero. Here, an affine change of variables shows that once the extension of $V$ to the whole $\mathcal X$ is taken, the pole order is at most one in the $z$-variables. 

It is then straightforward to verify that these packages of vector field generate the tangent bundle, and the poles are compensated by twisting by $\mathcal O_{\mathbb P^{n+1}}(1)$, since they appear at order at most one and only in the variables living in $\mathbb P^{n+1}$.
\end{proof}

\subsection{Consequences of the twisted global generation}

Two remarkable consequences of the twisted global generation of the tangent space of the universal family are the following. First, the very generic projective hypersurface of high degree (is of general type and) admits only subvarieties of general type, that is very generic projective hypersurfaces of high degree satisfy Lang's conjecture stated above, which is conjecturally equivalent to Kobayashi hyperbolicity. Second, very generic projective hypersurfaces are algebraically hyperbolic, which would be implied by their hyperbolicity (and should be in principle equivalent) as we have seen: this can be regarded as another evidence towards Kobayashi's conjecture.

\begin{thm}
Let $X\subset\mathbb P^{n+1}$ be a (very) generic projective hypersurface of degree $d\ge 2n+2$. If $Y\subset X$ is any subvariety, let $\nu\colon\widetilde Y\to Y$ be a desingularization. Then
$$
H^0(\widetilde Y,K_{\widetilde Y}\otimes\nu^*\mathcal O_{\mathbb P^{n+1}}(-1))\ne 0.
$$
\end{thm}

\begin{proof}
Let $\mathcal X\subset\mathbb P^{n+1}\times\mathbb P^{N_d-1}$ the universal hypersurface of degree $d$ and $\mathcal Y\subset\mathcal X$ be a subvariety such that the second projection $\mathcal Y\to\mathbb P^{N_d-1}$ is dominant of relative dimension $\ell$. For simplicity, we shall skip here a technical point which consists to allow an tale base change $U\to\mathbb P^{N_d-1}$ for the family. 

Let $\nu\colon\widetilde{\mathcal Y}\to\mathcal Y$ be a desingularization and consider an open dense subset $U\subset\mathbb P^{N_d-1}$ over which both $\widetilde{\mathcal Y}$ and $\mathcal X$ are smooth. What we have to show is that 
$$
H^0(\widetilde Y_s,K_{\widetilde Y_s}\otimes\nu^*\mathcal O_{\mathbb P^{n+1}}(-1))\ne 0,
$$
for $Y_s$ the fiber over a generic point $s\in U$. To this aim, observe that, since the normal bundle of a fiber in a family is trivial, 
$$
K_{\widetilde Y_s}\simeq K_{\widetilde{\mathcal Y}}|_{\widetilde Y_s}=\bigwedge^{k+N_d-1}T^*_{\widetilde{\mathcal Y}}\biggr |_{\widetilde Y_s},
$$
by adjunction and that 
$$
\bigwedge^{k+N_d-1} T^*_{\mathcal X}\biggr |_{X_s}\simeq K_{X_s}\otimes\bigwedge^{n-k}T_{\mathcal X}\biggr |_{X_s}
$$
by linear algebra and adjunction again.

Therefore, we have to show that $\bigwedge^{k+N_d-1}T^*_{\widetilde{\mathcal Y}}\otimes\nu^*\mathcal O_{\mathbb P^{n+1}}(-1)\bigr |_{\widetilde Y_s}$ is effective. Now, we have a map
$$
\bigwedge^{k+N_d-1} T^*_{\mathcal X}\otimes\mathcal O_{\mathbb P^{n+1}}(-1)\biggr |_{X_s}\to
\bigwedge^{k+N_d-1}T^*_{\widetilde{\mathcal Y}}\otimes\nu^*\mathcal O_{\mathbb P^{n+1}}(-1)\biggr |_{\widetilde Y_s}
$$
induced by the generically surjective restriction $T^*_{\mathcal X}\to T^*_{\widetilde{\mathcal Y}}$, which is non zero for a generic choice of $s\in U$.

It is then sufficient to prove that $K_{X_s}\otimes\bigwedge^{n-k}T_{\mathcal X}\bigr |_{X_s}\otimes\mathcal O_{X_s}(-1)$ is globally generated. Now, 
$$
K_{X_s}=\mathcal O_{X_s}(d-n-2)=\mathcal O_{X_s}((n-k)+(d-2n+k-2))
$$
and thus

\begin{multline*}
K_{X_s}\otimes\bigwedge^{n-k}T_{\mathcal X}\biggr |_{X_s}\otimes\mathcal O_{X_s}(-1)= \\
\bigwedge^{n-k}T_{\mathcal X}\otimes\mathcal O_{\mathbb P^{n+1}}(1)\biggr |_{X_s}\otimes\mathcal O_{X_s}(d-2n+k-3).
\end{multline*}

By the global generation of $T_{\mathcal X}\otimes\mathcal O_{\mathbb P^{n+1}}(1)$, the right hand term is globally generated as soon as $d\ge 2n+3-k$ so that $d\ge 2n+2$ will do the job.

We have thus proved that the theorem holds for the general fiber of the family $\mathcal Y$. To conclude, it suffices to let the family $\mathcal Y$ vary, that is to let vary the Hilbert polynomial. In this way we obtain the same statement for all subvarieties of $X_s$ outside a countable union of closed algebraic subvarieties of the parameter space $U$, that is for very generic $X$.
\end{proof}

\begin{cor}
Let $X\subset\mathbb P^{n+1}$ be a (very) generic projective hypersurface of degree $d\ge 2n+2$. Then any subvariety $Y\subset X$ (and of course $X$ itself) is of general type.
\end{cor}

\begin{proof}
This is an immediate consequence of the theorem above: such a subvariety has in fact a desingularization whose canonical bundle can be written as an effective divisor twisted by a big one (the pull-back by a modification of the ample divisor $\mathcal O_{\mathbb P^{n+1}}(1)$) and hence it is big. 
\end{proof}

This corollary can be sharpened as soon as $n\ge 6$, see \cite{Pac04}.

\begin{cor}
A very generic projective hypersurfaces in $\mathbb P^{n+1}$ of degree greater than or equal to $2n+2$ is algebraically hyperbolic.
\end{cor}

\begin{proof}
Let $\omega=i\,\Theta(\mathcal O_X(1))$ be the reference hermitian metric on $X$ and $C\subset X$ a curve. Consider the finite-to-one normalization morphism $\nu\colon\widetilde C\to C$, which is in fact a desingularization, if necessary. Then, the preceding theorem states that $K_{\widetilde C}\otimes\nu^*\mathcal O_{\mathbb P^{n+1}}(-1)$ is effective and so of non negative degree on $\widetilde C$. By the Hurwitz formula $c_1(K_{\widetilde C})=2g(\widetilde C)-2$ and thus
$$
-\chi(\widetilde C)=2g(\widetilde C)-2\ge\nu^*\mathcal O_{\mathbb P^{n+1}}(1)\cdot \widetilde C=\int_C\omega.
$$
\end{proof}

Another consequence of the global generation statement is the following result on the non-deformability of entire curves in projective hypersurfaces of high degree.

\begin{thm}[\cite{DPP}]
Consider $\mathcal X\subset\mathbb P^{n+1}\times\mathbb P^{N_d-1}$ the universal hypersurface of degree $d$, $U \subset \mathbb P^{N_d-1}$ an open set and $\Phi\colon\mathbb C \times U \to \mathcal X$ a holomorphic map such that $\Phi( \mathbb C \times \{t\}) \subset X_t$ for all $t \in U$. If $d \geq 2n+2$, the rank of $\Phi$ cannot be maximal anywhere.
\end{thm}

In other words, the Kobayashi conjecture may possibly fail only if there is an entire curve on a general hypersurface $X$ which is not preserved by a deformation of $X$.

Now, let us sketch the proof of the previous result.

\begin{proof}
Suppose that $\Phi\colon\mathbb C \times U \to \mathcal X$ has maximal rank and $U$
is the polydisc $ B(\delta_0)^{N_d-1}$. We consider the sequence of maps
$$
\Phi_k\colon \mathbb B(\delta_0 k)^{N_d} \to  \mathcal X
$$
given by $\Phi_k(z,\xi_1,\dots\xi_{N_d-1})=\Phi(zk^{N_d-1},\frac{1}{k}\xi_1,\dots,\frac{1}{k}\xi_{N_d-1}).$ The sections
$$
J_{\Phi_k}(z,\xi)=\frac{\partial \Phi}{\partial z} \wedge \frac{\partial \Phi}{\partial \xi_1}\wedge \cdots\wedge  \frac{\partial \Phi}{\partial \xi_{N_d-1}} (z,\xi) \in \Lambda^{N_d}T_{\mathcal{X},\Phi(z,\xi)}
$$
are not identically zero and we can assume $J_{\Phi_k}(0)$ non-zero. 
Thanks to the global generation statement of $T_\mathcal{X}\otimes \mathcal{O}_{\mathbb P^{n+1}}(1)$, we can choose $n-1$ vector fields
$$V_1, \dots, V_{n-1} \in T_\mathcal{X}\otimes \mathcal{O}_{\mathbb P^{n+1}}(1)$$
such that
$$
J_{\Phi_k}(0) \wedge \Phi_k^*(V_1 \wedge \dots \wedge V_{n-1}) \neq 0
$$
in $K_\mathcal{X}^{-1}\otimes \mathcal{O}_{\mathbb P^{n+1}}(n-1)_{\Phi_k(0)}.$
We consider the sections
$$\sigma_k=J_{\Phi_k}\wedge \Phi_k^*(V_1\wedge \dots \wedge V_{n-1}),$$
of $\Phi_k^*(K_\mathcal{X}^{-1}\otimes \mathcal{O}_{\mathbb P^{n+1}}(n-1))$ over the polydisk.
If $d \geq 2n+2$ the restriction of $K_\mathcal{X} \otimes \mathcal{O}_{\mathbb P^{n+1}}(1-n)$ over $U$ is ample and we can endow this bundle with a metric $h$ of positive curvature.
We consider the sequence of functions $f_k: \mathbb B(\delta_0 k)^{N_d} \to \mathbb R^+$ defined by
$$f_k(w)=||\sigma_k(w)||^{2/N_d}_{\Phi_k^*h^{-1}}.$$
The ampleness implies that there exists a positive $C$ such that
$$\Delta \log f_k \geq Cf_k.$$ This gives
$$f_k(0) \leq Ck^{-2},$$
and therefore $f_k(0) \to 0$ which contradicts the fact that, by construction, there exists a positive constant $b$ such that for all $k$, $f_k(0)=b$.
\end{proof}

Let us briefly describe the generalization of the above results to the logarithmic case, that is the case of complements of hypersurfaces. If $X$ is a $n$-dimensional complex manifold and $D$ a normal crossing divisor, {\sl i.e.} in local coordinates $D=\{z_1\dots z_l = 0\}$, $l\le n$, we call the pair $(X,D)$ a {\it log-manifold}.

In the case of complements we have the following notion stronger than hyperbolicity.

\begin{defn}
Let $(X,D)$ be a log-manifold and $\omega$ a hermitian metric on $X$. The complement $X \setminus D$ is said to be hyperbolically embedded in $X$, if there exists $\varepsilon >0$ such that for every $x \in X \setminus D$ and $\xi \in T_{X,x}$, we have $$k_X(\xi) \geq \varepsilon ||\xi||_\omega.$$
\end{defn}

To generalize to this setting the notion of algebraic hyperbolicity, we need to introduce the following.

\begin{defn}
Let $(X,D)$ be a log-manifold, $C \subset X$ a curve not contained in $D$ and $\nu\colon\widehat{C} \to C$ the normalization. Then we define $i(C,D)$ to be the number of distinct points in $\nu^{-1}(D)$.
\end{defn}

Then, we have the next.

\begin{defn}
The pair $(X,D)$ is \emph{algebraically hyperbolic} if there exists $\varepsilon> 0$ such that
$$
2g(\widehat{C})-2+i(C,D) \geq \varepsilon \deg_\omega(C)
$$
for all curves $C \subset X$ not contained in $D$.
\end{defn}

As in the compact case, analytic and algebraic hyperbolicity are closely related.

\begin{prop}[\cite{PaRou}]
Let $(X,D)$ be a log-manifold such that $X \setminus D$ is hyperbolic and hyperbolically embedded in $X$. Then $(X,D)$ is algebraically hyperbolic.
\end{prop}

The algebraic version of the Kobayashi conjecture is also verified.
\begin{thm}[\cite{PaRou}]
Let $X_d \subset \mathbb P^n$ be a very generic hypersurface of degree $d\geq 2n+1$ in $\mathbb P^n$. Then $(\mathbb P^n, X_d)$ is algebraically hyperbolic.
\end{thm}

\section{A little history of the above results}

The chronicle of the above results about algebraic hyperbolicity is the following. 

First, in \cite{Cle86} it is shown that if $X$ is a generic hypersurface of degree $d\ge 2$ in $\mathbb{P}^{n+1}$, then $X$ does not admit an irreducible family $f\colon\mathcal C\to X$ of immersed curves of genus $g$ and fixed immersion degree $\deg f$ which cover a variety of codimension less than $D=((2-2g)/\deg f)+d-(n+2)$. As an immediate consequence, one gets, for example, that there are no rational curves on generic hypersurfaces $X$ of degree $d\ge 2n+1$ in $\mathbb P^{n+1}$. 

Two years later, \cite{Ein88} studies the Hilbert scheme of $X\subseteq G$, a generic complete intersection of type $(m_1,\dots,m_k)$ in the Grassmann variety $G=G(r,n+2)$. As a remarkable corollary one gets that any smooth projective subvariety of $X$ is of general type if $m_1+m_2+\cdots +m_k\ge\dim X+n+2$. It is also proved that the Hilbert scheme of $X$ is smooth at points corresponding to smooth rational curves of \lq\lq low\rq\rq{} degree.

The variational method presented here, is due to \cite{Voi96}. By variational method we mean the idea of putting the hypersurfaces in family and to use the positivity property of the tangent bundle of the family itself. The main result of this paper is the following theorem which improves Ein's result in the case of hypersurfaces: let $X\subset\mathbb P^{n+1}$ be a hypersurface of degree $d$. If $d \geq 2n-\ell+1$, $1 \leq \ell \leq n-2$, then any $\ell$-dimensional subvariety $Y$ of $X$ has a desingularization $\widetilde Y$ with an effective canonical bundle. Moreover, if the inequality is strict, then the sections of $K_{\widetilde Y}$ separate generic points of $\widetilde Y$. The bound is now optimum and, in particular, the theorem implies that generic hypersurfaces in $\mathbb P^{n+1}$ of degree $d \geq 2n$, $n \geq 3$, contain no rational curves. The method also gives an improvement of a result of \cite{Xu94} as well as a simplified proof of Ein's original result.

Lastly, let us cite \cite{Pac04}: this paper gives the sharp bound $d\geq 2n$ for a general projective hypersurface $X$ of degree $d$ in $\mathbb P^{n+1}$ containing only subvarieties of general type, for $n\ge 6$. This result improves the aforesaid results of Voisin and Ein. The author proves the bound by showing that, under some numerical conditions, the locus $W$ spanned by subvarieties not of general type (even more than this), is contained in the locus spanned by lines. This is obtained in two steps. First, with the variational technique inherited by Voisin the author proves that $W$ is contained in the locus spanned by lines with highly nonreduced intersection with $X$, the so called bicontact locus. Then the latter is proved to be contained in the locus of lines by using the global generation of certain bundles. Finally, let us mention that similar results have also been obtained independently and at the same time in \cite{C-R04}.

\chapter{Jets spaces}

{\small\textsc{Abstract}. This chapter is devoted to the theory of jet spaces and jet differentials. The idea of using differential equations in hyperbolicity problems can be traced back to work of Bloch \cite{Blo}. The modern language adopted here has been initiated by \cite{G-G80} and later refined by several authors, such as \cite{Dem97} and \cite{S-Y97}. We shall describe the construction of the vector bundle of jet differentials and explain how to build in a functorial way a tower of projective bundles together with the corresponding tautological line bundles on any given manifold which provide a relative compactification of the classical jets spaces. Then, a characterization of jet differentials in terms of direct images of these tautological line bundles will be given.}

\section{Projectivization of directed manifolds}

We introduce a functorial construction in the category of directed manifold in order to produce the so-called projectivized space of $1$-jets over $X$.

So, let $(X,V)$ be a complex directed manifold, $\rank V=r$, and set $\widetilde X = P(V)$. Here, $P(V)$ is the projectivized bundle of lines of $V$ and there is a natural projection $\pi\colon\widetilde X\to X$; moreover, if $\dim X=n$, then $\dim\widetilde X=n+r-1$. On $\widetilde X$, we consider the tautological line bundle $\mathcal O_{\widetilde X}(-1)\subset\pi^* V$ which is defined fiberwise as 
$$
\mathcal O_{\widetilde X}(-1)_{(x,[v])}\overset{\text{\rm def}}=\mathbb C\,v,
$$
for $(x,[v])\in\widetilde X$, with $x\in X$ and $v\in V_x\setminus\{0\}$. We also have the following short exact sequence which comes from the very definition of $\widetilde X$:
$$
0\to T_{\widetilde X/X}\to T_{\widetilde X} \to \pi^*T_X\to 0.
$$
Of course, the surjection here is given by the differential $\pi_*$ and $T_{\widetilde X/X}=\ker\pi_*$ is the relative tangent bundle.

Now, in the above exact sequence, we want to replace $\pi^*T_X$ by $\mathcal O_{\widetilde X}(-1)\subset\pi^* V\subset\pi^*T_X$, in order to build a subbundle of $T_{\widetilde X}$ which takes into account just one significant \lq\lq horizontal\rq\rq{} direction and the \lq\lq vertical\rq\rq{} ones; namely we define $\widetilde V$ to be the inverse image $\pi_*^{-1}\mathcal O_{\widetilde X}(-1)$ so that we have a short exact sequence
$$
0\to T_{\widetilde X/X}\to \widetilde V\to\mathcal O_{\widetilde X}(-1)\to 0
$$
and $\rank\widetilde V=\rank V=r$. There is another short exact sequence attached to this projectivization, which is the relative version of the usual Euler exact sequence of projective spaces:
$$
0\to\mathcal O_{\widetilde X}\to\pi^*V\otimes\mathcal O_{\widetilde X}(1)\to T_{\widetilde X/X}\to 0.
$$
By definition, $(\widetilde X,\widetilde V)$ is a new complex directed manifold, which is compact as soon as $X$ is compact and such that $\pi\colon (\widetilde X,\widetilde V)\to(X,V)$ is a morphism of complex directed manifolds.

\subsection{Lifting of curves}

Let $\Delta_R\subset\mathbb C$ be the open disc $\{|z|<R\}$ of radius $R>0$ and center $0\in\mathbb C$ and $f\colon\Delta_R\to X$ a holomorphic map. Suppose moreover that $f(0)=x$ for some $x\in X$ and that $f$ is a non-constant tangent trajectory of the directed manifold, that is $f'(t)\in V_{f(t)}$ for each $t\in\Delta_R$.

In this case, there is a well-defined and unique tangent line $[f'(t)]\subset V_{f(t)}$ for every $t\in\Delta_R$ even at the stationary points of $f$: if $f'(t_0)=0$ for some $t_0\in\Delta_R$, write $f'(t)=(t-t_0)^mu(t)$ with $m\in\mathbb N\setminus\{0\}$ and $u(t_0)\ne 0$ and define the tangent line at $t_0$ to be $[u(t_0)]$.

We define the lifting $\widetilde f$ of $f$ as the map
$$
\widetilde f\colon\Delta_R\to\widetilde X
$$
which sends $t\mapsto\widetilde f(t)=(f(t),[f'(t)])$. It is clearly holomorphic and the derivative $f'$ gives rise to a section
$$
f'\colon T_{\Delta_R}\to\tilde f^*\mathcal O_{\widetilde X}(-1).
$$
Observe moreover that, as $\pi\circ\widetilde f=f$, one has $\pi_*\widetilde f'(t)=f'(t)$, so that $\widetilde f'(t)$ belongs to $\widetilde V_{(f(t),[f'(t)])}=\widetilde V_{\widetilde f(t)}$. Thus, if $f$ is a tangent trajectory of $(X,V)$ then $\widetilde f$ is a tangent trajectory of $(\widetilde X,\widetilde V)$.

On the other hand, if $g\colon\Delta_R\to\widetilde X$ is a tangent trajectory of $(\widetilde X,\widetilde V)$, then $f\overset{\text{\rm def}}=\pi\circ g$ is a tangent trajectory of $(X,V)$ and $g$ coincides with $\widetilde f$ unless $g$ is contained in a vertical fiber $P(V_x)$: in this case $f$ is constant. 

\subsection{Jets of curves}

Let $X$ be a complex $n$-dimensional manifold. Here, we follow \cite{G-G80} to define the bundle $J_kT_X\to X$ of $k$-jets of germs of parametrized holomorphic curves in $X$.

It is the set of equivalence classes of holomorphic maps $f\colon(\mathbb C,0)\to (X,x)$, with the equivalence relation $f\sim g$ if and only if all derivatives $f^{(j)}(0)=g^{(j)}(0)$ coincide, for $0\le j\le k$, in some (and hence in all) holomorphic coordinates system of $X$ near $x$.
Here, the projection is simply $f\mapsto f(0)$.

These are not vector bundles, unless $k=1$: in this case $J_1T_X$ is simply the holomorphic tangent bundle $T_X$. However, in general, the $J_kT_X$'s are holomorphic fiber bundles, with typical fiber $(\mathbb C^n)^k$ (in fact the elements of the fiber $J_{k}T_{X,x}$ are uniquely determined by the Taylor expansion up to order $k$ of a germ of curve $f$, once a system of coordinate is fixed).

Now, we translate these concepts to the setting of complex directed manifolds. 

\begin{defn}
Let $(X,V)$ be a complex directed manifold. We define the bundle $J_kV\to X$ to be the set of $k$-jets of curves $f\colon (\mathbb C,0)\to X$ which are tangent to $V$, together with the projection map $f\mapsto f(0)$.
\end{defn}

To check that this is in fact a subbundle of $J_k$ we shall describe a special choice of local coordinates: for any point $x_0\in X$, there are local coordinates $(z_1,\dots,z_n)$ on a neighborhood $\Omega$ of $x_0$ such that the fibers $V_x$, for $x\in\Omega$, can be defined by linear equations
$$
V_x=\left\{ v=\sum_{j=1}^nv_j\frac{\partial}{\partial z_j},\,\text{s.t.}\,\, v_j=\sum_{k=1}^ra_{jk}(x)\,v_k,\quad j=r+1,\dots,n\right\},
$$
where $(a_{jk}(x))$ is a holomorphic $(n-r)\times r$ matrix. From this description of the fibers, it follows that to determine a vector $v\in V_x$ it is sufficient to know its first $r$ components $v_1,\dots,v_r$, and the affine chart $v_r\ne 0$ of $P(V_x)$ can be endowed with the coordinates system $(z_1,\dots,z_n,\xi_1,\dots,\xi_{r-1})$, where $\xi_j=v_j/v_r$, $j=1,\dots,r-1$ (and in an analogous way for the other affine charts).

Now, if $f\simeq (f_1,\dots,f_n)$ is a holomorphic tangent trajectory to $(X,V)$ contained in $\Omega$, then by a simple Cauchy problem argument, we see that $f$ is uniquely determined by its initial value $x_0$ and its first $r$ components: as $f'(t)\in V_{f(t)}$, we can recover the remaining components by integrating the differential system
$$
f_j'(t)=\sum_{k=1}^r a_{jk}(f(t))\,f_k'(t),
$$
where $j=r+1,\dots,n$, and initial data $f(0)=x_0$. This shows that the fibers $J_kV_x$ are locally parametrized by
$$
\bigl((f_1',\dots,f_r'),\dots,(f_1^{(k)},\dots,f_r^{(k)})\bigr),
$$
for all $x\in\Omega$, hence $J_k V$ is a locally trivial $(\mathbb C^r)^k$-subbundle of $J_kT_X$.

\section{Projectivized jet bundles}

In this section, we iterate the construction of the projectivization of a complex directed manifold, in order to obtain a projectivized version of the jet bundles. This construction is essentially due to Jean-Pierre Demailly.

We start with a complex directed manifold $(X,V)$, with $\dim X=n$ and $\rank V=r$. We also suppose that $r\ge 2$, otherwise the projectivization of $V$ is trivial. Now, we start the inductive process in the directed manifold category by setting
$$
(X_{0},V_{0})=(X,V),\quad (X_{k},V_{k})=(\tilde X_{k-1},\tilde V_{k-1}).
$$
In other words, $(X_k,V_k)$ is obtained from $(X,V)$ by iterating $k$ times the projectivization construction $(X,V)\mapsto (\tilde X,\tilde V)$ described above.

In this process, the rank of $V_k$ remains constantly equal to $r$ while the dimension of $X_k$ growths linearly with $k$: $\dim X_k=n+k(r-1)$. Let us call $\pi_k\colon X_k\to X_{k-1}$ the natural projection. Then we have, as before, a tautological line bundle $\mathcal O_{X_k}(-1)\subset\pi_k^* V_{k-1}$ over $X_k$ which fits into short exact sequences
\begin{equation}\label{ses1}
0\to T_{X_k/ X_{k-1}}\to V_k \overset{(\pi_k)_*}\to \mathcal O_{X_k}(-1)\to 0
\end{equation}
and
\begin{equation}\label{ses2}
0\to\mathcal O_{X_k}\to\pi_k^* V_{k-1}\otimes\mathcal O_{X_k}(1)\to T_{X_k/X_{k-1}}\to 0.
\end{equation}
Now we come back to the lifting of curves. Our precedent discussion has shown that given a non-constant tangent trajectory $f\colon\Delta_R\to X$ to $(X,V)$ we have a well-defined non-constant tangent trajectory $\tilde f\colon\Delta_R\to\tilde X=X_1$ to $(\tilde X,\tilde V)=(X_1,V_1)$. Now, set inductively 
$$
f_{[0]}=f,\quad f_{[k]}=\tilde f_{[k-1]}.
$$
Then, for each $k$, we get a tangent trajectory $f_{[k]}\colon\Delta_R\to X_k$ to $(X_k,V_k)$ and the derivative $f'_{[k-1]}$ gives rise to a section
$$
f'_{[k-1]}\colon T_{\Delta_R}\to f_{[k]}^*\mathcal O_{X_k}(-1).
$$

\subsection{Regular and singular loci}

By construction, there exists a canonical injection $\mathcal O_{X_k}(-1)\hookrightarrow \pi_k^* V_{k-1}$ and, a composition with the projection $(\pi_{k-1})_*$ gives for all $k\ge 2$ a line bundle morphism
$$
\xymatrix{
\mathcal O_{X_k}(-1) \ar@{^{(}->}[r] \ar@/_1pc/[rr] & \pi_k^* V_{k-1} \ar[r] & \pi_k^*\mathcal O_{X_{k-1}}(-1).}
$$
The zero divisor of this morphism is clearly the projectivization of the relative tangent bundle $T_{X_{k-1}/X_{k-2}}$, which is, of course, (fiber-wise, with respect to $\pi_k\colon X_k\to X_{k-1}$) a hyperplane subbundle of $X_k$. Thus, if we set
$$
D_k=P(T_{X_{k-1}/X_{k-2}})\subset P(V_{k-1})=X_k,\quad k\ge 2,
$$
we find

\begin{equation}\label{divisor}
\mathcal O_{X_k}(-1)\simeq\pi_k^*\mathcal O_{X_{k-1}}(-1)\otimes\mathcal O_{X_k}(-D_k).
\end{equation}

Now, take a regular germ of curve $f\colon (\mathbb C,0)\to (X,x)$ tangent to $V$, that is $f'(0)\ne 0$, and consider, for $j=2,\dots,k$, its $j$-th lifting $f_{[j]}$: we claim that then $f_{[j]}(0)\notin D_j$. In this case, in fact, all  the liftings of $f$ are regular and  $f_{[j]}(0)\in D_j$ if and only if $(\pi_{j-1})_*f_{[j-1]}'(0)=f_{[j-2]}'(0)=0$.

On the other hand, if $f$ is a non-constant germ of curve tangent to $V$ such that, for all $j=2,\dots,k$, $f_{[j]}(0)\notin D_j$ then $f'(0)\ne 0$.

Summarizing, if we define
$$
\pi_{j,k}\overset{\text{\rm def}}=\pi_{j+1}\circ\dots\circ\pi_k\colon X_k\to X_j,
$$
then a point $w\in X_k$ can be reached by a lifting of some regular germ of curve (if and) only if $\pi_{j,k}(w)\notin D_j$, for all $j=2,\dots,k$. It is then natural to define
$$
X_k^{\text{\rm reg}}\,\overset{\text{\rm def}}=\bigcap_{j=2}^k\pi_{j,k}^{-1}(X_j\setminus D_j)
$$
and
$$
X_k^{\text{\rm sing}}\,\overset{\text{\rm def}}=\bigcup_{j=2}^k\pi_{j,k}^{-1}(D_j)=X_k\setminus X_k^{\text{\rm reg}}.
$$
This singular locus comes out also if one studies the base locus of the linear system associated to the anti-tautological line bundle $\mathcal O_{X_k}(1)$. In fact, we have the following proposition:

\begin{prop}[\cite{Dem97}]
For every $m>0$, the base locus of the linear system associated to the restriction of $\mathcal O_{X_k}(m)$ to every fiber $\pi_{0,k}^{-1}(x)$, $x\in X$, is exactly $X_k^{\text{\rm sing}}\cap\pi_{0,k}^{-1}(x)$.
In other words, $X_k^{\text{\rm sing}}$ is the \lq\lq relative\rq\rq{} base locus of $|\mathcal O_{X_k}(m)|$.
Moreover, $\mathcal O_{X_k}(1)$ is relatively big. 
\end{prop}

This proposition also shows that $\mathcal O_{X_k}(1)$ cannot be relatively ample, unless $k=1$. Observe finally that the fibers $\pi_{0,k}^{-1}(x)$ are all isomorphic to a \lq\lq universal\rq\rq{} (and quite mysterious) nonsingular projective variety of dimension $k(r-1)$ which will be denoted by $\mathbb R_{r,k}$: it is not hard to see that $\mathbb R_{r,k}$ is rational. 

\section{Jet differentials}

Let $(X,V)$ be a complex directed manifold.
Let $\mathbb G_k$ be the group of germs of $k$-jets of biholomorphisms of $(\mathbb C,0)$, that is, the group of germs of biholomorphic maps
$$
t\mapsto\varphi(t)=a_1\, t+ a_2\, t^2+\cdots+a_k\,t^k,\quad a_1\in\mathbb C^*,a_j\in\mathbb C,j\ge 2,
$$
in which the composition law is taken modulo terms $t^j$ of degree $j>k$. Then $\mathbb G_k$ admits a natural fiberwise right action on $J_kV$ consisting of reparametrizing $k$-jets of curves by a biholomorphic change of parameter. Moreover the subgroup $\mathbb H\simeq\mathbb C^*$ of homotheties $\varphi(t)=\lambda\,t$ is a (non normal) subgroup of $\mathbb G_k$ and we have a semidirect decomposition $\mathbb G_k=\mathbb G'_k\ltimes\mathbb H$, where $\mathbb G'_k$ is the group of $k$-jets of biholomorphisms tangent to the identity. The corresponding action on $k$-jets is described in coordinates by
$$
\lambda\cdot(f',f'',\dots,f^{(k)})=(\lambda f',\lambda^2 f'',\dots,\lambda^k f^{(k)}).
$$
As in \cite{G-G80}, we introduce the vector bundle $\mathcal J_{k,m}V^*\to X$ whose fibres are complex valued polynomials $Q(f',f'',\dots,f^{(k)})$ on the fibres of $J_kV$, of weighted degree $m$ with respect to the $\mathbb C^*$ action defined by $\mathbb H$, that is, such that
$$
Q(\lambda f',\lambda^2 f'',\dots,\lambda^k f^{(k)})=\lambda^mQ(f',f'',\dots,f^{(k)}),
$$
for all $\lambda\in\mathbb C^*$ and $(f',f'',\dots,f^{(k)})\in J_kV$.

Next, we define the bundle of Demailly-Semple jet differentials (or invariant jet differentials) as a subbundle of the Green-Griffiths bundle.

\begin{defn}[\cite{Dem97}]
The \emph{bundle of invariant jet differentials of order $k$ and weighted degree $m$} is the subbundle $E_{k,m}V^*\subset\mathcal J_{k,m}V^*$ of polynomial differential operators $Q(f',f'',\dots,f^{(k)})$ which are equivariant under arbitrary changes of reparametrization, that is, for every $\varphi\in\mathbb G_k$
$$
Q((f\circ\varphi)',(f\circ\varphi)'',\dots,(f\circ\varphi)^{(k)})=\varphi'(0)^m\,Q(f',f'',\dots,f^{(k)}).
$$
Alternatively, $E_{k,m}V^*=(\mathcal J_{k,m}V^*)^{\mathbb G'_k}$ is the set of invariants of $\mathcal J_{k,m}V^*$ under the action of $\mathbb G'_k$.
\end{defn}

\begin{rem}
From the hyperbolicity point of view, it is of course more natural to consider the invariant jet differentials. In fact, we are only interested in the geometry of the entire curves in a given manifold. For this reason, it is redundant how the entire curves are parametrized: we just want to look at their conformal class.
\end{rem}

\begin{rem}
In the sequel, it will be useful to look at global invariant jet differentials not only as sections of a vector bundle but also as holomorphic maps $Q\colon J_kV\to\mathbb C$ which are invariant with respect to the fiberwise action of $\mathbb G_k$.
\end{rem}

We now define a filtration on $\mathcal J_{k,m}V^*$: a coordinate change $f\mapsto\Psi\circ f$ transforms every monomial $(f^{(\bullet)})^\ell=(f')^{\ell_1}(f'')^{\ell_2}\cdots(f^{(k)})^{\ell_k}$ of partial weighted degree $|\ell|_s:=\ell_1+2\ell_2+\cdots+s\ell_s$, $1\le s\le k$, into a polynomial $((\Psi\circ f)^{(\bullet)})^\ell$ in $(f',f'',\dots,f^{(k)})$, which has the same partial weighted degree of order $s$ if $\ell_{s+1}=\cdots=\ell_k=0$ and a larger or equal partial degree of order $s$ otherwise. Hence, for each $s=1,\dots,k$, we get a well defined decreasing filtration $F_s^\bullet$ on $\mathcal J_{k,m}V^*$ as follows:
$$
F_s^p(\mathcal J_{k,m}V^*)=
\left\{
\begin{matrix}
\text{$Q(f',f'',\dots,f^{(k)})\in\mathcal J_{k,m}V^*$ involving} \\
\text{only monomials $(f^{(\bullet)})^\ell$ with $|\ell|_s\ge p$}
\end{matrix}
\right\},
\quad\forall p\in\mathbb N.
$$
The graded terms $\operatorname{Gr}^p_{k-1}(\mathcal J_{k,m}V^*)$, associated with the $(k-1)$-filtration $F_{k-1}^p(\mathcal J_{k,m}V^*)$, are the homogeneous polynomials, say $Q(f',\dots,f^{(k)})$, whose all monomials $(f^{(\bullet)})^\ell$ have partial weighted degree $|\ell|_{k-1}=p$; hence, their degree $\ell_k$ in $f^{(k)}$ is such that $m-p=k\ell_k$ and $\operatorname{Gr}^p_{k-1}(\mathcal J_{k,m}V^*)=0$ unless $k|m-p$. Looking at the transition automorphisms of the graded bundle induced by the coordinate change $f\mapsto\Psi\circ f$, it turns out that $f^{(k)}$ behaves as an element of $V\subset T_X$ and, as a simple computation shows, we find
$$
\operatorname{Gr}^{m-k\ell_k}_{k-1}(\mathcal J_{k,m}V^*)=\mathcal J_{k-1,m-k\ell_k}V^*\otimes S^{\ell_k}V^*.
$$
Combining all filtrations $F^\bullet_s$ together, we find inductively a filtration $F^\bullet$ on $\mathcal J_{k,m}V^*$ such that the graded terms are
$$
\operatorname{Gr}^{\ell}(\mathcal J_{k,m}V^*)=S^{\ell_1}V^*\otimes S^{\ell_2}V^*\otimes\cdots\otimes S^{\ell_k}V^*,\quad\ell\in\mathbb N^k,|\ell|_k=m.
$$
Moreover there are natural induced filtrations 
$$
F^p_s(E_{k,m}V^*)=E_{k,m}V^*\cap F_{s}^p(\mathcal J_{k,m}V^*)
$$ 
in such a way that
$$
\operatorname{Gr}^\bullet(E_{k,m}V^*)=\left(\bigoplus_{|\ell|_k=m}S^{\ell_1}V^*\otimes S^{\ell_2}V^*\otimes\cdots\otimes S^{\ell_k}V^*\right)^{\mathbb G'_k}.
$$
Let us see more concretely which are the elements of the bundles we have introduced above in the following examples. For the sake of simplicity we shall consider here only the \lq\lq absolute\rq\rq{} case, that is $V=T_X$.
\begin{exmp}
Let us first look at the Green-Griffiths jet differentials. So, we fix a point $x\in X$ and look at the elements of the fiber $\mathcal J_{k,m}T^*_{X,x}$. 
\begin{itemize}
\item For $k=1$, we simply have $\mathcal J_{1,m}T^*_X=S^mT^*_X$. This is the usual bundle of symmetric differentials.
\item For another example, when $k=3$, we have that a typical element of the fiber is
$$
\begin{aligned}
& \sum a_i f_i',\quad\text{for $m=1,$} \\
& \sum a_{ij}f_i'f_j'+b_if_i'',\quad\text{for $m=2,$} \\
& \sum a_{ijk}f_i'f_j'f_k'+b_{ij}f_i'f_j''+c_if_i''',\quad\text{for $m=3,$}
\end{aligned}
$$
where the coefficients are holomorphic functions.
\end{itemize}
In conclusion, sections of $\mathcal J_{k,m}T^*_X$ are locally given by homogeneous polynomials with holomorphic coefficients in the variables $f',\dots,f^{(k)}$, of total weight $m$, where $f_i^{(l)}$ is assigned weight $l$. 
\end{exmp}

\begin{exmp}
For the invariant jet differentials, we still have $E_{1,m}T^*_X=S^mT^*_X$, but for $k\ge 2$ things become much more complicated. To explain the difficulty, let us translate the condition of being invariant under reparametrization in terms of classical invariant theory.

As above, let $\mathbb G_k'$ be the group of $k$-jets of biholomorphisms tangent to the identity 
$$\varphi(t)= t +a_2t^2+\dots+ a_kt^k.$$
Then, $\mathbb G_k'$ acts on $(f',f'',\dots,f^{(k)})$ linearly:
$$(f\circ \varphi)'= f',$$
$$(f\circ \varphi)''= f''+2a_2f',$$
$$(f\circ \varphi)'''= f'''+6a_2f''+6a_3f',\quad\dots$$
We see that $\mathbb G_k'$ acts by explicit matrix multiplication by the group of matrices

$$
\left(
\begin{array}{ccccc}
1 & 0 & 0 & 0 & 0 \\
2a_2 & 1 & 0 & 0 & 0 \\
6a_3 & 6a_2 & 1 & 0 & 0 \\
... & ... & ... & ... & 0 \\
k! a_k& ... & ... & ... & 1
\end{array}
\right).
$$
For instance, for $X$ a complex surface, local sections of $E_{2,m}T^*_X$ are given by polynomials invariant under the action of the unipotent group $\mathbb U(2)$
$$
\sum_{\alpha_1+\alpha_2+3\beta=m}a_{\alpha_1\alpha_2\beta}(f_1')^{\alpha_1}(f_2')^{\alpha_2}(f'_1f''_2-f''_1f'_2)^\beta.
$$

The algebraic characterization of $E_{2,m}T^*_X$ enables us to make explicit the filtration described above, namely
$$
\operatorname{Gr}^\bullet(E_{2,m}T_X^*)=\bigoplus_{0 \leq j \leq m/3} S^{m-3j}T_X^*\otimes K_X^j.
$$
For $k\geq 3$ this group of matrices is a proper subroup of the unipotent group, hence it is non-reductive and therefore we cannot apply the well-known invariant theory of reductive actions. It turns out that for proper subgroup of the unipotent group the theory is much less developed.

Thus, in general, it is still an unsolved (and probably very difficult) problem to determine the structure of the algebra
$$
\mathcal A_k=\bigoplus_{m\ge 0}E_{k,m}T^*_{X,x}.
$$

Among the few results we have, we know \cite{Rou06a} for instance, for $X$ a threefold,  that  $\mathcal A_3$ is generated by the following polynomials:
$$W= \left| 
\begin{array}{ccc}
f_{1}^{\prime } & f_{2}^{\prime } & f_{3}^{\prime } \\ 
f_{1}^{\prime \prime } & f_{2}^{\prime \prime } & f_{3}^{\prime \prime } \\ 
f_{1}^{\prime \prime \prime } & f_{2}^{\prime \prime \prime } & 
f_{3}^{\prime \prime \prime }
\end{array}
\right|,$$

$$w_{ij}=f_{i}^{\prime }f_{j}^{\prime \prime }-f_{i}^{\prime \prime
}f_{j}^{\prime },$$

$$w_{ij}^{k}=f_{k}^{\prime }(f_{i}^{\prime }f_{j}^{\prime \prime \prime
}-f_{i}^{\prime \prime \prime }f_{j}^{\prime })-3f_{k}^{\prime \prime
}(f_{i}^{\prime }f_{j}^{\prime \prime }-f_{i}^{\prime \prime }f_{j}^{\prime
}).$$

As above, one can then deduce the filtration
$$
\operatorname{Gr}^\bullet(E_{3,m}T_X^*)=\bigoplus_{a+3b+5c+6d=m}\Gamma^{
(a+b+2c+d,\,\,b+c+d,\,\,d)}T_X^{\ast },
$$
where $\Gamma$ denotes the Schur functor which provides the list of all irreducible representation of the general linear group.

In this direction, we want to cite here the other results we have: the structure of $\mathcal A_2$, $\mathcal A_3$ and $\mathcal A_4$ for $\dim X=2$ ($9$ generators) was found by Demailly and, recently, Merker \cite{Mer08} found $\mathcal A_5$ when $\dim X=2$ ($56$ generators) and $\mathcal A_4$ for $\dim X=4$ ($2835$ generators).

The general structure of $\mathcal A_k$ appears to be far from being understood even in the surface case.

We would like to mention here that recent progresses have been made in the invariant theory of non-reductive groups with applications to Demailly-Semple jets by Berczi and Kirwan. 
In particular, they can prove that $\mathcal A_k$ is finitely generated \cite{B-K10}.

\end{exmp}

\subsection{Invariant jet differentials and projectivized jet bundles}

Associated to the graded algebra bundle $\mathcal J_{k,\bullet}V^*=\bigoplus_{m\ge 0}\mathcal J_{k,m}V^*$, there is an analytic fiber bundle, namely $\operatorname{Proj}(\mathcal J_{k,\bullet}V^*)=J_kV^{\text{\rm nc}}/\mathbb C^*$, where $J_kV^{\text{\rm nc}}$ is the bundle of non-constant $k$-jets tangent to $V$, whose fibers are weighted projective spaces $\mathbb P(r,\dots,r;1,2,\dots,k)$ (for a definition of weighted projective spaces and much more, see \cite{Dol82}).  

However, we would be mostly interested in a more \lq\lq geometric\rq\rq{} quotient, for instance something like $J_kV^{\text{\rm nc}}/\mathbb G_k$.

In \cite{Dem97}, it has been constructed something similar, that is the quotient space of $J_kV^{\text{\rm reg}}/\mathbb G_k$ of regular (i.e. with non-vanishing first derivative) $k$-jets tangent to $V$ and we shall see how the projectivized jet bundles can be seen as a relative compactification of this quotient space.

This is exactly the content of the next theorem.

\begin{thm}[\cite{Dem97}]\label{di}
Suppose $\rank V\ge 2$ and let $\pi_{0,k}\colon X_k\to X$ be the projectivized $k$-th jet bundle of $(X,V)$. Then
\begin{itemize}
\item the quotient $J_kV^{\text{\rm reg}}/\mathbb G_k$ has the structure of a locally trivial bundle over $X$ and there is a holomorphic embedding $J_kV^{\text{\rm reg}}/\mathbb G_k\hookrightarrow X_k$ over $X$, which identifies $J_kV^{\text{\rm reg}}/\mathbb G_k$ with $X_k^{\text{\rm reg}}$.
\item The direct image sheaf
$$
(\pi_{0,k})_*\mathcal O_{X_k}(m)\simeq\mathcal O(E_{k,m}V^*)
$$
can be identified with the sheaf of holomorphic sections of the bundle $E_{k,m}V^*$.
\end{itemize}
\end{thm}

Let us say a few words about this result. First of all, one needs to use $J_kV^{\text{\rm reg}}$ instead of $J_kV^{\text{\rm nc}}$ in order to lift a $k$-jet of curve $f$ by taking the derivative $(f,[f'])$ without any cancellation of zeroes in $f'$: in this way one gets a uniquely defined $(k-1)$-jet $\tilde f$ so that, inductively, $f_{[k]}(0)$ is independent of the choice of the representative $f$.

Moreover, as the reparametrization commutes with the lifting process, that is $\widetilde{(f\circ\varphi)} =\widetilde{f} \circ\varphi$, and more generally $(f\circ\varphi)_{[k]}=f_{[k]}\circ\varphi$, we get a well defined map
$$
J_kV^{\text{\rm reg}}/\mathbb G_k\to X_k^{\text{\rm reg}}.
$$

This map can be described explicitely in local coordinates. Take coordinates $(z_1,\dots,z_n)$ near $x_0 \in X$ such that we have that $V_{x_0}= \operatorname{Vect}\left(\frac{\partial}{\partial z_1},\dots,\frac{\partial}{\partial z_r}\right)$. Let $f=(f_1,\dots,f_n)$ be a regular $k$-jet tangent to $V$ such that $f_r(t)=t$. $X_k$ is a $k$-stage tower of $\mathbb P^{r-1}$-bundles. In the corresponding inhomogeneous coordinates, the point $f_{[k]}(0)$ is given by
$$
\left(f'_1(0),\dots,f'_{r-1}(0),f''_1(0),\dots,f''_{r-1}(0),\dots,f^{(k)}_1(0),\dots,f^{(k)}_{r-1}(0)\right).
$$
We see easily that the map $J_kV^{\text{\rm reg}}/\mathbb G_k\to X_k^{\text{\rm reg}}$ is a bijection onto $X_k^{\text{\rm reg}}$. This is the embedding of the first part of the theorem.

Next, part two of the theorem says that, for $x\in X$, we have an identification $H^0(\pi_{0,k}^{-1}(x),\mathcal O_{X_k}(m))\simeq E_{k,m}V^*_{x}$: we want to describe briefly what this identification is. To begin with, fix a section $\sigma\in H^0(\pi_{0,k}^{-1}(x),\mathcal O_{X_k}(m))$. Recall that given regular $k$-jet of curve at $x\in X$, the derivative $f_{[k-1]}'(0)$ defines an element of the fiber of $\mathcal O_{X_k}(-1)$ at $f_{[k]}(0)$. Then we get a well-defined complex valued operator
$$
Q(f',f'',\dots,f^{(k)})=\sigma(f_{[k]}(0))\cdot(f_{[k-1]}'(0))^m.
$$
Such a $Q$ is holomorphic and extends to singular jets by an easy Riemann's extension theorem argument (since $\operatorname{codim}_{J_kV_x} J_kV^{\text{\rm sing}}_x=r\ge 2$). The $\mathbb G_k$-invariance is satisfied since $f_{[k]}(0)$ does not depend on the reparametrization and $(f\circ \varphi)'_{[k-1]}(0)=f'_{[k-1]}(0)\cdot\varphi'(0).$ Moreover, the invariance implies in particular that $Q$ must be polynomial. Thus $Q \in E_{k,m}V^*_x$.
This correspondence is easily shown to be bijective and is in fact the one given in the theorem.

\subsection{Sufficient conditions for relative positivity}

The relative structure of the fibration $\pi_{0,k}\colon X_k\to X$ is completely universal and its fibers are smooth rational varieties which depend only on $k$ and on the rank of $V$.

Moreover, as $X_k$ arises as a sequence of successive compactifications of vector bundles, its Picard group has a quite simple structure, namely we have
$$
\operatorname{Pic}(X_k)\simeq\operatorname{Pic}(X)\oplus\mathbb Z u_1\cdots\oplus\mathbb Z u_k,
$$   
where $u_j$, $j=1,\dots,k$, is the class of $\mathcal O_{X_j}(1)$.

As we already observed, the line bundle $\mathcal O_{X_k}(1)$ is never relatively ample over $X$ for $k\ge 2$. Now, for each $\mathbf a=(a_1,\dots,a_k)\in\mathbb Z^k$, we define a line bundle $\mathcal O_{X_k}(\mathbf a)$ as
$$
\mathcal O_{X_k}(\mathbf a)\overset{\text{\rm def}}=\pi_{1,k}^*\mathcal O_{X_1}(a_1)\otimes\pi_{2,k}^*\mathcal O_{X_2}(a_2)\otimes\cdots\otimes\mathcal O_{X_k}(a_k).
$$
By formula (\ref{divisor}), we get inductively 
$$
\pi_{j,k}^*\mathcal O_{X_j}(1)=\mathcal O_{X_k}(1)\otimes\mathcal O_{X_k}(-\pi_{j+1,k}^*D_{j+1}-\cdots -D_k).
$$
Set, for $j=1,\dots,k-1$, $D_j^\star=\pi_{j+1,k}^*D_{j+1}$ and $D^\star_k=0$. Then, if we define the weight $\mathbf b=(b_1,\dots,b_k)\in\mathbb Z^k$ by $b_j=a_1+\cdots+a_j$, $j=1,\dots,k$, we find an identity
$$
\mathcal O_{X_k}(\mathbf a)\simeq\mathcal O_{X_k}(b_k)\otimes\mathcal O_{X_k}(-\mathbf b\cdot D^\star),
$$
where
$$
\mathbf b\cdot D^\star\overset{\text{\rm def}}=\sum_{j=1}^{k-1}b_j\,\pi_{j+1,k}^*D_{j+1}.
$$
In particular, as all the $D_j$'s are effective, if $\mathbf b\in\mathbb N^k$, that is $a_1+\cdots+a_j\ge 0$ for all $j=1,\dots,k$, we get a non-trivial bundle morphism
\begin{equation}\label{twist}
\mathcal O_{X_k}(\mathbf a)\simeq\mathcal O_{X_k}(b_k)\otimes\mathcal O_{X_k}(-\mathbf b\cdot D^\star)\to\mathcal O_{X_k}(b_k).
\end{equation}
Set theoretically, we have seen that the relative base locus of the complete linear system $|\mathcal O_{X_k}(m)|$ is exactly $X_k^{\text{\rm sing}}=\bigcup_{j=2}^k\pi_{j,k}^{-1}(D_j)$. 

Now, we would like to twist the line bundle $\mathcal O_{X_k}(m)$ by an ideal sheaf $\mathcal I$, possibly co-supported on $X_k^{\text{\rm sing}}$, in order to get rid of this base locus. If one wants to remain in the category of invertible sheaves, then this ideal sheaf should be something of the form $\mathcal O_{X_k}(-\mathbf b\cdot D^\star)$, for $\mathbf b\in\mathbb N^k$. 

Next proposition gives sufficient conditions to solve this problem.

\begin{prop}[\cite{Dem97}]\label{relnef}
Let $\mathbf a=(a_1,\dots,a_k)\in\mathbb N^k$ be a weight and $m=b_k=a_1+\cdots+a_k$. Then
\begin{itemize}
\item we have the direct image formula
$$
(\pi_{0,k})_*\mathcal O_{X_k}(\mathbf a)\simeq\mathcal O(\overline F^{\mathbf a}E_{k,m}V^*)\subset\mathcal O(E_{k,m}V^*)
$$
where $\overline F^{\mathbf a}E_{k,m}V^*$ is the subbundle of homogeneous polynomials $Q(f',\dots,f^{(k)})\in E_{k,m}V^*$ involving only monomials $(f^{(\bullet)})^\ell$ such that
$$
\ell_{s+1}+2\ell_{s+2}+\cdots+(k-s)\ell_{k}\le a_{s+1}+\cdots+a_k
$$
for all $s=0,\dots,k-1$.
\item if $a_1\ge 3 a_2,\dots,a_{k-2}\ge 3a_{k-1}$ and $a_{k-1}\ge 2a_k\ge 0$, then the line bundle $\mathcal O_{X_k}(\mathbf a)$ is relatively nef over $X$.
\item if $a_1\ge 3 a_2,\dots,a_{k-2}\ge 3a_{k-1}$ and $a_{k-1}> 2a_k> 0$, then the line bundle $\mathcal O_{X_k}(\mathbf a)$ is relatively ample over $X$.
\end{itemize}
\end{prop}

Note that formula \ref{twist} gives a sheaf injection
$$(\pi_{0,k})_*\mathcal{O}_{X_k}(\mathbf a) \hookrightarrow (\pi_{0,k})_*\mathcal{O}_{X_k}(m)=\mathcal{O}(E_{k,m}V^*),$$
which is the inclusion of the first part of the proposition.
The two last positivity properties are obtained by induction on $k$.

\chapter{Hyperbolicity and negativity of the curvature}

{\small\textsc{Abstract}. In this chapter we shall explain how negativity properties of the curvature of complex manifolds is connected to hyperbolicity. We start with some basic notions of curvature and then prove the classical Ahlfors-Schwarz lemma. Then, we come back to higher order jets, and prove the basic result that every entire curve automatically satisfies every global jet differential with values in an antiample line bundle; as a consequence we deduce Bloch's theorem about entire curves on complex tori. To conclude the chapter we illustrate a general strategy to prove algebraic degeneracy of entire curves.}

\section{Curvature and positivity}

Let $X$ be a complex manifold of complex dimension $n$ and $\pi\colon E\to X$ a hermitian vector bundle of rank $r$ with hermitian metric $h$. Fix a point $x_0\in X$, some local holomorphic coordinates $(z_1,\dots,z_n)$ centered in $x_0$ and a local holomorphic frame $(e_1,\dots,e_r)$ of $E$, which we can suppose without loss of generality orthonormal in $x_0$ with respect to $h$.

We recall that, on $E$, there exists a unique linear connection $D_h$ which respects both the complex structures of $E$ and $X$ and the hermitian structure of $E$ given by $h$: it is called the \emph{Chern connection}. Its curvature
$$
i\,\Theta(E,h)=i\,D^2_h
$$
is the \emph{Chern curvature} of the pair $(E,h)$: it is a $(1,1)$-form with values in the hermitian endomorphisms of $E$. Locally, in terms of the natural hermitian matrix $H$ associated to $h$ with respect to the local frame $(e_\lambda)$, it is given by
$$
i\,\Theta(E,h)=i\,\overline\partial(\overline H^{-1}\partial\overline H).
$$
At the given point $x_0\in X$, write
$$
i\,\Theta_{x_0}(E,h)=i\sum_{j,k=1}^n\sum_{\lambda,\mu=1}^r c_{jk\lambda\mu}\,dz_j\wedge d\overline z_k\otimes e_\lambda^*\otimes e_\mu.
$$
To $i\,\Theta_{x_0}(E,h)$ corresponds a natural hermitian form $\theta_E$ on $T_{X,x_0}\otimes E_{x_0}$ defined by
$$
\theta_E=\sum_{j,k,\lambda,\mu}c_{jk\lambda\mu}\,(dz_j\otimes e_\lambda^*)\otimes(\overline{dz_k\otimes e_\mu^*}).
$$
Its evaluation on rank one tensors gives rise to the notion of Griffiths positivity for vector bundles.

\begin{defn}
The hermitian vector bundle $(E,h)$ is said to be \emph{Griffiths positive} (resp. \emph{Griffiths semi-negative}) at $x_0$ if for all $v\in T_{X,x_0}\setminus\{0\}$ and $s\in E_{x_0}\setminus\{0\}$ we have
$$
\theta_E(v\otimes s,v\otimes s)>0\quad\text{(resp. $\le 0$)}.
$$
The bundle $(E,h)$ is said to be \emph{Griffiths positive} (resp. \emph{Griffiths semi-negative}) if it is Griffiths positive (resp. semi-negative) at all point $x\in X$.
\end{defn}
We shall not list here all the remarkable properties of Griffiths positive bundles, but just mention that Griffiths positivity implies (and is conjecturally equivalent to) the ampleness for the bundle $E$, that is global section of high symmetric powers of $E$ generates $1$-jets of sections at any point.

Obviously these notions are still valid in the particular case when $E=T_X$ is the tangent bundle of $X$ and the metric $h$ is then a hermitian metric on $X$. In this case the notion of Griffiths curvature coincides with the classical notion of holomorphic bisectional curvature and if the hermitian form $\theta_{T_X}$ is just tested on tensors of the form $v\otimes v$, with $v\in T_X$, this gives back the holomorphic sectional curvature. In particular, if the tangent bundle of a manifold is Griffiths positive (resp. Griffiths semi-negtive) then it has positive (resp. semi-negative) holomorphic sectional curvature.

When hermitian metrics on the tangent bundle are given, we have and we shall often confuse the hermitian form $h$ and its naturally associated $(1,1)$-form $\omega=-\Im h$ (and vice-versa): if $h$ is given locally by
$$
h=\sum_{j,k=1}^n h_{jk}\,dz_j\otimes d\overline z_k,
$$
then $\omega$ is given simply by
$$
\omega=\frac i2\sum_{j,k=1}^n h_{jk}\,dz_j\wedge d\overline z_k,
$$
as one can immediately check.

\subsection{Special case of hermitian line bundles}

When $r=1$, that is when $E$ is a hermitian line bundle, the positive definite hermitian matrix $H$ is just a positive function which we write $H=e^{-\varphi}$; we call $\varphi$ a \emph{local weight} for $h$. The above formulae then give locally
$$
i\,\Theta(E,h)=i\,\partial\overline\partial\varphi.
$$
Especially, we see that in this case $i\,\Theta(E,h)$ is a closed real $(1,1)$-form. It is (semi-)positive (in the sense of Griffiths) if and only if the local weight $\varphi$ is strictly plurisubharmonic (resp. plurisubharmonic). By the celebrated Kodaira theorem, this happens if and only if $E$ is an ample line bundle. With a slight abuse of notation, in the case of line bundle we could indicate in the sequel with
$$
i\,\Theta(E,h)(v),\quad v\in T_X,
$$
the evaluation on $v$ of the hermitian form naturally associated to the $(1,1)$-form $i\,\Theta(E,h)$.

More generally, we can relax the smoothness requirements on $h$ and just ask the local weights to be locally integrable: this gives the so-called notion of \emph{singular hermitian metric}. In this framework, the curvature is still well defined provided we take derivatives in the sense of distributions. The positivity in the sense of distribution becomes now to ask the local weights $\varphi$ to satisfy the following property: for all $w=(w_1,\dots,w_n)\in\mathbb C^n\setminus\{0\}$ the distribution
$$
\sum_{j,k=1}^n\frac{\partial^2\varphi}{\partial z_j\partial \overline z_k}\,w_j\overline w_k
$$
is a positive measure. This is again equivalent to ask the local weights $\varphi$ to be plurisubharmonic. We say that a singular hermitian $h$ admits a closed subset $\Sigma_h\subset X$ as its degeneration set if $\varphi$ is locally bounded on $X\setminus\Sigma_h$ and is unbounded on a neighborhood of any point of $\Sigma_h$.

\begin{exmp}
Let $L\to X$ be a holomorphic line bundle and suppose we have non zero global sections $\sigma_1,\dots,\sigma_N\in H^0(X,L)$. Then, there is a natural way to construct a (posssibily) singular hermitian metric $h^{-1}$ on $L^{-1}$ (and therefore on $L$ by taking its dual): we let
$$
h^{-1}(\xi)=\biggl(\sum_{j=1}^N\bigl(\xi(\sigma_{j}(x))\bigr)^2\biggr)^{1/2},\quad\forall\xi\in L_x^{-1}.
$$
The local weights of the induced metric on $L$ are then given by 
$$
\varphi(x)=\frac 12\log\sum_{j=1}^N|s_j(x)|^2
$$
where $s_j$ is the holomorphic function which is the local expression in a given trivialization of the global section $\sigma_j$, $j=1,\dots,N$.  

Since the $s_j$'s are holomorphic, the weights of such a metric are always plurisubharmonic and hence the corresponding Chern curvature is semi-positive (in the sense of distribution). In particular every effective line bundle carries a singular hermitian metric with semi-positive curvature current constructed in this fashion and its degeneracy locus is exactly the base locus of the sections.

If the sections $\sigma_1,\dots,\sigma_N$ globally generate the space of sections of $L$ (that is, for each $x\in X$ there exists a $j_0=1,\dots,N$, such that $\sigma_{j_0}(x)\ne 0$) then the metric $h$ is smooth and semi-positively curved, in particular $L$ is nef.
If they generate $1$-jets of sections at any given point, then the metric $h$ has strictly plurisubharmonic weights and it is positively curved. Both of these properties can be easily shown by direct computation on the local weights.
\end{exmp}

\subsubsection{Riemann surfaces}

In the special case $\dim X=1$, the tangent bundle of $X$ is in fact a line bundle.

The curvature of $T_X$ becomes then a real number and coincide (modulo a positive factor) with the classical notion of Gaussian curvature for Riemann surfaces. For if the metric $h$ is given locally by the single positive smooth function $e^{-\varphi}$, the Chern curvature is given by 
$$
i\,\partial\overline\partial\varphi=i\,\frac{\partial^2\varphi}{\partial z\partial \overline z}\,dz\wedge d\overline z
$$
and 
$$
\frac{\partial^2\varphi}{\partial z\partial \overline z}=\frac 14\Delta\varphi=\frac 14 \Delta(-\log h),
$$
while the Gaussian curvature of $h$ is given by
$$
\kappa(X,h)=-\frac{1}{2h}\Delta\log h.
$$
Thus, we see that hyperbolic Riemann surfaces are exactly the ones which admits a negatively curved hermitian metric (the ones of genus greater than or equal to two which are covered by the disc and inherit its Poincar metric).  

\subsection*{Tautological line bundle on projectivized vector bundles}

We want to derive here the relation between the curvature of a hermitian vector bundle and the curvature of the associated tautological line bundle over its projectivization.

So let $E\to X$ be a hermitian vector bundle of rank $r$ with hermitian metric $h$, where $X$ is any $n$-dimensional complex manifold and consider the projectivization $\pi\colon P(E)\to X$ of lines of $E$. 

Then, $\mathcal O_{P(E)}(-1)\subset\pi^*E$ admits a natural hermitian metric which is just the restriction of the pull-back of $h$ by $\pi$. Fix an arbitrary point $x_0\in X$, a unit vector $v_0\in E_{x_0}$, local holomorphic coordinates $(z_j)$ centered at $x_0$ and choose a local holomorphic frame $(e_\lambda)$ for $E$ near $x_0$ such that
$$
h(e_\lambda(z),e_\mu(z))=\delta_{\lambda\mu}-c_{jk\lambda\mu}\,z_j\overline z_k+O(|z^3|),
$$
and $e_r(x_0)=v_0$, where the $c_{jk\lambda\mu}$'s are the coefficients of the Chern curvature
$$
i\,\Theta_{x_0}(E,h)=i\,\sum_{j,k}\sum_{\lambda,\mu}c_{jk\lambda\mu}\,dz_j\wedge d\overline z_k\otimes e_\lambda^*\otimes e_\mu.
$$
It is a standard fact in hermitian differential geometry that such a choice of a holomorphic local frame is always possible.

We shall compute the curvature $i\,\Theta(\mathcal O_{P(E)}(-1))$ at the (arbitrary) point $(x_0,[v_0])\in P(E)$. Local holomorphic coordinates centered at $(x_0,[v_0])$ are given by $(z_1,\dots,z_n,\xi_1,\dots,\xi_{r-1})$, where the $(r-1)$-tuple $(\xi_1,\dots,\xi_{r-1})$ corresponds to the direction $[\xi_1\,e_1(z)+\cdots+\xi_{r-1}\,e_{r-1}(z)+e_r(z)]$ in the fiber over the point of coordinates $(z_1,\dots,z_n)$. Next, a local holomorphic non vanishing section of $\mathcal O_{P(E)}(-1)$ around $(x_0,[v_0])\simeq (0,0)$ is given by
$$
\eta(z,\xi)=\xi_1\,e_1(z)+\cdots+\xi_{r-1}\,e_{r-1}(z)+e_r(z)
$$
and its squared length by
$$
||\eta||_h^2=1+|\xi|^2-\sum_{j,k=1}^n c_{jkrr}\,z_j\overline z_k + O((|z|+|\xi|)^3).
$$
Thus, we obtain
\begin{equation}\label{curvature}
\begin{aligned}
i\,\Theta_{(x_0,[v_0])}(\mathcal O_{P(E)}(-1)) & =\left.- i\,\partial\overline\partial\log||\eta||^2_h\right|_{(z,\xi)=(0,0)} \\
& = i\left(\sum_{j,k=1}^n c_{jkrr}\,dz_j\wedge d\overline z_k-\sum_{\lambda=1}^{r-1}d\xi_\lambda\wedge d\overline\xi_\lambda\right) \\
&=\theta_E(\bullet\otimes v_0,\bullet\otimes v_0)-||\bullet||_{FS}^2,
\end{aligned}
\end{equation}
where $||\bullet||_{FS}^2$ is the Fubini-Study metric induced by $h$ on the projective space $P(E_{x_0})$.

In particular, we see that if $(E,h)$ is Griffiths negative then the line bundle $\mathcal O_{P(E)}(-1)$ is negative (and $\mathcal O_{P(E)}(1)$ is positive).
\section{The Ahlfors-Schwarz lemma}

A basic idea is that Kobayashi hyperbolicity is somehow related with suitable properties of negativity of the curvature even in dimension greater than one. The first result in this direction is the following, which was already observed in \cite{Kob70}.

\begin{prop}\label{bisectional}
Let $X$ be a compact hermitian manifold. Assume that $T_X$ has negative Griffiths curvature, or more generally that $T^*_X$ is ample. Then $X$ is hyperbolic.
\end{prop}

\begin{proof}
Since $T^*_X$ is ample, the large symmetric powers $S^mT^*_X$ of $T^*_X$ have enough global section $\sigma_1,\dots,\sigma_N$ in order to generate $1$-jets of sections at any point. This means that the function $\eta\colon T_X\to\mathbb R$ defined by
$$
\eta(v)=\biggl(\sum_{j=1}^N|\sigma_j(x)\cdot v^{\otimes m}|^2\biggr)^{1/2m},\quad v\in T_{X,x},
$$
is strictly plurisubharmonic on $T_X$ minus the zero section.

Now assume that $X$ is not hyperbolic: then, by Brody's lemma, there exists a non constant entire curve $g\colon\mathbb C\to X$ with bounded first derivative, with respect to any hermitian metric on $X$. Then, the composition $\eta\circ g'$ is a bounded subharmonic function on $\mathbb C$ which is strictly plurisubharmonic on $\{g'\ne 0\}$. But then $g$ must be constant by the maximum principle and we get a contradiction.
\end{proof}

We now state and give a sketch of the proof of the Ahlfors-Schwarz lemma, in order to generalize the preceding proposition to higher order jets.

\begin{lem}[Ahlfors-Schwarz]
Let $\gamma(\zeta)=i\,\gamma_0(\zeta)\,d\zeta\wedge d\overline\zeta$ be a hermtian metric on $\Delta_R$, where $\log\gamma_0$ is a subharmonic function such that 
$$
i\,\partial\overline\partial\log\gamma_0(\zeta)\ge A\gamma(\zeta)
$$ 
in the sense of distribution, for some positive constant $A$. Then $\gamma$ can be compared with the Poincar metric of $\Delta_R$ as follows:
$$
\gamma(\zeta)\le \frac 2A\frac{R^{-2}|d\zeta|^2}{(1-|\zeta|^2/R^2)^2}.
$$
\end{lem}

\begin{proof}
Assume first that $\gamma_0$ is smooth and defined on $\overline\Delta_R$. Take a point $\zeta_0\in\Delta_R$ at which $(1-|\zeta|^2/R^2)^2\gamma_0$ is maximum. Then its logarithmic $i\,\partial\overline\partial$-derivative at $\zeta_0$ must be non positive, hence
$$
i\,\partial\overline\partial\log\gamma_0(\zeta)|_{\zeta=\zeta_0}-2i\,\partial\overline\partial\log(1-|\zeta|^2/R^2)^{-1}|_{\zeta=\zeta_0}\le 0.
$$
But now the hypothesis implies that
$$
A\gamma_0(\zeta_0)\le i\,\partial\overline\partial\log\gamma_0(\zeta)|_{\zeta=\zeta_0}\le 2i\,\partial\overline\partial\log(1-|\zeta|^2/R^2)^{-1}|_{\zeta=\zeta_0}
$$
and 
$$
2i\,\partial\overline\partial\log(1-|\zeta|^2/R^2)^{-1}|_{\zeta=\zeta_0}=2R^{-2}(1-|\zeta_0|^2/R^2)^{-2}.
$$
Hence, we conclude that
$$
(1-|\zeta|^2/R^2)^2\gamma_0(\zeta)\le (1-|\zeta_0|^2/R^2)^2\gamma_0(\zeta_0)\le\frac 2{AR^2}.
$$
If $\gamma_0$ is just defined on $\Delta_R$, we consider instead $\gamma_0^\varepsilon(\zeta)=\gamma_0((1-\varepsilon)\zeta)$ and then we let $\varepsilon$ tend to $0$.

When $\gamma_0$ is not smooth one uses a standard regularization argument by taking the convolution with a family of smoothing kernels $(\rho_\varepsilon)$, but we shall skip the details here. 
\end{proof}

The general philosophy of the theory of jet differentials is that their global sections with values in an antiample divisor provide algebraic differential equations which every entire curve must satisfy. We start illustrating this philosophy with the following application of the Ahlfors-Schwarz lemma. First of all, we give the following definition.

\begin{defn}
Let $h_k$ be a $k$-jet metric on the directed manifold $(X,V)$, that is a (possibly singular) hermitian metric on the $k$-th tautological line bundle $\mathcal O_{X_k}(-1)$. We say that $h_k$  has \emph{negative jet curvature} if the curvature $i\,\Theta(\mathcal O_{X_k}(-1),h_k)$ is negative definite along the subbundle $V_k\subset T_{X_k}$ (possibly in the sense of distributions).
\end{defn} 

\begin{rem}
In the special case of $1$-jets, formula (\ref{curvature}) shows that if $h_1$ is smooth and comes from a hermitian metric $h$ on $V$, then $h_1$ has negative jet curvature if and only if $V$ has negative holomorphic sectional curvature. 
\end{rem}

\begin{thm}[\cite{Dem97}]\label{link}
Let $(X,V)$ be a compact complex directed manifold. If $(X,V)$ has a $k$-jet metric $h_k$ with negative jet curvature, then every entire curve $f\colon\mathbb C\to X$ tangent to $V$ is such that $f_{[k]}(\mathbb C)$ is contained in the degeneration set $\Sigma_{h_k}$ of $h_k$. 
\end{thm}

As an immediate corollary we get the following generalization of Proposition \ref{bisectional}.

\begin{cor}
Let $X$ be a compact complex manifold with negative holomorphic sectional curvature. Then $X$ is Brody hence Kobayashi hyperbolic.

More generally, if $X$ is a complex manifold whose holomorphic sectional curvature is bounded above by a negative constant then $X$ is Brody hyperbolic. 
\end{cor}

Here is a more algebraic consequence.

\begin{cor}[\cite{G-G80},\cite{S-Y97}\cite{Dem97}]\label{baselocus}
Assume that there exist integers $k,m>0$ and an ample line bundle $A\to X$ such that
$$
H^0(X_k,\mathcal O_{X_k}(m)\otimes\pi^*_{0,k}A^{-1})\simeq H^0(X,E_{k,m}V^*\otimes A^{-1})
$$
has non zero sections $\sigma_1,\dots,\sigma_N$. Let $Z\subset X_k$ be the base locus of these sections. Then every entire curve $f\colon\mathbb C\to X$ tangent to $V$ is such that $f_{[k]}(\mathbb C)\subset Z$.
\end{cor}

Remark that an analogous statement holds for non necessarily invariant jet differentials, when $V=T_X$.

\begin{proof}
In fact, from the sections $(\sigma_j)$ we get a singular hermitian metric on $\mathcal O_{X_k}(-1)$ whose degeneration set is exactly the base locus $Z$ and whose curvature is bounded above by the one of $\pi^*_{0,k}A^{-1}$ (this is not completely correct, since $\pi^*_{0,k}A^{-1}$ does not bound the \lq\lq vertical eigenvalues\rq\rq{}, see \cite{Dem97} for all the details).
\end{proof}

\begin{proof}[Proof of Theorem \ref{link}]
Chose an arbitrary smooth hermitian metric $\omega_k$ on $T_{X_k}$. By hypothesis there exists $\varepsilon >0$ such that
$$
i\,\Theta(\mathcal O_{X_k}(1),h_k^{-1})(\xi)\ge\varepsilon ||\xi||^2_{\omega_k},\quad\forall\xi\in V_k.
$$
Next, by definition, $(\pi_k)_*$ maps continuously $V_k$ onto $\mathcal O_{X_k}(-1)$; write $h_k$ locally $e^\varphi$ and notice that the weight $\varphi$ is locally bounded from above. Therefore, we can find a constant $C>0$ such that
$$
||(\pi_k)_*\xi||_{h_k}^2\le C\,||\xi||^2_{\omega_k},\quad\forall\xi\in V_k.
$$
Putting together the two inequalities, we get
$$
i\,\Theta(\mathcal O_{X_k}(1),h_k^{-1})(\xi)\ge\frac\varepsilon C||(\pi_k)_*\xi||_{h_k}^2,\quad\forall\xi\in V_k.
$$
Now take any $f\colon\mathbb C\to X$ tangent to $V$, fix an arbitrary radius $R>0$ and consider the restriction $f\colon\Delta_R\to X$. Thus, we get a line bundle morphism
$$
F=f_{[k-1]}'\colon T_{\Delta_R}\to f^*_{[k]}\mathcal O_{X_k}(-1)
$$
by which we can pull-back the metric $h_k$ to obtain a 
$$
\gamma=\gamma_0\,d\zeta\otimes d\overline\zeta=F^*h_k.
$$
If $f_{[k]}(\Delta_R)\subset\Sigma_{h_k}$ then $\gamma\equiv 0$. If not, $\gamma_0$ vanishes precisely at points where $F$ vanishes (which are isolated) and at points of the degeneration set $f_{[k]}^{-1}(\Sigma_{h_k})$, which is polar in $\Delta_R$. At other points $\zeta$, the Gaussian curvature of $\gamma$ is
$$
i\,\partial\overline\partial\log\gamma_0(\zeta)=-i\,f^*_{[k]}\Theta(\mathcal O_{X_k}(-1),h_k)=i\,f^*_{[k]}\Theta(\mathcal O_{X_k}(1),h_k^{-1})
$$ 
and, when computed as hermitian form on $\partial/\partial\zeta$, one gets
$$
\begin{aligned}
i\,f^*_{[k]}\Theta(\mathcal O_{X_k}(1),h_k^{-1})\biggl(\frac\partial{\partial\zeta}\biggr) & =
i\,\Theta(\mathcal O_{X_k}(1),h_k^{-1})(f_{[k]}'(\zeta)) \\
& \ge\frac\varepsilon C||f'_{[k-1]}(\zeta)||_{h_k}^2 \\
& = \frac\varepsilon C\gamma_0(\zeta),
\end{aligned}
$$
since $f'_{[k-1]}(\zeta)=(\pi_k)_*f'_{[k]}(\zeta)$ and $\gamma=F^* h_k$. Thus, the Ahlfors-Schwarz lemma implies 
$$
\gamma(\zeta)\le\frac{2C}\varepsilon\frac{R^{-2}|d\zeta|^2}{(1-|\zeta|^2/R^2)^2},
$$
that is
$$
||f'_{[k-1]}(\zeta)||_{h_k}^2\le\frac{2C}\varepsilon\frac{R^{-2}}{(1-|\zeta|^2/R^2)^2}.
$$
Letting $R$ tend to infinity, we obtain that $f_{[k-1]}$ must be constant, and hence $f$, too.
\end{proof}

\subsection{The Bloch theorem}

Bloch's theorem is a characterization of the (Zariski) closure of entire curves on a complex torus. Following \cite{Dem97}, we shall derive it here as a consequence of Theorem \ref{link}.

\begin{thm}\label{bloch}
Let $Z$ be a complex torus and let $f\colon\mathbb C\to Z$ be an entire curve. Then the (analytic) Zariski closure $\overline{f(\mathbb C)}^{\text{\rm Zar}}$ is a translate of a subtorus.
\end{thm}

The converse is clearly true, too: for any translate of a subtorus $a+Z'\subset Z$, one can choose a dense line $L\subset Z'$ and the corresponding map $f\colon\mathbb C\simeq a+L\hookrightarrow Z$ has Zariski closure $\overline{f(\mathbb C)}^{\text{\rm Zar}}=a+Z'$. 

Before giving the proof, we list here some immediate consequences (the first of which has already been proved by elementary methods in Chapter 1).

\begin{cor}
Let $X$ be a complex analytic subvariety in a complex torus $Z$. Then $X$ is hyperbolic if and only if $X$ does not contain any translate of a subtorus.
\end{cor}

\begin{cor}
Let $X$ be a complex analytic subvariety of a complex torus $Z$. If $X$ is not a translate of a subtorus then every entire curve drawn in $X$ is analytically degenerate.
\end{cor}

In particular, if $X$ is a complex analytic subvariety of general type of a complex torus $Z$ then it cannot be a translate of a subtorus and the corollary applies. Anyway, observe that {\sl a priori} this corollary just states that there are no Zariski dense entire curves in such a subvariety and not the stronger property of the existence of a closed proper subvariety of $X$ which contains the images of every entire curve in $X$. The corollary is thus a weak confirmation of the Green-Griffiths conjecture for subvarieties of general type of complex tori.

Nevertheless, once we have the analytic degeneracy of entire curves, the stronger version of the Green-Griffiths conjecture can be deduced from the following result of Kawamata.

\begin{thm}[\cite{Kaw80}]
Let $X$ be a subvariety of general type of a complex torus $Z$. Then there is a proper subvariety
$Y\subset X$ which contains all the translates of subtori contained in $X$.
\end{thm}

Here is the Bloch theorem itself \cite{Blo26, Kaw80}.

\begin{cor}[Bloch's theorem]
Let $X$ be a compact complex K\"ahler manifold such that the irregularity $q=h^0(X,T^*_X)$ is larger than the dimension $n=\dim X$. Then every entire curve drawn in $X$ is analytically degenerate.
\end{cor}

Observe that if the irregularity is larger than the dimension there exists a map from $X$ to a complex torus $Z$ such that the image of $X$ is of general type.

\begin{proof}
The Albanese map $\alpha\colon X\to\operatorname{Alb}(X)$ sends $X$ onto a proper subvariety $Y\subset\operatorname{Alb}(X)$, since $\dim\operatorname{Alb}(X)=q>n$. Moreover, by the universal property of the Albanese map, $\alpha(X)$ is not a translate of a subtorus. Hence, given an entire curve $f\colon\mathbb C\to X$, the composition $\alpha\circ f\colon\mathbb C\to Y$ is analytically degenerate; but then, $f$ itself is analytically degenerate.
\end{proof}

Now, we give a

\begin{proof}[Proof of Theorem \ref{bloch}]
Let $f\colon\mathbb C\to Z$ be an entire curve and let $X$ be the Zariski closure of its image. Call $Z_k$ the projectivized $k$-jet bundle of $(Z,T_Z)$ and $X_k$ the closure of the projectivized $k$-jet bundle of $(X_{\text{\rm reg}},T_{X_{\text{\rm reg}}})$ in $Z_k$. As $T_Z\simeq Z\times\mathbb C^n$, $\dim Z=n$, we have that $Z_k=Z\times\mathbb R_{n,k}$, where $\mathbb R_{n,k}$ is the rational variety introduced in the previous chapter. By Proposition \ref{relnef}, there exists a weight $\mathbf a\in\mathbb N^k$ such that $\mathcal O_{Z_k}(\mathbf a)$ is relatively very ample. This means that there exists a very ample line bundle $\mathcal O_{\mathbb R_{n,k}}(\mathbf a)$ over $\mathbb R_{n,k}$ such that $\mathcal O_{Z_k}(\mathbf a)$ is its pull-back by the second projection. Now, consider the restriction to $X_k$ of the second projection map and call it $\Phi_k\colon X_k\to\mathbb R_{n,k}$; by functoriality, one has that $\mathcal O_{X_k}(\mathbf a)=\Phi_k^*\mathcal O_{\mathbb R_{n,k}}(\mathbf a)$.

Define $B_k\subset X_k$ to be the set of points $x\in X_k$ such that the fiber of $\Phi_k$ passing through $x$ is positive dimensional and assume that $B_k\ne X_k$.

\begin{lem}
There exists a singular hermitian metric on $\mathcal O_{X_k}(\mathbf a)$ with strictly positive curvature current and such that its degeneration set is exactly $B_k$.
\end{lem}

\begin{proof}
This a quite standard fact. A proof can be found for example in \cite{Dem97}. 
\end{proof}

Thus, Theorem \ref{link} shows that $f_{[k]}(\mathbb C)\subset B_k$ (and this is trivially also true if $B_k=X_k$).This means that through every point $f_{[k]}(t_0)$ there is a germ of positive dimensional variety in the fiber $\Phi_k^{-1}(\Phi_k(f_{[k]}(t_0)))$, say a germ of curve $\zeta\to u(\zeta)=(z(\zeta),j_k)\in X_k\subset Z\times\mathbb R_{n,k}$ with $u(0)=f_{[k]}(t_0)=(z_0,j_k)$ and $z_0=f(t_0)$. Then $(z(\zeta),j_k)$ is the image of $f_{[k]}(t_0)$ by the $k$-th lifting of the translation $\tau_s\colon z\mapsto z+s$ defined by $s=z(\zeta)-z_0$.

Next, we have $f(\mathbb C)\not\subset X_{\text{\rm sing}}$ since $X$ is the Zariski closure of $f(\mathbb C)$ and we may then choose $t_0$ so that $f(t_0)\in X_{\text{\rm reg}}$ and $f'(t_0)\ne 0$. Define
$$
A_k(f)=\{s\in Z\mid f_{[k]}(t_0)\in X_k\cap \tau_{-s}(X)_k\}.
$$
Clearly $A_k(f)$ is an analytic subset of $Z$ containing the curve $\zeta\to s(\zeta)=z(\zeta)-z_0$ through $0$. By the Noetherian property, since 
$$
A_1(f)\supset A_(f)\supset\cdots\supset A_k(f)\supset\cdots,
$$
this sequence stabilizes at some $A_k(f)$.
Therefore, there is a curve $\Delta_r\to Z$, $\zeta\mapsto s(\zeta)$ such that the infinite jet $j_\infty$ defined by $f$ at $t_0$ is $s(\zeta)$-translation invariant for all $t\in\mathbb C$ and $\zeta\in\Delta_r$. As $X$ is the Zariski closure of $f(\mathbb C)$, we must have $s(\zeta)+X\subset X$ for all $\zeta\in\Delta_r$; moreover, $X$ is irreducible and then we have in fact $s(\zeta)+X=X$.

Now, define
$$
W=\{s\in Z\mid s+X=X\}.
$$
Then, $W$ is a closed positive dimensional subgroup of $Z$. Let $p\colon Z/W$ be the natural projection. As $Z/W$ is a complex torus with of strictly lower dimension that $Z$, we conclude by inducion on dimension that the curve $\widehat f=p\circ f\colon\mathbb C\to Z/W$ has its Zariski closure $\widehat X=\overline{\widehat f(\mathbb C)}^{\text{\rm Zar}}=p(X)$ equal to a translate $\widehat s+\widehat T$ of some subtorus $\widehat T\subset Z/W$. Since $X$ is $W$-invariant, we get $X=s+p^{-1}(\widehat T)$, where $p^{-1}(\widehat T)$ is a closed subgroup of $Z$. This implies that $X$ is a translate of a subtorus.
\end{proof}

\section{A general strategy for algebraic degeneracy}

We state here a quite general theorem which gives sufficient conditions in order to have algebraic degeneracy of entire curves in a given compact complex manifold. This statement somehow combine and sums up several ideas and strategies by different authors in the last three decades.

\begin{thm}[\cite{Siu04}, \cite{Pau08}, \cite{Rou07}, \cite{D-M-R10}, \cite{D-T10}]\label{general}
Let $X$ be a compact complex manifold. Suppose there exist two ample line bundle $A,B\to X$ and integers $k,m>0$ such that
\begin{itemize}
\item[(i)] there is a non zero section $P\in H^0(X,E_{k,m}T^*_X\otimes A^{-1})$,
\item[(ii)] the twisted tangent space $T_{J_k T_X}\otimes p_k^*B$ of the space of $k$-jets $p_k\colon J_k T_X\to X$ is globally generated over its regular part $J_kT_X^{\text{\rm reg}}$ by its global sections, and suppose moreover that one can choose such generating vector fileds to be equivariant with respect to the action of $\mathbb G_k$ on $J_k T_X$,
\item[(iii)] the line bundle $A\otimes B^{-\otimes m}$ is ample.
\end{itemize}
Call $Y\subsetneq X$ the zero locus of the section $P$. Then every holomorphic entire curve $f\colon\mathbb C\to X$ has image contained in $Y$.

If moreover the effective cone of $X$ is contained in the ample cone (which is the case for instance if the Picard group of $X$ is $\mathbb Z$), then the ample line bundle $A$ can be chosen in such a way to force $Y$ to have codimension at least two in $X$.
\end{thm}

\begin{proof}
Start with the given non zero section $P\in H^0(X,E_{k,m}T^*_X\otimes A^{-1})$ and call 
$$
Y=\{P=0\}\subsetneq X
$$ 
its zero locus. Look at $P$ as an invariant (under the action of the group $\mathbb G_k$) map 
$$
J_kT_X\to p^*_k A^{-1}
$$
where $p_k\colon J_k T_X\to X$ is the space of $k$-jets of germs of holomorphic curves $f\colon (\mathbb C,0)\to X$. Then $P$ is a weighted homogeneous polynomial in the jet variables of degree $m$ with coefficients holomorphic functions of the coordinates of $X$ and values in $p^*_k A^{-1}$.

By hypothesis (ii), we have enough global holomorphic invariant vector fields on $J_k T_X$ with values in the pull-back from X of the ample divisor $B$ in order to generate $T_{J_k T_X}\otimes p^*_kB$, at least over the dense open set $J_kT_X^{\text{\rm reg}}$ of regular $k$-jets, {\sl i.e.} of $k$-jets with nonvanishing first derivative. Considering jet differentials as functions on $J_k T_X$, the idea is to produce lots of them starting from the first one simply by derivation.

If $f\colon\mathbb C\to X$ is an entire curve, consider its lifting $j_k(f)\colon\mathbb C\to J_k T_X$ and suppose that $j_k(f)(\mathbb C)\not\subset J_kT_X^{\text{\rm sing}}\overset{\text{\rm def}}=J_k T_X\setminus J_kT_X^{\text{\rm reg}}$ (otherwise $f$ is constant). Arguing by contradiction, let $f(\mathbb C)\not\subset Y$ and $x_0=f(\zeta_0)\in X\setminus Y$.

Then, by Corollary \ref{baselocus}, the point $j_k(f)(\zeta_0)$ must lie in the zero locus of the restriction $P_{x_0}$ of $P$ to the affine fiber $p_k^{-1}(x_0)$ (see \figurename~\ref{fig:algdeg}): this restriction is not identically zero by construction, since the point $x_0$ is outside the zero locus of $P$ regarded as a section of $E_{k,m}T^*_X\otimes A^{-1}$.

\begin{figure}
\centering
\includegraphics[scale=0.8]{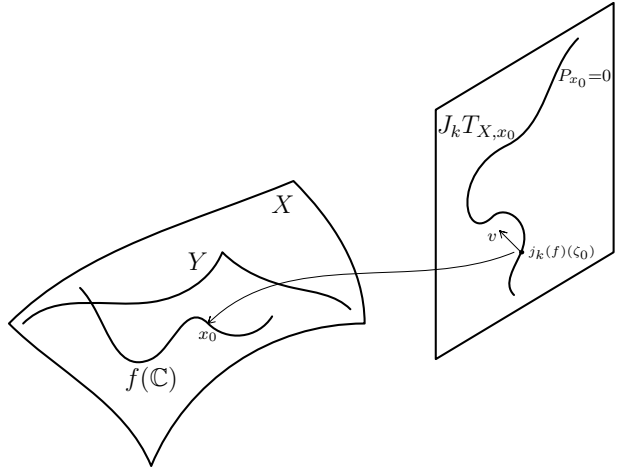}
\caption{The lifting of the curve $f$ and the zero locus of $P_{x_0}$\label{fig:algdeg}}
\end{figure}

If we take a global section $V\in H^0(J_kT_X,T_{J_k T_X}\otimes p_k^*B)$, by combining with the pairing 
$$
H^0(J_kT_X,p_k^*A^{-1})\times H^0(J_kT_X,p_k^*B)\to H^0(J_kT_X,p_k^*(A^{-1}\otimes B)),
$$ 
we can form the Lie derivative $L_VP$ of $P$ with respect to $V$ and obtain in this way a new global invariant (since $V$ is an invariant vector field) jet differential of weighted degree $m$ and order $k$ with values in $A^{-1}\otimes B$.

Since $P_{x_0}$ is a polynomial of weighted degree $m$, one can find $v_1,\dots,v_p\in T_{J_kT_X,j_k(f)(\zeta_0)}$, $p\le m$, such that given $V_1,\dots V_p\in H^0(J_kT_X,T_{J_k T_X}\otimes p_k^*B)$ with the property that $V_j(j_k(f)(\zeta_0))=v_j$, $j=1,\dots,p$, one has
$$
L_{V_p}\cdots L_{V_1} P(j_k(f)(\zeta_0))\ne 0.
$$
For instance, if the point $j_k(f)(\zeta_0)$ is a regular point of the zero locus of $P_{x_0}$ in $J_kT_{X,x_0}$, it suffices to take $p=1$ and any $v=v_1\in T_{J_kT_X,j_k(f)(\zeta_0)}$ transverse to the zero locus will do the job (see again \figurename~\ref{fig:algdeg}). In general, this zero locus has at most a singularity of order $m$ at $j_k(f)(\zeta_0)$ and thus one needs to take at most $m$ derivatives in order to guarantee the non vanishing of the derived polynomial at the given point.
The existence of the global sections with prescribed valued at $j_k(f)(\zeta_0)$ is assured by the global generation hypothesis (ii).

Summing up, one can produce, by differentiating at most $m$ times, a new invariant $k$-jet differential $Q=L_{V_p}\cdots L_{V_1}P$ of weighted degree $m$ with values in $A^{-1}\otimes B^{\otimes p}$ such that $Q(j_k(f)(\zeta_0))\ne 0$, thus contradicting Corollary \ref{baselocus}, provided $A^{-1}\otimes B^{\otimes p}$ is antiample, {\sl i.e.} provided $Q$ is still with value in an antiample divisor. But this is assured by hypothesis (iii), since $p\le m$.

For the last assertion, the starting point is the following general, straightforward remark. Let $E\to X$ be a holomorphic vector bundle over a compact complex manifold $X$ and let $\sigma\in H^0(X,E)\ne 0$; then, up to twisting by the dual of an effective divisor, one can suppose that the zero locus of $\sigma$ has no divisorial components. This is easily seen, for let $D$ be the divisorial (and effective) part of the zero locus of $\sigma$ and twist $E$ by $\mathcal O_X(-D)$. Then, $\sigma$ is also a holomorphic section of $H^0(X,E\otimes\mathcal O_X(-D))$ and seen as a section of this new bundle, it vanishes on no codimension $1$ subvariety of $X$.

Now, we use this simple remark in our case, the vector bundle $E$ being here $E_{k,m}T^*_X\otimes A^{-1}$. Since, by hypothesis, every effective line bundle on $X$ is ample the corresponding $\mathcal O_X(-D)$ is antiample. After all, we see that in the above proof of one can suppose that the \lq\lq first\rq\rq{} invariant jet differential vanishes at most on a codimension $2$ subvariety of $X$, provided one looks at it as a section of $H^0(E_{k,m}T^*_X\otimes A^{-1}\otimes\mathcal O_X(-D))$ and $A^{-1}\otimes\mathcal O_X(-D)$ is again antiample. In fact it is even \lq\lq more antiample\rq\rq{} than $A^{-1}$: in particular, condition (iii) in the hypotheses is still fulfilled:
$$
A\otimes\mathcal O_X(D)\otimes B^{-\otimes m}>A\otimes B^{-\otimes m}>0
$$
and everything works in the same way, but now with a $Y$ of codimension at least two. 
\end{proof}

One then gets immediately.

\begin{cor}
A compact complex surface which satisfies the hypotheses of Theorem \ref{general} is Kobayashi hyperbolic.
\end{cor}

\begin{cor}
Let $X$ be a compact complex threefold which satisfies the hypotheses of Theorem \ref{general}. Suppose moreover that $X$ does not contain any rational or elliptic curve. Then $X$ is Kobayashi hyperbolic. 
\end{cor}

\begin{proof}
Let $f\colon\mathbb C\to X$ an entire curve in $X$. Then $\overline{f(\mathbb C)}^{\text {\rm Zar}}$ is an algebraic curve of $X$ which admits a non constant holomorphic image of $\mathbb C$. By uniformization, it must be rational or elliptic, contradiction.
\end{proof}

\chapter{Hyperbolicity of generic surfaces in projective $3$-space}

{\small\textsc{Abstract}. The main topic of these notes is hyperbolicity of generic projective hypersurfaces of high degree. In this chapter we shall describe how to prove it in the simpler case of surfaces in projective $3$-space. While Kobayashi's conjecture predicts in the case of surfaces a lower bound for the degree equal to $5$, nowadays the hyperbolicity is only known for degree greater than or equal to $18$ \cite{Pau08}, after 36 \cite{McQ99} and 21 \cite{D-EG00}.}

\section{General strategy}

The idea we present here to attack the Kobayashi conjecture is to apply Theorem \ref{general} in the context of projective hypersurfaces. Unluckily, to verify the validity of hypothesis (ii) even in this framework is very difficult (and maybe even not true): we shall instead use a kind of the variational method introduced in \cite{Voi96} as explained in \cite{Siu04}. 

The strategy is the following. First of all, it is known \cite{Sak79} that on a smooth projective hypersurface of dimension greater than or equal to two there is no symmetric differentials at all. This means that we cannot expect to work with order one jet differentials for this conjecture. In fact, in general, one has the following 

\begin{thm}[\cite{Div08}]\label{vanishing}
On a smooth projective complete intersection $X$ one has the following vanishing:
$$
H^0(X,E_{k,m}T^*_X)=0,\quad\forall m>0\quad\text{and}\quad 0<k<\dim X/\operatorname{codim} X. 
$$
In particular, whenever $X$ is a smooth hypersurface one has to look at least for invariant jet differentials of order equal to the dimension of $X$. 
\end{thm}

Following \cite{Dem97}, for $X$ a smooth surface in projective $3$-space we shall see that to find the first non zero jet differential (see hypothesis (i) of Theorem \ref{general}) it suffices to apply a Riemann-Roch argument together with a vanishing theorem of Bogomolov \cite{Bog78}. One gets in this way an order two jet differential on every smooth hypersurface in $\mathbb P^3$ of degree greater than or equal to $15$; a more involved Riemann-Roch computation made in \cite{G-G80} shows in fact that on every smooth projective surface of general type there is a global jet differential of some (possibly very high) order\footnote{Recently, in july 2010, Demailly has announced to be able to construct a non zero jet differential on every variety of general type by differential geometric and probabilistic techniques: this is a substantial step toward the Green-Griffiths conjecture in arbitrary dimension.}. Thus one has a global jet differential on every smooth projective hypersurface in $\mathbb P^3$ starting from degree $5$ (this is the first degree for which the canonical bundle is ample). Anyway, for our purposes order two techniques will suffice.

Then, one has to produce meromorphic vector fields with controlled pole order like in the case of the proof of algebraic hyperbolicity. For this, one consider the universal hypersurface $\mathcal X\subset\mathbb P^3\times\mathbb P^{N_d-1}$ of degree $d$ in $\mathbb P^3$ and the vertical tangent bundle $\mathcal V\subset T_{\mathcal X}$, kernel of the differential of the second projection. Next, one forms the corresponding directed manifold $(\mathcal X,\mathcal V)$: entire curves in $\mathcal X$ tangent to $\mathcal V$ are in fact contained in some fiber and then map to some hypersurface. The aim is to globally generate the (twisted) tangent space $T_{J_2\mathcal V}\otimes\mathcal O_{\mathbb P^3}(\bullet)\otimes\mathcal O_{\mathbb P^{N_d-1}}(\bullet)$ to vertical  $2$-jets: this is done in \cite{Pau08} using \lq\lq slanted\rq\rq{} vector fields which permit to gain some positivity from the moduli space.

But now, the problem is that in order to be able to take derivatives one needs global jet differentials not only defined over the fibers (that is over every single smooth hypersurface) but over an open set in $\mathcal X$.

The vector bundle $E_{k,m}\mathcal V^*\to\mathcal X$ has the tautological property that its restriction to any smooth fiber $X_s$ of the second projection (that is to the hypersurface corresponding to a given modulus $s$ in $\mathbb P^{N_d-1}$) is just the vector bundle $E_{k,m}T^*_{X_s}\to X_s$. Since we know that on any smooth projective hypersurface of degree $\ge 15$ there is an order two global jet differential we can use a standard semicontinuity argument in order to extend it to a section of $E_{2,m}\mathcal V^*$ over (the inverse image by the second projection of) a Zariski open set of the moduli space: this is the desired extension to use the meromorphic vector fields. 

We shall see later how the more technical in nature hypothesis (iii) of Theorem \ref{general} is fulfilled in this situation. 

\section{Existence of jet differentials}

We start this section with a simple Riemann-Roch computation for order one jet differential on a smooth surface in order to give the flavor of the kind of methods employed.

Let $X$ be a smooth compact surface. Consider the projectivization $\pi\colon P(T_X)\to X$ of lines of $T_X$ and the corresponding (anti)tautological line bundle $\mathcal O_{P(T_X)}(1)$; call $u$ the first Chern class of $\mathcal O_{P(T_X)}(1)$.

On the one hand, we have the following standard relation:
$$
u^2+\pi^* c_1(X)\cdot u+\pi^*c_2(X)=0.
$$
On the other hand, we know that the following higher direct image formula is valid:
$$
R^q\pi_*\mathcal O_{P(T_X)}(m)=
\begin{cases}
\mathcal O_X(S^m T^*_X) & \text{if $q=0$ and  $m\ge 0$,} \\
0 & \text{otherwise.}
\end{cases}
$$
Therefore, we have the following isomorphism in cohomology:
$$
H^q(X,S^mT^*_X)\simeq H^q(P(T_X),\mathcal O_{P(T_X)}(m)),\quad\forall m,q\ge 0.
$$
In particular, we have equality for the Euler characteristics 
$$
\chi(S^mT^*_X)=\chi(\mathcal O_{P(T_X)}(m)).
$$
Now, the HirzebruchÐRiemannÐRoch theorem gives us
$$
\begin{aligned}
\chi(\mathcal O_{P(T_X)}(m)) & =\int_{P(T_X)}\operatorname{ch}(\mathcal O_{P(T_X)}(m))\cdot \operatorname{Td}(P(T_X)) \\
& = \frac{m^3}{3!}u^3+O(m^2) \\
& = \frac{m^3}{6}(c_1(X)^2-c_2(X))+O(m^2),
\end{aligned}
$$
where the last equality is obtained using the above relation. We then have that if the second Segre number $c_1(X)^2-c_2(X)$ of $X$ is positive, then the asymptotic Euler characteristic of the symmetric powers of the cotangent bundle of $X$ has maximal growth. Moreover, since we deal with the asymptotic Euler characteristic, the same result holds true if we twist the symmetric powers by any fixed line bundle. 

Concerning the existence of sections, suppose now that $X$ is a smooth surface of general type (that is $K_X$ is big). Then, a vanishing theorem of Bogomolov contained in \cite{Bog78} implies that
$$
H^0(X,S^pT_X\otimes K_X^{\otimes q})=0,\quad\forall p-2q>0.
$$
In particular, $H^0(X,S^mT_X\otimes K_X)=0$ whenever $m\ge 3$. But then, 
$$
\begin{aligned}
h^0(X,S^mT^*_X) & = h^0(P(T_X),\mathcal O_{P(T_X)}(m)) \\
& \ge \chi (\mathcal O_{P(T_X)}(m))-h^2(P(T_X),\mathcal O_{P(T_X)}(m)) \\
& = \chi (\mathcal O_{P(T_X)}(m))-h^2(X,S^mT^*_X) \\
& = \chi (\mathcal O_{P(T_X)}(m))-h^0(X,S^mT_X\otimes K_X) \\
& =  \chi (\mathcal O_{P(T_X)}(m)),\quad m\ge 3.
\end{aligned}
$$
Thus, the zeroth cohomology group is asymptotically minorated by the Euler characteristic, which is asymptotically positive if the second Segre number is.

Unfortunately, as we have said, this first order result is not sufficient to deal with surfaces in $\mathbb P^3$. For order two jets, we have seen that the full composition series of $E_{2,m}T^*_X$ is given by
$$
\operatorname{Gr}^\bullet(E_{2,m}T^*_X)=\bigoplus_{j=0}^{m/3}S^{m-3j}T^*_X\otimes K_X^j.
$$
Thus, a slightly more involved Euler characteristic computation gives the following:
$$
\chi(E_{2,m}T^*_X)=\chi(\operatorname{Gr}^\bullet(E_{2,m}T^*_X))=\frac{m^4}{648}(13\,c_1(X)^2-9\,c_2(X))+O(m^3).
$$
Again by Bogomolov's vanishing we obtain

\begin{prop}[\cite{Dem97}]
If $X$ is an algebraic surface of general type and $A$ an ample line bundle over $X$, then
$$
h^0(X,E_{2,m}T^*_X\otimes A^{-1})\ge \frac{m^4}{648}(13\,c_1(X)^2-9\,c_2(X))+O(m^3).
$$
In particular, every smooth surface $X\subset\mathbb P^3$ of degree $d\ge 15$ admits non trivial sections of $E_{2,m}T^*_X\otimes A^{-1}$ for $m$ large.
\end{prop}

\begin{proof}
Only the last assertion remains to be proved. But it follows from a standard computation of Chern classes of smooth projective hypersurface. In this case, in fact, we have
$$
c_1(X)=(4-d)\,h,\quad c_2(X)=(d^2-4d+6)\,h^2,
$$
where $h$ is the hyperplane class on $X$, $h^2=d$, and then $13\,c_1(X)^2-9\,c_2(X)>0$ for $d\ge 15$.
\end{proof}

This means that smooth projective surfaces in $\mathbb P^3$ of degree $d\ge 15$ satisfy hypothesis (i) of Theorem \ref{general} for $k=2$ and $m$ large enough.

In the sequel, we shall need a slightly more general and precise knowledge of the vanishing order of these sections. Thus, we shall need the following result.

\begin{prop}[see also \cite{Pau08}]\label{existencesurfaces}
Let $X$ be a projective surface of general type. Then
\begin{multline*}
h^0(X,E_{2,m}T^*_X\otimes K_X^{-\delta m})\ge \\ \frac{m^4}{648}\bigl((54\delta^2-48\delta+13)\,c_1(X)^2-9\,c_2(X)\bigr)+O(m^3),
\end{multline*}
provided $0\le\delta<1/3$.
\end{prop}

Please note that the original statement contained in \cite{Pau08} has a minor computational error for the Euler characteristic which has been fixed here.

\begin{proof}
First of all, one shows using 
$$
\operatorname{Gr}^\bullet(E_{2,m}T^*_X\otimes K_X^{-\delta m})=\bigoplus_{j=0}^{m/3}S^{m-3j}T^*_X\otimes K_X^{j-\delta m}
$$
that the right hand side is the Euler characteristic of $E_{2,m}T^*_X\otimes K_X^{-\delta m}$. Then, as usual, one applies Serre's duality for the second cohomology group and the Bogomolov vanishing to each piece of the dualized graded bundle: they are of the form
$$
S^{m-3j}T_X\otimes K_X^{\delta m -j+1}
$$
and they vanish if $m-3j-2(\delta m -j+1)>0$, which is the case if $\delta<1/3-1/m$ since $0\le j\le m/3$. The result follows, since the inequality in the statement is asymptotic with $m\to +\infty$.
\end{proof}

\section[Global generation on the universal family]{Global generation of the twisted tangent space of the universal family}

In this section, we shall reproduce the proof of \cite{Pau08} of the twisted global generation of the tangent space of the space of vertical two jets. First of all, we fix again the notation.

Consider the universal hypersurface $\mathcal X\subset\mathbb P^3\times\mathbb P^{N_d-1}$ of degree $d$ given by the equation
$$
\sum_{|\alpha|=d}A_{\alpha}\,Z^\alpha,
$$
where $[A]\in\mathbb P^{N_d-1}$, $[Z]\in\mathbb P^3$ and $\alpha=(\alpha_0,\dots,\alpha_3)\in\mathbb N^4$. 

Next, we fix the affine open set $U=\{Z_0\ne 0\}\times\{A_{0d00}\ne 0\}\simeq\mathbb C^{3}\times\mathbb C^{N_d-1}$ in $\mathbb P^{3}\times\mathbb P^{N_d-1}$ with the corresponding inhomogeneous coordinates $(z_j=Z_j/Z_0)_{j=1,2,3}$ and $(a_{\alpha}=A_\alpha/A_{0d00})_{|\alpha|=d,\alpha_1<d}$. Since $\alpha_0$ is determined by $\alpha_0=d-(\alpha_1+\alpha_2+\alpha_3)$, with a slight abuse of notation in the sequel $\alpha$ will be seen as a multiindex $(\alpha_1,\alpha_2,\alpha_3)$ in $\mathbb N^3$, with moreover the convention that $a_{d00}=1$.

On this affine open set we have
$$
\mathcal X_0:=\mathcal X\cap U=\left\{z_1^d+\sum_{|\alpha|\le d,\alpha_1<d}a_{\alpha}\,z^\alpha=0\right\}.
$$
We now write down equations for the open variety $J_2\mathcal V_0$, where we indicated with $\mathcal V_0$ the restriction of $\mathcal V\subset T_{\mathcal X}$, the kernel of the differential of the second projection, to $\mathcal X_0$: elements in $J_2\mathcal V_0$ are therefore $2$-jets of germs of  \lq\lq vertical\rq\rq{} holomorphic curves in $\mathcal X_0$, that is curves tangent to vertical fibers. Here are the equations, which naturally live in $\mathbb C^3_{z_j}\times\mathbb C^{N_d-1}_{a_\alpha}\times\mathbb C^3_{z_j'}\times\mathbb C^3_{z_j''}$:
$$
\sum_{|\alpha|\le d}a_{\alpha}\,z^{\alpha}=0,
$$
$$
\sum_{j=1}^3\sum_{|\alpha|\le d}a_{\alpha}\,\frac{\partial z^\alpha}{\partial z_j}\,z_j'=0,
$$
$$
\sum_{j=1}^3\sum_{|\alpha|\le d}a_{\alpha}\,\frac{\partial z^\alpha}{\partial z_j}\,z_j''
+\sum_{j,k=1}^3\sum_{|\alpha|\le d}a_{\alpha}\,\frac{\partial^2 z^\alpha}{\partial z_j\partial z_k}\,z_j'z_k'=0.
$$
Define $\Sigma_0$ to be the closed algebraic subvariety of $J_2\mathcal V_0$ defined by
$$
\Sigma_0=\{(z,a,z',z'')\in J_2\mathcal V_0\mid z'\wedge z''=0\}
$$
and let $\Sigma$ be the Zariski closure of $\Sigma_0$ in $J_2\mathcal V$: we call this set the \emph{wronskian locus} of $J_2\mathcal V$.

To begin with, observe that an affine change of coordinates $z\mapsto 1/z$ induces on jet variables the following transformation rules
$$
z'\mapsto -\frac{z'}{z^2}\quad\text{and}\quad z''\mapsto \frac{2(z')^2-zz''}{z^3}.
$$
Now, consider a general vector field in the vector space $\mathbb C^3_{z_j}\times\mathbb C^{N_d-1}_{a_\alpha}\times\mathbb C^3_{z_j'}\times\mathbb C^3_{z_j''}$; it is of the form
$$
V=\sum_{|\alpha|\le d,\alpha_1<d}v_\alpha\,\frac{\partial}{\partial a_{\alpha}}
+\sum_{j=1}^3v_j\,\frac{\partial}{\partial z_j}+\sum_{j=1}^3\xi_j^{(1)}\,\frac{\partial}{\partial z_j'}
+\sum_{j=1}^3\xi_j^{(2)}\,\frac{\partial}{\partial z_j''}.
$$
Thus, the conditions to be satisfied by the coefficients of $V$ in order to belong to $J_2\mathcal V_0$ are:

$$
\sum_{|\alpha|\le d, \alpha_1<d}v_\alpha\,z^{\alpha}+\sum_{j=1}^3\sum_{|\alpha|\le d}a_\alpha\,\frac{\partial z^\alpha}{\partial z_j}\,v_j=0,
$$

\begin{multline*}
\sum_{j=1}^3\sum_{|\alpha|\le d,  \alpha_1<d}v_\alpha\,\frac{\partial z^\alpha}{\partial z_j}\,z_j'
 \\ + \sum_{j,k=1}^3\sum_{|\alpha|\le d}a_\alpha\,\frac{\partial^2 z^\alpha}{\partial z_k\partial z_j}\,v_kz_j'
+\sum_{j=1}^3\sum_{|\alpha|\le d}a_\alpha\,\frac{\partial z^\alpha}{\partial z_j}\,\xi_j^{(1)},
\end{multline*}

\begin{multline*}
\sum_{|\alpha|\le d,  \alpha_1<d}\left(\sum_{j=1}^3\frac{\partial z^\alpha}{\partial z_j}\,z_j''+
\sum_{j,k=1}^3\frac{\partial^2 z^\alpha}{\partial z_k\partial z_j}\,z_j'z_k'\right)v_\alpha \\
+\sum_{j=1}^3\sum_{|\alpha|\le d}a_\alpha\left(\sum_{k=1}^3\frac{\partial^2 z^\alpha}{\partial z_k\partial z_j}\,z_k''+\sum_{i,k=1}^3\frac{\partial^3 z^\alpha}{\partial z_i\partial z_k\partial z_j}z_k'z_i'\right)v_j \\
+\sum_{|\alpha|\le d}\sum_{j,k=1}^3 a_\alpha\,\frac{\partial^2 z^\alpha}{\partial z_k\partial z_j}\bigl(\xi_j^{(1)}z_k'+\xi_k^{(1)}z_j'\bigr)+\sum_{j=1}^3 a_\alpha\,\frac{\partial z^\alpha}{\partial z_j}\,\xi_j^{(2)}.
\end{multline*}

\subsection*{First family of tangent vector fields}

For all multiindex $\alpha$ such that $\alpha_1\ge 3$, consider the vector field
$$
V_{\alpha}^{300}=\frac{\partial}{\partial a_{\alpha}}-3z_1\,\frac{\partial}{\partial a_{\alpha-\delta_1}}+3z_1^2\,\frac{\partial}{\partial a_{\alpha-2\delta_1}}-z_1^3\,\frac{\partial}{\partial a_{\alpha-3\delta_1}},
$$
where $\delta_j\in\mathbb N^4$ is the multiindex whose $j$-th component is equal to $1$ and the others are zero. For the multiindexes $\alpha$ which verify $\alpha_1\ge 2$ and $\alpha_2\ge 1$, define
$$
\begin{aligned}
V_{\alpha}^{210}&=\frac{\partial}{\partial a_{\alpha}}-2z_1\,\frac{\partial}{\partial a_{\alpha-\delta_1}}-z_2\,\frac{\partial}{\partial a_{\alpha-\delta_2}}+z_1^2\,\frac{\partial}{\partial a_{\alpha-2\delta_1}} \\ &\qquad+2z_1z_2\,\frac{\partial}{\partial a_{\alpha-\delta_1-\delta_2}}-z_1^2z_2\,\frac{\partial}{\partial a_{\alpha-2\delta_1-\delta_2}}.
\end{aligned}
$$
Finally, for those $\alpha$ for which $\alpha_1,\alpha_2,\alpha_3\ge 1$, set
$$
\begin{aligned}
V_{\alpha}^{111}&=\frac{\partial}{\partial a_{\alpha}}-z_1\,\frac{\partial}{\partial a_{\alpha-\delta_1}}-z_2\,\frac{\partial}{\partial a_{\alpha-\delta_2}}-z_3\,\frac{\partial}{\partial a_{\alpha-\delta_3}} \\
&\qquad+z_1z_2\,\frac{\partial}{\partial a_{\alpha-\delta_1-\delta_2}}+z_1z_3\,\frac{\partial}{\partial a_{\alpha-\delta_1-\delta_3}}+z_2z_3\,\frac{\partial}{\partial a_{\alpha-\delta_2-\delta_3}} \\
&\qquad\quad -z_1z_2z_3\frac{\partial}{\partial a_{\alpha-\delta_1-\delta_2-\delta_3}}.
\end{aligned}
$$
The pole order of these vector fields is equal to $3$, as a change of variables easily shows. Moreover, they are all tangent to $J_2\mathcal V_0$ and invariant under the action of $\mathbb G_2$ (because they do not contain any jet variable, on which the group acts).

Of course, there are similarly defined vector fields constructed by permuting the $z$-variables, and changing the multiindex $\alpha$ as indicated by permutations: it is straightforward to see that all these vector fields together span a codimension $7$ vector space in $\ker (T_{J_2\mathcal V}\to T_{J_2T_{\mathbb P^3}})$. The vector fields which generate the remaining seven directions will be constructed at the end of this section.

\subsection*{Second family of tangent vector fields}

We construct here the holomorphic vector fields in order to span the $\partial/\partial z_j$-directions. For $j=1,2,3$, consider the vector field
$$
\frac{\partial}{\partial z_j}-\sum_{|\alpha+\delta_j|\le d}(\alpha_j+1)a_{\alpha+\delta_j}\,\frac{\partial}{\partial a_{\alpha}}.
$$

It is immediate to check that these vector fields, once applied to the first defining equation of $J_2\mathcal V_0$, make it identically vanish. Since the other equations of $J_2\mathcal V_0$ are obtained by deriving the first just with respect to the $z_j$ and $z_j'$ variables, they make identically vanish the other two defining equations, too. Therefore they are tangent to $J_2\mathcal V_0$. Their pole order is one in the $a_{\alpha}$'s variables and they are $\mathbb G_2$-invariant since they do not contain jet variables.

\subsection*{Third family of tangent vector fields}

In order to span the jet directions, consider a vector field of the following form:

$$
V_B=\sum_{|\alpha|\le d,\alpha_1<d}p_{\alpha}(z,a,b)\,\frac{\partial}{\partial a_{\alpha}}+\sum_{j=1}^3\sum_{k=1}^2\xi_j^{(k)}\,\frac{\partial}{\partial z_j^{(k)}},
$$
where $\xi^{(k)}=B\cdot z^{(k)}$, $k=1,2$, and $B=(b_{jk})$ varies among $3\times 3$ invertible matrices with complex entries. The additional condition on the Wronskian $z'\wedge z''\ne 0$ implies that the family $(V_B)$ spans all the $\partial/\partial z^{(k)}_j$-directions on $\Sigma_0$, as it is straightforward to see. We claim that one can choose the coefficients $p_{\alpha}(z,a,b)$ to be polynomials of degree at most $2$ in $z$ and at most one in $a$ in such a way that $V_B$ is tangent to $J_2\mathcal V_0$. To see the invariance with respect to $\mathbb G_2$, observe that the action is the following: if $\varphi\colon (\mathbb C,0)\to(\mathbb C,0)$ is a $2$-jet of biholomorphism of the origin then the action is
$$
\varphi\cdot(z,a,z',z'')\mapsto (z,a,\varphi'\cdot z',(\varphi')^2\cdot z''+\varphi''\cdot z')
$$
and the corresponding induced action on vector fields is
$$
\frac{\partial}{\partial z}\mapsto\frac{\partial}{\partial z},\quad
\frac{\partial}{\partial a}\mapsto\frac{\partial}{\partial a},\quad
\frac{\partial}{\partial z'}\mapsto\varphi'\frac{\partial}{\partial z'}+\varphi''\frac{\partial}{\partial z''},\quad
\frac{\partial}{\partial z''}\mapsto(\varphi')^2\frac{\partial}{\partial z''}.
$$
For $V_B$, only the second addendum needs to be verified to be invariant: it is of the form
$$
z'\,\frac{\partial}{\partial z'}+z''\,\frac{\partial}{\partial z''}.
$$
On the one hand, letting $\varphi$ act on coordinates, one has
$$
z'\,\frac{\partial}{\partial z'}+z''\,\frac{\partial}{\partial z''}\mapsto 
\varphi'\cdot z'\,\frac{\partial}{\partial z'}+\bigl((\varphi')^2\cdot z''+\varphi''\cdot z'\bigr)\,\frac{\partial}{\partial z''};
$$
on the other hand, letting $\varphi$ act on vector fields by its differential, one has
$$
z'\,\frac{\partial}{\partial z'}+z''\,\frac{\partial}{\partial z''}\mapsto z'\,\left(\varphi'\frac{\partial}{\partial z'}+\varphi''\frac{\partial}{\partial z''}\right)+z''\,\left((\varphi')^2\frac{\partial}{\partial z''}\right),
$$
and the invariance follows.

As the proof of the claim is not so difficult and rests just on some linear algebra, we skip it here and refer the reader to \cite{Pau08}.

\bigskip

Finally, as announced, we have to span the remaining directions in the vector space $\ker (T_{J_2\mathcal V}\to T_{J_2T_{\mathbb P^3}})$. So, consider a vector field with the following shape:
$$
\sum_{|\alpha|\le 2}v_\alpha\,\frac{\partial}{\partial a_{\alpha}}.
$$
To be tangent to $J_2\mathcal V_0$, its coefficients have to satisfy
$$
\sum_{|\alpha|\le 2}v_{\alpha}\,z^\alpha=0,
$$
$$
\sum_{|\alpha|\le 2}\sum_{j=1}^3v_\alpha\,\frac{\partial z^\alpha}{\partial z_j}\,z_j'=0
$$
and
$$
\sum_{\alpha\le 2}\left(\sum_{j=1}^3\frac{\partial z^\alpha}{\partial z_j}\,z_j''+\sum_{j,k=1}^3\frac{\partial^2 z^\alpha}{\partial z_j\partial z_k}\,z_j'z_k'\right)v_\alpha.
$$
We place ourself outside $\Sigma_0$ and we suppose for simplicity that $z_1'z_2''-z_2'z_1''\ne 0$, the other cases being analogous. Then, we can solve this system with $v_{000}$, $v_{100}$ and $v_{010}$ as unknowns:
$$
\begin{cases}
v_{000}+z_1\,v_{100}+z_2\,v_{010}=\cdots \\
\quad\qquad z'_1\,v_{100}+z'_2\,v_{010}=\cdots \\
\quad\qquad z''_1\,v_{100}+z''_2\,v_{010}=\cdots \\
\end{cases}
$$
By the Cramer rule, we see that each of these quantities are linear combinations of the $v_\alpha$'s, where $|\alpha|\le 2$, $\alpha\ne (000), (100), (010)$, with coefficients rational functions in $z,z',z''$. The denomitaor of each such coefficient is just the Wronskian $z_1'z_2''-z_2'z_1''$ and the numerator is a polynomial whose monomials have either degree at most $2$ in $z$ and at most $1$ in $z'$ and $z''$, or degree $1$ in $z$ and three in $z'$; thus, the pole order here is at most $7$. Next, the system itself is $\mathbb G_2$-invariant: letting $\varphi\in\mathbb G_2$ act on it, we find
$$
\sum_{|\alpha|\le 2}v_{\alpha}\,z^\alpha=0,
$$
$$
\varphi'\sum_{|\alpha|\le 2}\sum_{j=1}^3v_\alpha\,\frac{\partial z^\alpha}{\partial z_j}\,z_j'=0
$$
and
\begin{multline*}
(\varphi')^2\sum_{\alpha\le 2}\left(\sum_{j=1}^3\frac{\partial z^\alpha}{\partial z_j}\,z_j''+\sum_{j,k=1}^3\frac{\partial^2 z^\alpha}{\partial z_j\partial z_k}\,z_j'z_k'\right)v_\alpha \\ +\varphi''\underbrace{\sum_{\alpha\le 2}\sum_{j=1}^3v_\alpha\,\frac{\partial z^\alpha}{\partial z_j}\,z_j'}_{=0}=0.
\end{multline*}
Therefore its solutions are invariant, too. Summing up, we have proved the following

\begin{thm}[\cite{Pau08}]
The twisted tangent space
$$
T_{J_2\mathcal V}\otimes\mathcal O_{\mathbb P^3}(7)\otimes\mathcal O_{\mathbb P^{N_d-1}}(1)
$$
is generated over $J_2\mathcal V\setminus\Sigma$ by its global sections. Moreover, one can choose the generating global sections in order to be invariant with respect to the action of $\mathbb G_2$ on $J_2\mathcal V$.
\end{thm}

\section{Proof of the hyperbolicity}

In this last section we want to show the following.

\begin{thm}[\cite{McQ99}, \cite{D-EG00}, \cite{Pau08}]
Let $X\subset\mathbb P^3$ be a (very) generic smooth surface of degree $d\ge 18$. Then $X$ is Kobayashi hyperbolic.
\end{thm}

We shall in fact prove a slightly weaker form of this theorem, as far as the lower bound on the degree is concerned: the strategy of proof adopted here, which will be the one that we will use in all dimensions, will provide the worst bound $d\ge 90$ for the degree of $X$.

Let us fix once again the notations. We consider $X\subset\mathbb P^3$ a generic (or very generic) smooth surface of degree $d$. Its canonical bundle is then expressed in term of the hyperplane bundle as $K_X=\mathcal O_X(d-4)$; thus, $K_X^{\delta m}$ is the (ample) $\mathbb Q$-line bundle $\mathcal O_X(\delta m(d-4))$. The Chern classes of $X$ are given by
$$
c_1(X)=(4-d)\,h,\quad c_2(X)=(d^2-4d+6)\,h^2,
$$
so that the quantity $(54\delta^2-48\delta+13)\,c_1(X)^2-9\,c_2(X)$ considered above is equal to
\begin{equation}\label{delta}
(54\delta^2 - 48\delta + 4)\,d^3 +(-432\delta^2 + 384\delta - 68)\,d^2  + (864\delta^2 - 768\delta + 154)\, d.
\end{equation}
Notice that if $0\le\delta<1/3$ and $54\delta^2 - 48\delta + 4>0$ then by Proposition \ref{existencesurfaces}, for $m\gg d\gg 1$ we have a non zero global section of $E_{2,m}T^*_X\otimes K_X^{-\delta m}$ (compare with hypothesis (i) in Theorem \ref{general}).

Now, consider the universal hypersurface $\mathcal X\subset \mathbb P^3\times\mathbb P^{N_d-1}$ of degree $d$ in $\mathbb P^3$ and the holomorphic subbundle $\mathcal V\subset T_{\mathcal X}$ given by the differential of the kernel of the second projection. By the results of the previous section, we know that
$$
T_{J_2\mathcal V}\otimes\mathcal O_{\mathbb P^3}(7)\otimes\mathcal O_{\mathbb P^{N_d-1}}(1)
$$
is globally generated by its global holomorphic sections over $J_2\mathcal V\setminus\Sigma$ and moreover the generating sections can be chosen to be invariant by the action of $\mathbb G_2$ on $J_2\mathcal V$ (compare with hypothesis (ii) of Theorem \ref{general}, the bundle $\mathcal O_{\mathbb P^3}(7)$ here plays the role of the bundle $B$ there).

Concerning hypothesis (iii) of Theorem \ref{general}, we want $K_X^{\delta m}\otimes\mathcal O_{\mathbb P^3}(7)^{-\otimes m}$ to be ample: this is the case if 
$$
\delta>\frac 7{d-4},
$$ 
so $\delta$ will be chosen a little bit larger than $7/(d-4)$.

Start with a non zero section $P\in H^0(X,E_{2,m}T^*_X\otimes K_X^{-\delta m})$.
Call 
$$
Y=\{P=0\}\subsetneq X
$$ 
the base locus of such a non zero section. 

If $s\in\mathbb P(H^0(\mathbb P^{3},\mathcal O_{\mathbb P^3}(d)))$ parametrizes any smooth hypersurface $X_s$, then one has
$$
H^0(X_s,E_{2,m}T^*_{X_s}\otimes K_{X_s}^{-\delta m})\simeq H^0(X_s,E_{2,m}\mathcal V^*\otimes\mathcal O_{\mathbb P^3}(-\delta m(d-4))|_{X_s}).
$$
Next, suppose $X=X_0$ corresponds to some point $0$ in the parameter space $\mathbb P(H^0(\mathbb P^{3},\mathcal O_{\mathbb P^3}(d)))$. Since we have chosen $X$ to be generic, standard semicontinuity arguments show that there exists an open neighborhood $U\ni 0$ such that the restriction morphism
$$
H^0(\text{pr}_2^{-1}(U),E_{2,m}\mathcal V^*\otimes\mathcal O_{\mathbb P^3}(-\delta m(d-4)))\to H^0(X_0,E_{2,m}T^*_{X_0}\otimes K_{X_0}^{-\delta m})
$$
is surjective. Therefore the \lq\lq first\rq\rq{} jet differential $P$ may be extended to a neighborhood of the starting hypersurface.

Now, suppose we have a holomorphic entire curve $f\colon\mathbb C\to X$ and consider its lifting $j_2(f)\colon\mathbb C\to J_2T_X\subset J_2\mathcal V$. If $f(\mathbb C)\subset\Sigma$ then we have the following.

\begin{lem}
Let $f\colon\mathbb C\to\mathbb C^{N}$ be a holomorphic map. If $f'\wedge f''\wedge\cdots\wedge f^{(k)}\equiv 0$, then $f(\mathbb C)$ lies inside a codimension $N-k+1$ affine linear subspace. 
\end{lem}

\begin{proof}
Without loss of generality, we can suppose $k>1$, $f'\wedge f''\wedge\cdots\wedge f^{(k-1)}\not\equiv 0$, $f'(0)\ne 0$ and $(f'\wedge f''\wedge\cdots\wedge f^{(k-1)})(0)\ne 0$. Then there exists an open neighborhood $\Omega\subset\mathbb C$ of $0$ such that for each $t\in\Omega$ we have a linear combination
$$
f^{(k)}(t)=\sum_{j=1}^{k-1}\lambda_{j}(t)\,f^{(j)}(t)
$$
and the $\lambda_j$'s depend holomorphically on $t$. By taking derivatives, one sees inductively that, in $\Omega$, every $f^{(\ell)}$, $\ell\ge k$, is a linear combination of the $f^{(j)}$'s, $1\le j\le k-1$. Thus, all the derivatives in $0$ of $f$ lie in the linear space generated by $f'(0),\dots,f^{(k-1)}(0)$. The conclusion follows by expanding $f$ in power series at $0$.
\end{proof}

In fact, this lemma shows that the image of the entire curve lies in a codimension two subvariety of $X$ (the intersection of $X$ with a codimension two linear subspace of $\mathbb P^3$), provided $X$ is generic. Therefore $f$ is constant in this case.

Thus, we can suppose $j_2(f)\not\subset \Sigma$. Then, if $f(\mathbb C)\not\subset Y$ the proof proceed exactly as the one of Theorem \ref{general}: we have extended on an open set the global jet differential $P$ and we can now take derivatives with meromorphic vector fields in $T_{J_2\mathcal V}$. After having taken at most $m$ derivatives, we restrict again everything to $X_0$ and we are done.

Observe now that if we take $X$ very generic, by the Noether-Lefschetz theorem, the Picard group of $X$ is infinite cyclic. Thus, we can apply in full strength the statement of Theorem \ref{general} and obtain that $Y$ has codimension two in $X$. But then $Y$ has dimension zero and $f$ is constant: this proves hyperbolicity.

We now come to the effective part concerning the degree of $X$. If we plainly substitute in (\ref{delta}) $\delta=7/(d-4)$, we obtain 
$$
4\,d^3 - 404\,d^2 + 4144\,d,
$$
whose larger root is a little smaller than $90$. Then, by continuity, if $\delta$ is chosen to be a little larger than $7/(d-4)$ as announced, everything fit for $d\ge 90$: we have in fact that $7/(d-4)<\delta < 1/3$ and (\ref{delta}) is positive.

Thus, very generic smooth projective surfaces in $\mathbb P^3$ of degree greater than or equal to $90$ are hyperbolic.

Finally, we remark here that in order to go down till degree $18$, a more complicated combination of this variational method together with deep result by Mc Quillan \cite{McQ98} on parabolic leaves of algebraic (multi)foliations on surfaces of general type is required, but we shall skip here this part (see \cite{Pau08} for more details).

In the next chapter, we shall treat the general case in arbitrary dimension: the strategy will be exactly the same as here (anyway, we would not be able to invoke anymore the work of Mc Quillan which unfortunately for the moment is available in dimension two only). A major difficulty will be the step one, that is to find the first jet differential; this will be overcome by means of the algebraic version of Demailly's holomorphic Morse inequalities. Then, a generalization of the global generation statement obtained in \cite{Mer09} (compare also with \cite{Siu04}) will permit us to let the strategy work in full generality.

\chapter[Algebraic degeneracy for projective hypersurfaces]{Algebraic degeneracy for generic projective hypersurfaces}

{\small\textsc{Abstract}. This chapter will treat the general case of projective hypersurfaces in every dimension. We will prove an algebraic degeneracy result for entire curve in generic projective hypersurfaces of high degree, taken from \cite{D-M-R10}. The first part will be concerned in finding jet differentials, as in \cite{Div09}, then we shall cite the general result on meromorphic vector fields contained in \cite{Mer09}. Finally, we shall discuss some effective aspects of the proof.}

\section{Statement of the result and scheme of proof}

The aim of this chapter is to give the proof of (the tools needed to prove) the following.

\begin{thm}[\cite{D-M-R10}, \cite{D-T10}]\label{mainthm}
Let $X\subset\mathbb P^{n+1}$ be a generic smooth projective hypersurface of arbitrary dimension $n \ge 2$. 
If the degree of $X$ satisfies the \emph{effective} lower bound:
$$
\deg (X)\ge 2^{n^5},
$$
then there exists a proper, of codimension at least two, closed subvariety $Y\subsetneq X$ such that every entire non constant holomorphic curve $f \colon \mathbb C \to X$ has its image contained in $Y$.
\end{thm}

In small dimensions, better bounds can be obtained.

\begin{thm}[\cite{D-M-R10}, \cite{D-T10}]\label{mainthmlow}
Let $X\subset\mathbb P^{n+1}$ be a (very) generic smooth projective hypersurface. 
If the degree of $X$ satisfies the \emph{effective} bounds:
\begin{itemize}
\item for $n=3$, $\deg X \geqslant 593${\rm ;} 
\item for $n=4$, $\deg X \geqslant 3203${\rm ;}
\item for $n=5$, $\deg X \geqslant 35355${\rm ;}
\item for $n=6$, $\deg X \geqslant 172925$,
\end{itemize}
then there exists a proper, codimension two, closed subvariety $Y\subsetneq X$ such that every entire non constant holomorphic curve $f \colon \mathbb C \to X$ has its image contained in $Y$.
\end{thm}

Before going on, we state the following

\begin{cor}[\cite{D-T10}]\label{kob3}
Let $X\subset\mathbb P^4$ be a very generic smooth projective hypersurface of degree $d$. Then $X$ is Kobayashi hyperbolic, provided $d\ge 593$.
\end{cor}

\begin{proof}
The Zariski closure of the image of a nonconstant entire curve, if any, must be an algebraic curve in $X$. Then, such an algebraic curve must be rational or elliptic. But, as we have seen, this contradicts the following classical result by Clemens \cite{Cle86}: \emph{Let $X\subset\mathbb P^{n+1}$ be a smooth very generic hypersurface. Then $X$ contains no rational curves (resp. elliptic curves) provided $\deg X\ge 2n$ (resp. $2n+1$).}

The considerably better lower bound $593\ll 2^{3^5}$ is reached by using the knowledge of the full composition series of $E_{3,m}T^*_X$ in dimension $3$ and the global generation of meromorphic vector fields on the total space of vertical jets outside the wronskian locus, obtained respectively in \cite{Rou06a} and \cite{Rou07}, see \cite{D-T10} for details.

Note that the slightly weaker result of algebraic degeneracy of entire curves in the same setting was already proved in \cite{Rou07}. 
\end{proof}

The scheme of the proof of Theorem \ref{mainthm} is, as in the two dimensional case, the following.
Start with a nonzero section $P\in H^0(X,E_{n,m}T^*_X\otimes K_X^{-\delta m})$, for some $m\gg 0$ and $0<\delta\ll 1$, where $X\subset\mathbb P^{n+1}$ is a smooth generic projective hypersurface of degree $d$ large enough (in order to have such a section).
Call 
$$
Y=\{P=0\}\subsetneq X
$$ 
the base locus of such a nonzero section. Look at $P$ as an invariant (under the action of the group $\mathbb G_n$ of $n$-jets of biholomorphic changes of parameter of $(\mathbb C,0)$) map 
$$
J_nT_X\to p^* K_X^{-\delta m}
$$
where $p\colon J_n T_X\to X$ is the space of $n$-jets of germs of holomorphic curves $f\colon (\mathbb C,0)\to X$. Then $P$ is a weighted homogeneous polynomial in the jet variables of degree $m$ with coefficients holomorphic functions of the coordinates of $X$ and values in $p^*K_X^{-\delta m}$.

Suppose for a moment that we have enough global holomorphic $\mathbb G_n$-invariant vector fields on $J_n T_X$ with values in the pull-back from X of some ample divisor in order to generate $T_{J_n T_X}\otimes p^*\mathcal O_X(\ell)$, at least over the dense open set $J_nT_X^{\text{\rm reg}}$ of regular $n$-jets, \emph{i.e.} of $n$-jets with nonvanishing first derivative. 

If $f\colon\mathbb C\to X$ is an entire curve, consider its lifting $j_n(f)\colon\mathbb C\to J_n T_X$ and suppose that $j_n(f)(\mathbb C)\not\subset J_nT_X^{\text{\rm sing}}\overset{\text{\rm def}}=J_n T_X\setminus J_nT_X^{\text{\rm reg}}$ (otherwise $f$ is constant). Arguing by contradiction, let $f(\mathbb C)\not\subset Y$ and $x_0=f(t_0)\in X\setminus Y$.
Thus, one can produce, by differentiating at most $m$ times, a new invariant $n$-jet differential $Q$ of weighted degree $m$ with values in 
$$
K_X^{-\delta m}\otimes\mathcal O_X(m\ell)\simeq\mathcal O_X(-\delta m(d-n-2)+m\ell)
$$ 
such that $Q(j_n(f)(t_0))\ne 0$, thus contradicting Corollary \ref{baselocus}, provided $\delta>\ell/(d-n-2)$, \emph{i.e.} provided $Q$ is still with value in an antiample divisor (this last condition is clearly achieved by letting the degree $d$ of $X$ grow sufficiently).

Unfortunately, in this setting probably we can't hope for such a global generation statement for meromorphic vector fields of $J_nT_X$ to hold. Thus, as in \cite{Siu04,Pau08,Rou07}, one has to use \lq\lq slanted vector fields\rq\rq{} in order to gain some positivity.

Consider the universal hypersurface 
$$
\mathcal X\subset\mathbb P^{n+1}\times\mathbb P(H^0(\mathbb P^{n+1},\mathcal O(d)))
$$ 
of degree $d$ in $\mathbb P^{n+1}$. Next, consider the subbundle $\mathcal V\subset T_{\mathcal X}$ given by the kernel of the differential of the second projection. If $s\in\mathbb P(H^0(\mathbb P^{n+1},\mathcal O(d)))$ parametrizes any smooth hypersurface $X_s$, then one has
$$
H^0(X_s,E_{n,m}T^*_{X_s}\otimes K_{X_s}^{-\delta m})\simeq H^0(X_s,E_{n,m}\mathcal V^*\otimes\text{pr}_1^*K_{X_s}^{-\delta m}|_{X_s}).
$$
Suppose $X=X_0$ corresponds to some parameter $0$ in the moduli space $\mathbb P(H^0(\mathbb P^{n+1},\mathcal O(d)))$. Since we have chosen $X$ to be generic, standard semicontinuity arguments show that there exists an open neighborhood $U\ni 0$ such that the restriction morphism
\begin{multline*}
H^0(\text{pr}_2^{-1}(U),E_{n,m}\mathcal V^*\otimes\text{\rm pr}_1^*\mathcal O_{\mathbb P^{n+1}}(-\delta m(d-n-2))) \\ \to H^0(X_0,E_{n,m}T^*_{X_0}\otimes K_{X_0}^{-\delta m})
\end{multline*}
is surjective. Therefore the \lq\lq first\rq\rq{} jet differential may be extended to a neighborhood of the starting hypersurface, and one can use the following global generation statement (we shall come back on this result later on).

\begin{thm}[\cite{Mer09}, compare also with \cite{Siu04}]\label{gg}
The twisted tangent bundle 
$$
T_{J_n\mathcal V}\otimes\text{\rm pr}_1^*\mathcal O(n^2+2n)\otimes\text{\rm pr}_2^*\mathcal O(1)
$$
is generated over $J_n\mathcal V^{\text{\rm reg}}$ by its global sections. Moreover, one can choose such generating global sections to be invariant under the action of $\mathbb G_n$ on $J_n\mathcal V$. 
\end{thm}
Thus, by replacing $\ell$ by $n^2+2n$ in our previous discussion and by removing the hyperplane which corresponds to the poles given by $\text{\rm pr}_2^*\mathcal O(1)$ in the parameter space, one gets the desired result of algebraic degeneracy.

Observe that no information is known about the multiplicity of the subvariety $Y$, thus we are not able to bound \emph{a priori} the number of derivative needed in order to reduce the vanishing locus of the first jet differential.

Finally, to obtain codimension two, we just remark that if $X$ is a very generic surface in $\mathbb P^3$ or any smooth projective hypersurface of dimension at least three, then its Picard group is infinite cyclic, so that the final statement of Theorem \ref{general} applies (the contribution of \cite{D-T10} in Theorem \ref{mainthm} and \ref{mainthmlow} is this qualitative part on the codimension of the subvariety $Y$).

At this point it should be clear that in order to prove Theorem \ref{mainthm} we have to 

\begin{itemize}
\item[(i)] construct a global invariant jet differential of order $n$ and degree $m$ with values in $K_X^{-\delta m}$ for $X$ a smooth projective hypersurface of high degree, 
\item[(ii)] globally generate the tangent space of vertical $n$ jets with meromorphic vector fields of controlled pole order,
\item[(iii)] estimate the minimal degree of the hypersurface which makes the machinery work.
\end{itemize}

This will be the content of the next sections.

\section{Existence of jet differentials}

In this section, we want to solve the above point (i). For surfaces, we have used the vanishing of the cohomology group $H^2(X,E_{2,m}T_X^*)$. In higher dimensions, one major difficulty is that we do not have the vanishing of the higher cohomology groups anymore. A useful tool to control the cohomology are the {\it holomorphic Morse inequalities} by Demailly.

\subsection{Algebraic holomorphic Morse inequalities}

Let $L\to X$ be a holomorphic line bundle over a compact K\"ahler manifold of dimension $n$ and $E\to X$ a holomorphic vector bundle of rank $r$. Suppose that $L$ can be written as the difference of two nef line bundles, say $L=F\otimes G^{-1}$, with $F,G\to X$ numerically effective. Then we have the following asymptotic estimate for the dimension of cohomology groups of powers of $L$ with values in $E$.

\begin{thm}[\cite{Dem01}]
With the previous notation, we have 
\begin{itemize}
\item Weak algebraic holomorphic Morse inequalities:
$$
h^q(X,L^{\otimes m}\otimes E)\le r\frac{m^n}{(m-q)!q!} F^{n-q}\cdot G^q+ o(m^n).
$$
\item Strong algebraic holomorphic Morse inequalities:
\begin{multline*}
\sum_{ j=0}^q(-1)^{q-j}h^j(X,L^{\otimes m}\otimes E) \\ \le r\frac{m^n}{n!}\sum_{j=0}^q(-1)^{q-j}\binom nj F^{n-j}\cdot G^j+o(m^n).
\end{multline*}
In particular {\rm \cite{Tra95}} (see also \cite{Siu93}), $L^{\otimes m}\otimes E$ has a global section for $m$ large as soon as $F^n-n\,F^{n-1}\cdot G>0$.
\end{itemize}
\end{thm}

Then we have two possible strategies to obtain global jet differentials. The first one consists in computing the Euler characteristic $\chi(X,E_{k,m}T_X^*)$ and then in finding upper bounds for the higher even cohomology groups $H^{2i}(X,E_{k,m}T_X^*)$ using the weak Morse inequalities. The drawback with this method is that to do the Riemann-Roch computation, one needs the algebraic characterization of $E_{k,m}T_X^*$ which is not available in general as we have seen. We shall describe this method in dimension $3$ below.
\begin{thm}[\cite{Rou06b}]
Let $X\subset\mathbb P^{4}$ a smooth projective hypersurface of degree $d$ and $A\to X$ any ample line bundle. Then
$$
H^0(X,E_{3,m}T_X^*\otimes A^{-1})\ne 0
$$
provided $d\geq 97$.
\end{thm}

The second method consists in using the strong Morse inequalities for $q=1$ to $L=\mathcal{O}_{X_k}(\textbf{a}).$ Using this, we shall prove below the following.

\begin{thm}[\cite{Div09}]\label{existence}
Let $X\subset\mathbb P^{n+1}$ a smooth projective hypersurface of degree $d$ and $A\to X$ any ample line bundle. Then
$$
H^0(X,E_{n,m}T_X^*\otimes A^{-1})\ne 0
$$
provided $m\gg d\gg 1$.
\end{thm}

\section[Existence of jet differentials in dimension $3$]{Proof of the existence of jet differentials in dimension $3$}
Let $X\subset \mathbb{P}^4$ be a sooth hypersurface of degree $d$. Recall that we have
$$
\operatorname{Gr}^\bullet(E_{3,m}T_X^*)=\bigoplus_{a+3b+5c+6d=m}\Gamma^{
(a+b+2c+d,\,\,b+c+d,\,\,d)}T_X^{\ast }.
$$
Thus, we can do a Riemann-Roch computation which provides
$$\chi(X,E_{3,m}T_X^*)=m^9P(d)+O(m^8)$$
where $P$ is an explicit polynomial of degree $4$ such that $P(d)>0$ for $d\geq 43.$

Now, to control the cohomology of the vector bundle of jet differentials using weak Morse inequalities, we have to reduce the problem to the control of the cohomology of a {\it line} bundle. This can be achieved by working on flag manifolds as follows. We denote by $\pi \colon Fl(T_X^*) \to X$ the flag manifold of $T_X^*$ i.e. the bundle whose fibers $Fl(T_{X,x}^*)$ consists of sequences of vector spaces
$$\mathcal{D}:=\{0= E_3 \subset E_2 \subset E_1 \subset E_0:=T_{X,x}^*\}.$$
Let $\lambda=(\lambda_1,\lambda_2,\lambda_3)$ be a partition such that $\lambda_1>\lambda_2>\lambda_3$. Then we can define a line bundle $\mathcal{L}^\lambda:=\mathcal{L}^\lambda(T_X^*)$ over $Fl(T_X^*)$ such that its fiber over $\mathcal{D}$ is $$\mathcal{L}^\lambda_\mathcal{D}:=\bigotimes_{i=1}^{3} \det (E_{i-1}/E_i)^{\otimes \lambda_i}.$$
Then by a classical theorem of Bott, if $m\geq 0$,
$$\pi_*(\mathcal{L}^\lambda)^{\otimes m}=\mathcal O(\Gamma^{m\lambda} T_X^*),$$
$$R^q\pi_*(\mathcal{L}^\lambda)^{\otimes m}=0 \text{ if } q>0.$$
Thus $\Gamma^{m\lambda} T_X^*$ and $(\mathcal{L}^\lambda)^{\otimes m}$ have the same cohomology. We can use the weak holomorphic Morse inequalities on the flag manifold to obtain
$$h^2(X,E_{3,m}T_X^*) \leq C d(d+13)m^9+O(m^8).$$
The key point is to write in our situation $\mathcal{L}^\lambda$ as the difference of two nef line bundles. This is done as follows. It is well known that the cotangent space of the projective space twisted by $\mathcal O(2)$ is globally generated. Hence, $T^*_X\otimes\mathcal O_X(2)$ is globally generated as a quotient of $T^*_{\mathbb P^{4}}|_X\otimes\mathcal O_X(2)$. Therefore
$$\mathcal{L}^\lambda(T^*_X\otimes\mathcal O_X(2))\cong \mathcal{L}^\lambda \otimes \pi^*(\mathcal{O}_X(2|\lambda|)=:F$$ is nef, and we can write
$$\mathcal{L}^\lambda=F\otimes G^{-1},$$
as a difference of two nef line bundles where $G:= \pi^*(\mathcal{O}_X(2|\lambda|).$
Obviously we have
$$h^0(X,E_{3,m}T_X^*)\geq \chi(X,E_{3,m}T_X^*)-h^2(X,E_{3,m}T_X^*),$$
and we obtain that
$$
H^0(X,E_{3,m}T_X^*\otimes A^{-1})\ne 0,
$$
provided $d\geq 97$, where $A$ is an ample line bundle.
As a corollary, any entire curve in such an hypersurface must satisfy the corresponding differential equation.

As we have noted above, the previous computations rely on the knowledge of the decomposition in irreducible representations of the bundle $E_{k,m}T_X^*$ and therefore, for the moment, cannot be generalized to dimension $5$ and more. Nevertheless, in the recent work \cite{Mer10}, Merker has been able to apply the strategy described above to the bundle of Green-Griffths jet differentials and obtained the existence of jet differentials on smooth hypersurfaces $X\subset \mathbb{P}^{n+1}$ of general type, i.e. of degree $\deg(X) \geq n+3$ for any $n$.

\section[Existence of jet differentials in higher dimensions]{Proof of the existence of jet differentials in higher dimensions}

The idea of the proof is to apply the algebraic holomorphic Morse inequalities to a particular relatively nef line bundle over $X_n$ which admits a nontrivial morphism to (a power of) $\mathcal O_{X_n}(1)$ and then to conclude by the direct image argument of Theorem \ref{di}, and Proposition \ref{relnef}. 

From now on, we will set in the \lq\lq absolute\rq\rq{} case $V=T_X$. 

\subsection{Choice of the appropriate subbundle}

Recall that, for $\mathbf a=(a_1,\dots,a_k)\in\mathbb Z^k$, we have defined a line bundle $\mathcal O_{X_k}(\mathbf a)$ on $X_k$ as
$$
\mathcal O_{X_k}(\mathbf a)=\pi_{1,k}^*\mathcal O_{X_1}(a_1)\otimes\pi_{2,k}^*\mathcal O_{X_2}(a_2)\otimes\cdots\otimes\mathcal O_{X_k}(a_k)
$$
and an associated weight $\mathbf b=(b_1,\dots,b_k)\in\mathbb Z^k$ such that $b_j=a_1+\cdots+a_j$, $j=1,\dots,k$. Moreover, if $\mathbf b\in\mathbb N^k$, that is if $a_1+\cdots+a_j\ge 0$, we had a nontrivial morphism
$$
\mathcal O_{X_k}(\mathbf a)=\mathcal O_{X_k}(b_k)\otimes\mathcal O_{X_k}(-\mathbf b\cdot D^\star)\to\mathcal O_{X_k}(b_k)
$$
and, if 
\begin{equation}\label{c}
a_1\ge 3a_2,\dots,a_{k-2}\ge 3a_{k-1}\quad\text{and}\quad a_{k-1}\ge 2a_k> 0,
\end{equation}
then $\mathcal O_{X_k}(\mathbf a)$ is relatively nef over $X$.

Now, let $X\subset\mathbb P^{n+1}$ be a smooth complex projective hypersurface. Then it is always possible to express $\mathcal O_{X_k}(\mathbf a)$ as the difference of two globally nef line bundles, provided condition (\ref{c}) is satisfied. 

\begin{lem}\label{globnef}
Let $X\subset\mathbb P^{n+1}$ be a projective hypersurface. Set $\mathcal L_k=\mathcal O_{X_k}(2\cdot 3^{k-2},\dots,6,2,1)$. Then $\mathcal L_k\otimes\pi_{0,k}^*\mathcal O_X(\ell)$ is nef if $\ell\ge 2\cdot 3^{k-1}$. In particular, 
$$
\mathcal L_k=\mathcal F_k\otimes\mathcal G_k^{-1},
$$
where $\mathcal F_k\overset{\text{\rm def}}=\mathcal L_k\otimes\pi_{0,k}^*\mathcal O_X(2\cdot 3^{k-1})$ and $\mathcal G_k\overset{\text{\rm def}}=\pi_{0,k}^*\mathcal O_X(2\cdot 3^{k-1})$ are nef. 
\end{lem}

\begin{proof}
Of course, as a pull-back of an ample line bundle, 
$$
\mathcal G_k=\pi_{0,k}^*\mathcal O_X(2\cdot 3^{k-1})
$$ 
is nef. As we have seen, $T^*_X\otimes\mathcal O_X(2)$ is globally generated as a quotient of $T^*_{\mathbb P^{n+1}}|_X\otimes\mathcal O_X(2)$, so that $\mathcal O_{X_1}(1)\otimes\pi_{0,1}^*\mathcal O_X(2)=\mathcal O_{\mathbb P(T^*_X\otimes\mathcal O_X(2))}(1)$ is nef. 

Next, we construct by induction on $k$, a nef line bundle $A_{k}\to X_{k}$ such that $\mathcal O_{X_{k+1}}(1)\otimes\pi_k^*A_{k}$ is nef. By definition, this is equivalent to say that the vector bundle $V_{k}^*\otimes A_{k}$ is nef. By what we have just seen, we can take $A_0=\mathcal O_X(2)$ on $X_0=X$. Suppose $A_0,\dots,A_{k-1}$ as been constructed. As an extension of nef vector bundles is nef, dualizing the short exact sequence (\ref{ses1}) we find
$$
0\longrightarrow\mathcal O_{X_k}(1)\longrightarrow V_k^*\longrightarrow T_{X_k/X_{k-1}}^*\longrightarrow 0,
$$ 
and so we see, twisting by $A_k$, that it suffices to select $A_k$ in such a way that both $\mathcal O_{X_k}(1)\otimes A_k$ and $T_{X_k/X_{k-1}}^*\otimes A_k$ are nef. To this aim, considering the second wedge power of the central term in (\ref{ses2}), we get an injection
$$
0\longrightarrow T_{X_k/X_{k-1}}\longrightarrow\bigwedge{}\!\!^2(\pi_k^* V_{k-1}\otimes\mathcal O_{X_k}(1))
$$
and so dualizing and twisting by $\mathcal O_{X_{k}}(2)\otimes\pi_k^* A_{k-1}^{\otimes 2}$, we find a surjection
$$
\pi_k^*\bigwedge{}\!\!^2(V_{k-1}^*\otimes A_{k-1})\longrightarrow T^*_{X_k/X_{k-1}}\otimes\mathcal O_{X_k}(2)\otimes\pi_k^* A_{k-1}^{\otimes 2}\longrightarrow 0.
$$
By induction hypothesis, $V_{k-1}^*\otimes A_{k-1}$ is nef so the quotient $T^*_{X_k/X_{k-1}}\otimes\mathcal O_{X_k}(2)\otimes\pi_k^* A_{k-1}^{\otimes 2}$ is nef, too.
In order to have the nefness of both $\mathcal O_{X_k}(1)\otimes A_k$ and $T_{X_k/X_{k-1}}^*\otimes A_k$, it is enough to select $A_k$ in such a way that $A_k\otimes\pi_k^* A_{k-1}^*$ and $A_k\otimes\mathcal O_{X_k}(-2)\otimes\pi_k^*{A_{k-1}^*}^{\otimes 2}$ are both nef: therefore we set 
$$
A_k=\mathcal O_{X_k}(2)\otimes\pi_k^* A_{k-1}^{\otimes 3}=\bigl (\mathcal O_{X_k}(1)\otimes\pi_k^* A_{k-1}\bigr)^{\otimes 2}\otimes\pi_k^* A_{k-1},
$$
which, as a product of nef line bundles, is nef and satisfies the two conditions above.
This gives $A_k$ inductively, and the resulting formula for $\mathcal O_{X_k}(1)\otimes\pi_k^*A_{k-1}$ is
$$
\begin{aligned}
\mathcal O_{X_k}(1)\otimes\pi_k^*A_{k-1} &= \mathcal L_{k}\otimes\pi_{0,k}^*\mathcal O_X(2\cdot(1+2+\cdots+2\cdot 3^{k-2})) \\
& =   \mathcal L_{k}\otimes\pi_{0,k}^*\mathcal O_X(2\cdot 3^{k-1}).
\end{aligned}
$$
The lemma is proved.
\end{proof}

We now use the above lemma to deal with general weights satisfying condition (\ref{c}).

\begin{prop}
Let $X\subset\mathbb P^{n+1}$ be a smooth projective hypersurface and $\mathcal O_X(1)$ be the hyperplane divisor on $X$. If condition (\ref{c}) holds, then $\mathcal O_{X_k}(\mathbf a)\otimes\pi_{0,k}^*\mathcal O_X(\ell)$ is nef provided that $\ell\ge 2|\mathbf a|$, where $|\mathbf a|=a_1+\cdots +a_k$.

In particular $\mathcal O_{X_k}(\mathbf a)=\bigl(\mathcal O_{X_k}(\mathbf a)\otimes\pi_{0,k}^*\mathcal O_X(2|\mathbf a|)\bigr)\otimes\pi_{0,k}^*\mathcal O_X(-2|\mathbf a|)$ and both $\mathcal O_{X_k}(\mathbf a)\otimes\pi_{0,k}^*\mathcal O_X(2|\mathbf a|)$ and $\pi_{0,k}^*\mathcal O_X(2|\mathbf a|)$ are nef.
\end{prop}

\begin{proof}
By Lemma \ref{globnef}, we know that the line bundle 
$$
\mathcal O_{X_k}(2\cdot 3^{k-2},2\cdot 3^{k-3},\dots,6,2,1)\otimes\pi_{0,k}^*\mathcal O_X(\ell)
$$ 
is nef as soon as $\ell\ge 2\cdot(1+2+6+\cdots+2\cdot 3^{k-2})=2\cdot 3^{k-1}$. Now we take $\mathbf a=(a_1,\dots,a_k)\in\mathbb N^{k}$ such that $a_1\ge 3a_2,\dots,a_{k-2}\ge 3a_{k-1},a_{k-1}\ge 2a_k>0$ and we proceed by induction, the case $k=1$ being obvious. Write
$$
\begin{aligned}
\mathcal O_{X_k}&(a_1,a_2,\dots,a_k)\otimes\pi_{0,k}^*\mathcal O_X(2\cdot(a_1+\cdots+a_k)) \\
&=\bigr(\mathcal O_{X_k}(2\cdot 3^{k-2},\dots,6,2,1)\otimes\pi_{0,k}^*\mathcal O_X(2\cdot 3^{k-1})\bigl)^{\otimes a_k} \\
&\quad\otimes\pi_{k}^*\biggl(\mathcal O_{X_{k-1}}(a_1-2\cdot 3^{k-2}a_k,\dots,a_{k-2}-6a_k,a_{k-1}-2a_k) \\
&\quad\otimes\pi_{0,k-1}^*\mathcal O_X\bigl(2\cdot(a_1+\dots+a_k-3^{k-1}a_k)\bigr)\biggr).
\end{aligned}
$$
Therefore, we have to prove that 
$$
\begin{aligned}
&\mathcal O_{X_{k-1}}(a_1-2\cdot 3^{k-2}a_k,\dots,a_{k-2}-6a_k,a_{k-1}-2a_k) \\
&\quad\quad\otimes\pi_{0,k-1}^*\mathcal O_X\bigl(2\cdot(a_1+\dots+a_k-3^{k-1}a_k)\bigr)
\end{aligned}
$$
is nef.
Our chain of inequalities gives, for $1\le j\le k-2$, $a_j\ge 3^{k-j-1}a_k$ and $a_{k-1}\ge 2a_k$. Thus, condition (\ref{c}) is satisfied by the weights of
$$
\mathcal O_{X_{k-1}}(a_1-2\cdot 3^{k-2}a_k,\dots,a_{k-2}-6a_k,a_{k-1}-2a_k)
$$
and $2\cdot(a_1+\dots+a_k-3^{k-1}a_k)$ is exactly twice the sum of these weights. 
\end{proof}

\begin{rem}\label{rem1}
At this point it should be clear that  to prove Theorem \ref{existence} is sufficient to show the existence of an $n$-tuple $(a_1,\dots,a_n)$ satisfying condition (\ref{c}) and such that
\begin{equation}\label{*}
\begin{aligned}
\bigl(\mathcal O_{X_n}&(\mathbf a)\otimes\pi_{0,n}^*\mathcal O_X(2|\mathbf a|)\bigr)^{n^2} \\
&-n^2\bigl(\mathcal O_{X_n}(\mathbf a)\otimes\pi_{0,n}^*\mathcal O_X(2|\mathbf a|)\bigr)^{n^2-1}\cdot\pi_{0,n}^*\mathcal O_X(2|\mathbf a|)>0
\end{aligned}
\end{equation}
for $d=\deg X$ large enough, where $n^2=n+n(n-1)=\dim X_n$.

In fact, this would show the bigness of $\mathcal O_{X_n}(\mathbf a)\hookrightarrow\mathcal O_{X_n}(|\mathbf a|)$ and so the bigness of $\mathcal O_{X_n}(1)$.
\end{rem}

Now, we will explain how to compute this intersection number using the inductive structure of Demailly's tower.
\subsection{Cohomology ring of $X_k$}

Denote by $c_\bullet(E)$ the total Chern class of a vector bundle $E$. The short exact sequences (\ref{ses1}) and (\ref{ses2}) give us, for each $k>0$, the following formulae:
$$
c_\bullet(V_k)=c_\bullet(T_{X_k/X_{k-1}}) c_\bullet(\mathcal O_{X_k}(-1))
$$
and
$$
c_\bullet(\pi_k^*V_{k-1}\otimes\mathcal O_{X_k}(1))=c_\bullet(T_{X_k/X_{k-1}}),
$$
so that
\begin{equation}\label{chern1}
c_\bullet(V_k)=c_\bullet(\mathcal O_{X_k}(-1))c_\bullet(\pi_k^*V_{k-1}\otimes\mathcal O_{X_k}(1)).
\end{equation}
Let us call $u_j=c_1(\mathcal O_{X_j}(1))$ and $c_l^{[j]}=c_l(V_j)$. With these notations, (\ref{chern1}) becomes
\begin{equation}\label{chern2}
c_l^{[k]}=\sum_{s=0}^l\left[\binom{n-s}{l-s}-\binom{n-s}{l-s-1}\right]u_k^{l-s}\cdot\pi_k^*c_s^{[k-1]},\quad 1\le l\le r.
\end{equation}
Since $X_j$ is the projectivized bundle of line of $V_{j-1}$, we also have the polynomial relations
\begin{equation}\label{chern3}
u_j^r+\pi_j^*c_1^{[j-1]}\cdot u_j^{r-1}+\cdots+\pi_j^*c_{r-1}^{[j-1]}\cdot u_j+\pi_j^*c_{r}^{[j-1]}=0,\quad 1\le j\le k.
\end{equation}
After all, the cohomology ring of $X_k$ is defined in terms of generators and relations as the polynomial algebra $H^\bullet(X)[u_1,\dots,u_k]$ with the relations (\ref{chern3}) in which, utilizing recursively (\ref{chern2}), we have that $c_l^{[j]}$ is a polynomial with integral coefficients in the variables $u_1,\dots,u_j,c_1(V),\dots,c_l(V)$.

In particular, for the first Chern class of $V_k$, we obtain the very simple expression
\begin{equation}\label{c1}
c_1^{[k]}=\pi_{0,k}^*c_1(V)+(r-1)\sum_{s=1}^k\pi_{s,k}^*u_s.
\end{equation}

\subsection{Evaluation in terms of the degree}

For $X\subset\mathbb P^{n+1}$ a smooth projective hypersurface of degree $\deg X=d$, we have a short exact sequence
$$
0\longrightarrow T_X\longrightarrow T_{\mathbb P^{n+1}}|_X\longrightarrow \mathcal O_X(d)\longrightarrow 0;
$$
so we get the following relation for the total Chern class of $X$:
$$
(1+h)^{n+2}=(1+d\,h)c_\bullet(X),
$$
where $h=c_1(\mathcal O_{\mathbb P^{n+1}}(1))$ and $(1+h)^{n+2}$ is the total Chern class of $\mathbb P^{n+1}$. Thus, an easy computation shows that
$$
c_j(X)=c_j(T_X)=(-1)^jh^j\sum_{k=0}^j(-1)^k\binom{n+2}{k}d^{j-k},
$$
where $h\in H^2(X,\mathbb Z)$ is the hyperplane class. In particular
$$
c_j(X)=h^j\bigl((-1)^jd^j+o(d^j)\bigr),\quad j=1,\dots,n,
$$
and $o(d^j)$ is a polynomial in $d$ of degree at most $j-1$.

\begin{prop}\label{nef}
The quantities
$$
\begin{aligned}
&\bigl(\mathcal O_{X_k}(\mathbf a)\otimes\pi_{0,k}^*\mathcal O_X(2|\mathbf a|)\bigr)^{n+k(n-1)} \\
& -[n+k(n-1)]\bigl(\mathcal O_{X_k}(\mathbf a)\otimes\pi_{0,k}^*\mathcal O_X(2|\mathbf a|)\bigr)^{n+k(n-1)-1}\cdot\pi_{0,k}^*\mathcal O_X(2|\mathbf a|)
\end{aligned}
$$
and
$$
\mathcal O_{X_k}(\mathbf a)^{n+k(n-1)}
$$
are both polynomials in the variable $d$ with coefficients in $\mathbb Z[a_1,\dots,a_k]$ of degree at most $n+1$ and the coefficients of $d^{n+1}$ of the two expressions are equal.

Moreover this coefficient is a homogeneous polynomial in $a_1,\dots,a_k$ of degree $n+k(n-1)$ or identically zero.
\end{prop}

\begin{proof}
Set $\mathcal F_k(\mathbf a)=\mathcal O_{X_k}(\mathbf a)\otimes\pi_{0,k}^*\mathcal O_X(2|\mathbf a|)$ and $\mathcal G_k(\mathbf a)=\pi_{0,k}^*\mathcal O_X(2|\mathbf a|)$. Then we have
$$
\begin{aligned}
\mathcal F_k(\mathbf a)&^{n+k(n-1)}+[n+k(n-1)]\mathcal F_k(\mathbf a)^{n+k(n-1)-1}\cdot\mathcal G_k(\mathbf a)\\
&=\mathcal O_{X_k}(\mathbf a)^{n+k(n-1)}+\text{\rm terms which have $\mathcal G_k(\mathbf a)$ as a factor}.
\end{aligned}
$$
Now we use relations (\ref{chern2}) and (\ref{chern3}) to observe that
$$
\mathcal O_{X_k}(\mathbf a)^{n+k(n-1)}=\sum_{j_1+2j_2+\cdots+nj_n=n}P_{j_1\cdots j_n}^{[k]}(\mathbf a)\,c_1(X)^{j_1}\cdots c_n(X)^{j_n},
$$
where the $P_{j_1\cdots j_n}^{[k]}(\mathbf a)$'s are homogeneous polynomial of degree $n+k(n-1)$ in the variables $a_1,\dots,a_k$ (or possibly identically zero). Thus, substituting the $c_j(X)$'s with their expression in terms of the degree, we get
\begin{multline*}
\mathcal O_{X_k}(\mathbf a)^{n+k(n-1)} \\ =(-1)^n\left(\sum_{j_1+2j_2+\cdots+nj_n=n}P_{j_1\cdots j_n}^{[k]}(\mathbf a)\right)d^{n+1}+o(d^{n+1}),
\end{multline*}
since $h^n=d$. On the other hand, utilizing relations (\ref{chern2}) and (\ref{chern3}) on terms which have $\mathcal G_k(\mathbf a)$ as a factor, gives something of the form
$$
\sum_{\genfrac{}{}{0pt}{}{j_1+2j_2+\cdots+nj_n+i=n}{i>0}} Q_{j_1\cdots j_ni}^{[k]}(\mathbf a)\,h^i\cdot c_1(X)^{j_1}\cdots c_n(X)^{j_n},
$$
since $c_1(\mathcal G_k(\mathbf a))=|\mathbf a|h$ and $\mathcal G_k(\mathbf a)$ is always a factor. Substituting the $c_j(X)$'s with their expression in terms of the degree, we get here
$$
h^i\cdot c_1(X)^{j_1}\cdots c_n(X)^{j_n}=(-1)^{j_1+\cdots+j_n}\underbrace{h^n}_{=d}\cdot d^{j_1+\dots+j_n}=o(d^{n+1}).
$$
\end{proof}

At this point, we need an elementary lemma to deal with \lq\lq generic\rq\rq{} weights.

\begin{lem}\label{Zar}
Let $\mathfrak C\subset\mathbb R^k$ be a cone with nonempty interior. Let $\mathbb Z^k\subset\mathbb R^k$ be the canonical lattice in $\mathbb R^k$. Then $\mathbb Z^k\cap\mathfrak C$ is Zariski dense in $\mathbb R^k$.
\end{lem}

\begin{proof}
Since $\mathfrak C$ is a cone with nonempty interior, it contains cubes of arbitrary large edges, so $\mathbb Z^k\cap\mathfrak C$ contains a product of integral intervals $\prod[\alpha_i,\beta_i]$ with $\beta_i-\alpha_i>N$. By using induction on dimension, this implies that a polynomial $P$ of degree at most $N$ vanishing on $\mathbb Z^k\cap\mathfrak C$ must be identically zero. As $N$ can be taken arbitrary large, we conclude that  $\mathbb Z^k\cap\mathfrak C$ is Zariski dense.\end{proof}

\begin{rem}\label{effa}
This lemma is in fact the elementary key for the effective estimates on the degree. This will be explained in the next sections.
\end{rem}

\begin{cor}\label{bigness1}
If the top self-intersection $\mathcal O_{X_k}(\mathbf a)^{n+k(n-1)}$ has degree exac\-tly equal to $n+1$ in $d$ for some choice of $\mathbf a$, then $\mathcal O_{X_k}(m)\otimes\pi_{0,k}^*A^{-1}$ has a global section for all line bundle $A\to X$ and for all $d,m$ sufficiently large.
\end{cor}

\begin{proof}
The real $k$-tuples which satisfy condition (\ref{c}), form a cone with non-empty interior in $\mathbb R^k$.
Thus, by Lemma \ref{Zar}, there exists an integral $\mathbf a'$ satisfying condition (\ref{c}) and such that $\mathcal O_{X_k}(\mathbf a')^{n+k(n-1)}$ has degree exactly $n+1$ in $d$. For reasons similar to those in the proof of Proposition \ref{nef}, the coefficient of degree $n+1$ in $d$ of $\mathcal O_{X_k}(\mathbf a')^{n+k(n-1)}$ and $\bigl(\mathcal O_{X_k}(\mathbf a')\otimes\pi_{0,k}^*\mathcal O_X(2|\mathbf a'|)\bigr)^{n+k(n-1)}$ are the same; the second one being nef, this coefficient must be positive.

Now, by Proposition \ref{nef}, this coefficient is the same as the coefficient of degree $n+1$ in $d$ of
\begin{multline*}
\bigl(\mathcal O_{X_k}(\mathbf a')\otimes\pi_{0,k}^*\mathcal O_X(2|\mathbf a'|)\bigr)^{n+k(n-1)} \\
 -[n+k(n-1)]\bigl(\mathcal O_{X_k}(\mathbf a')\otimes\pi_{0,k}^*\mathcal O_X(2|\mathbf a'|)\bigr)^{n+k(n-1)-1}\cdot\pi_{0,k}^*\mathcal O_X(2|\mathbf a'|).
\end{multline*}
But then this last quantity is positive for $d$ large enough, and the corollary follows by an application of algebraic holomorphic Morse inequalities.
\end{proof}

\begin{cor}\label{bigness2}
For $k<n$, the coefficient of $d^{n+1}$ in the expression of 
$$
\mathcal O_{X_k}(\mathbf a)^{n+k(n-1)}
$$ 
is identically zero.
\end{cor}

\begin{proof}
Otherwise, we would have global sections of $\mathcal O_{X_k}(m)$ for $m$ large and $k<n$, which is impossible by Theorem \ref{vanishing}.
\end{proof}

\subsection{Bigness of $\mathcal O_{X_n}(1)$} Thanks to the results of the previous subsection, to show the existence of a global section of $\mathcal O_{X_n}(m)\otimes\pi_{0,n}^*A^{-1}$ for $m$ and $d$ large, we just need to show that $\mathcal O_{X_n}(\mathbf a)^{n^2}$ has degree exactly $n+1$ in $d$ for some $n$-tuple $(a_1,\dots,a_n)$.

The multinomial theorem gives
\begin{multline*}
(a_1\pi_{1,k}^*u_1+\cdots+a_ku_k)^{n+k(n-1)}\\ =\sum_{j_1+\dots+j_k=n+k(n-1)}\frac{(n+k(n-1))!}{j_1!\cdots j_k!}\,a_1^{j_1}\cdots a_k^{j_k}\, \pi_{1,k}^*u_1^{j_1}\cdots u_k^{j_k}.
\end{multline*}

We need two technical lemmas.

\begin{lem}\label{cruc1}
The coefficient of degree $n+1$ in $d$ of the two following intersections is zero:
\begin{itemize}
\item[$\bullet$] $\pi_{1,k}^*u_1^{j_1}\cdot\pi_{2,k}^*u_2^{j_2}\cdots u_k^{j_k}$ for all $1\le k\le n-1$ and $j_1+\dots+j_k=n+k(n-1)$
\item[$\bullet$] $\pi_{1,n-i-1}^*u_1^{j_1}\cdot\pi_{2,n-i-1}^*u_2^{j_2}\cdots u_{n-i-1}^{j_{n-i-1}}\cdot\pi_{0,n-i-1}^*c_1(X)^i$ for all $1\le i\le n-2$ and $j_1+\dots+j_{n-i-1}=(n-i-1)n+1$. 
\end{itemize}
\end{lem}

\begin{proof}
The first statement is straightforward: if it fails to be true, we would find an $\mathbf a$ which satisfies the hypothesis of Corollary \ref{bigness1} for $k<n$, contradicting Corollary \ref{bigness2}.  

For the second statement we proceed by induction on $i$. Let us start with $i=1$. By the first part of the present lemma, we have that 
$$
\pi_{1,n-1}^*u_1^{j_1}\cdot\pi_{2,n-1}^*u_2^{j_2}\cdots\pi_{n-1}^*u_{n-2}^{j_{n-2}}\cdot u_{n-1}^n=o(d^{n+1}).
$$ 
On the other hand, relation (\ref{chern3}) gives
$$
\begin{aligned}
&\pi_{1,n-1}^*u_1^{j_1}\cdot\pi_{2,n-1}^*u_2^{j_2}\cdots\pi_{n-1}^*u_{n-2}^{j_{n-2}}\cdot u_{n-1}^n \\
&=\pi_{1,n-1}^*u_1^{j_1}\cdot\pi_{2,n-1}^*u_2^{j_2}\cdots\pi_{n-1}^*u_{n-2}^{j_{n-2}}\\
&\quad \cdot\bigl(-\pi_{n-1}^*c_1^{[n-2]}\cdot u_{n-1}^{n-1}-\cdots-\pi_{n-1}^*c_{n-1}^{[n-2]}\cdot u_{n-1}-\pi_{n-1}^*c_{n}^{[n-2]}\bigr) \\
&=-\pi_{1,n-1}^*u_1^{j_1}\cdot\pi_{2,n-1}^*u_2^{j_2}\cdots\pi_{n-1}^*u_{n-2}^{j_{n-2}}\cdot\pi_{n-1}^*c_1^{[n-2]}\cdot u_{n-1}^{n-1}
\end{aligned}
$$
and the second equality is true for degree reasons: 
$$
u_1^{j_1}\cdot u_2^{j_2}\cdots u_{n-2}^{j_{n-2}}\cdot c_l^{[n-2]},\quad l=2,\dots,n,
$$ 
\lq\lq lives\rq\rq{} on $X_{n-2}$ and has total degree $n+(n-2)(n-1)-1+l$ which is strictly greater than $n+(n-2)(n-1)=\dim X_{n-2}$, so that $u_1^{j_1}\cdot u_2^{j_2}\cdots u_{n-2}^{j_{n-2}}\cdot c_l^{[n-2]}=0$. Now, we use relation (\ref{c1}) and obtain in this way
$$
\begin{aligned}
&\pi_{1,n-1}^*u_1^{j_1}\cdot\pi_{2,n-1}^*u_2^{j_2}\cdots\pi_{n-1}^*u_{n-2}^{j_{n-2}}\cdot u_{n-1}^n\\
&=-\pi_{1,n-1}^*u_1^{j_1}\cdot\pi_{2,n-1}^*u_2^{j_2}\cdots\pi_{n-1}^*u_{n-2}^{j_{n-2}}\cdot u_{n-1}^{n-1}\\
&\qquad\cdot\biggl(\pi_{0,n-1}^*c_1(X)+(n-1)\sum_{s=1}^{n-2}\pi_{s,n-1}^*u_s\biggr)\\
&=-\pi_{1,n-1}^*u_1^{j_1}\cdot\pi_{2,n-1}^*u_2^{j_2}\cdots\pi_{n-1}^*u_{n-2}^{j_{n-2}}\cdot u_{n-1}^{n-1}\cdot\pi_{0,n-1}^*c_1(X)\\
&\quad-(n-1)u_{n-1}^{n-1}\cdot\sum_{s=1}^{n-2}\pi_{1,n-1}^*u_1^{j_1}\cdots\pi_{s,n-1}^*u_s^{j_s+1}\cdots\pi_{n-1}^*u_{n-2}^{j_{n-2}}.
\end{aligned}
$$ 
An integration along the fibers of $X_{n-1}\to X_{n-2}$ then gives
$$
\begin{aligned}
\pi_{1,n-2}^*u_1^{j_1}&\cdot\pi_{2,n-2}^*u_2^{j_2}\cdots u_{n-2}^{j_{n-2}}\cdot\pi_{0,n-2}^*c_1(X)\\
&=-(n-1)\cdot\sum_{s=1}^{n-2}\underbrace{\pi_{1,n-2}^*u_1^{j_1}\cdots\pi_{s,n-2}^*u_s^{j_s+1}\cdots u_{n-2}^{j_{n-2}}}_{=o(d^{n+1})\,\,\text{\rm by the first part of the lemma}}\\
&\quad\quad+o(d^{n+1})
\end{aligned}
$$
and so $\pi_{1,n-2}^*u_1^{j_1}\cdot\pi_{2,n-2}^*u_2^{j_2}\cdots u_{n-2}^{j_{n-2}}\cdot\pi_{0,n-2}^*c_1(X)=o(d^{n+1})$.

To complete the proof, observe that -- as before -- relations (\ref{chern3}) and (\ref{c1}) together with a completely similar degree argument give 
$$
\begin{aligned}
&\pi_{1,n-i}^*u_1^{j_1}\cdot\pi_{2,n-i}^*u_2^{j_2}\cdots u_{n-i}^{j_{n-i-1}}\cdot\pi_{0,n-i}^*c_1(X)^i\cdot u_{n-i}^n \\
&=-\pi_{1,n-i}^*u_1^{j_1}\cdot\pi_{2,n-i}^*u_2^{j_2}\cdots\pi_{n-i}^*u_{n-i-1}^{j_{n-i-1}}\cdot u_{n-i}^{n-1}\cdot\pi_{0,n-i}^*c_1(X)^{i+1}\\
&\quad-(n-1)u_{n-i}^{n-1}\cdot\sum_{s=1}^{n-i-1}\pi_{1,n-i}^*u_1^{j_1}\cdots\pi_{s,n-i}^*u_s^{j_s+1}\cdots\pi_{n-i}^*u_{n-i-1}^{j_{n-i-1}}.
\end{aligned}
$$
But 
$$
\pi_{1,n-i}^*u_1^{j_1}\cdot\pi_{2,n-i}^*u_2^{j_2}\cdots u_{n-i}^{j_{n-i-1}}\cdot\pi_{0,n-i}^*c_1(X)^i\cdot u_{n-i}^n=o(d^{n+1})
$$ 
by induction, and 
$$
\pi_{1,n-i}^*u_1^{j_1}\cdots\pi_{s,n-i}^*u_s^{j_s+1}\cdots\pi_{n-i}^*u_{n-i-1}^{j_{n-i-1}}=o(d^{n+1}),
$$ 
$1\le s\le n-i-1$, thanks to the first part of the lemma.
\end{proof}

\begin{lem}\label{cruc2}
The coefficient of degree $n+1$ in $d$ of 
$$
\pi_{1,n}^*u_1^n\cdot\pi_{2,n}^*u_2^n\cdots u_n^n
$$ 
is the same of the one of $(-1)^nc_1(X)^n$, that is $1$.
\end{lem}

\begin{proof}
An explicit computation yields:
$$
\begin{aligned}
\pi_{1,n}^*u_1^{n}\cdot \pi_{2,n}^* & u_2^n\cdots u_n^{n}\overset{(i)}= \pi_{1,n}^*u_1^{n}\cdot \pi_{2,n}^*u_2^n\cdots \pi_{n}^*u_{n-1}^{n}\bigl(-\pi_n^*c_1^{[n-1]}\cdot u_n^{n-1} \\ &\quad\quad-\cdots-\pi_n^*c_{n-1}^{[n-1]}\cdot u_n-\pi_n^*c_{n}^{[n-1]}\bigr) \\
&\overset{(ii)}=-\pi_{1,n}^*u_1^{n}\cdot \pi_{2,n}^*u_2^n\cdots\pi_{n}^* u_{n-1}^{n}\cdot u_n^{n-1} \cdot \pi_n^*c_1^{[n-1]} \\
&\overset{(iii)}=-\pi_{1,n}^*u_1^{n}\cdot \pi_{2,n}^*u_2^n\cdots \pi_{n}^*u_{n-1}^{n}\cdot u_n^{n-1} \\ &\quad\quad\cdot\pi_n^*\biggl(\pi_{0,n-1}^*c_1(X)+(n-1)\sum_{s=1}^{n-1}\pi_{s,n-1}^*u_s\biggr) \\
&\overset{(iv)}=-\pi_{1,n}^*u_1^{n}\cdot \pi_{2,n}^*u_2^n\cdots \pi_{n}^*u_{n-1}^{n}\cdot u_n^{n-1}\cdot\pi_{0,n}^*c_1(X) \\ &\quad\quad+o(d^{n+1})\\
&=\cdots \\
&\overset{(v)}=(-1)^n\pi_{0,k}^*c_1(X)^n\cdot\pi_{1,k}^*u_1^{n-1}\cdots u_n^{n-1}+o(d^{n+1}) \\
&\overset{(vi)}=(-1)^nc_1(X)^n+o(d^{n+1}).
\end{aligned}
$$
Let us say a few words about the previous equalities. Equality (i) is just relation (\ref{chern3}). Equality (ii) is true for degree reasons: $u_1^{n}\cdot u_2^n\cdots u_{n-1}^{n}\cdot c_l^{[n-1]}$, $l=2,\dots,n$, \lq\lq lives\rq\rq{} on $X_{n-1}$ and has total degree $n(n-1)+l$ which is strictly greater than $n+(n-1)(n-1)=\dim X_{n-1}$, so that $u_1^{n}\cdot u_2^n\cdots u_{n-1}^{n}\cdot c_l^{[n-1]}=0$. Equality (iii) is just relation (\ref{c1}). Equality (iv) follows from the first part of Lemma \ref{cruc1}: $u_1^n\cdots u_s^{n+1}\cdots u_{n-1}^n=o(d^{n+1})$. Equality (v) is obtained by applying repeatedly the second part of Lemma \ref{cruc1}. Finally, equality (vi) is simply integration along the fibers. The lemma is proved.
\end{proof}

Now, look at the coefficient of degree $n+1$ in $d$ of the expression
$$
\mathcal O_{X_n}(\mathbf a)^{n^2}=\bigl(a_1\pi_{1,n}^*u_1+\cdots+a_n u_n\bigr)^{n^2},
$$
where we consider the $a_j$'s as variables: we claim that it is a non identically zero homogeneous polynomial of degree $n^2$. To see this, we just observe that, thanks to Lemma \ref{cruc2}, the coefficient of the monomial $a_1^n\cdots a_n^n$ is $(n^2)!/(n!)^n$.

Hence there exists an $\mathbf a$ which satisfies the hypothesis of Corollary \ref{bigness1} for $k=n$, and the existence of jet differentials is proved.

\subsection*{Existence of jet differential with controlled vanishing}

In our applications, it will be crucial to be able to control in a more precise way the order of vanishing of these differential operators along the ample divisor. Thus, we shall need here a slightly different version of Theorem \ref{existence}.

\begin{thm}[\cite{D-M-R10}]\label{existenceKX}
Let $X\subset\mathbb P^{n+1}$ by a smooth complex hypersurface of degree $\deg X=d$. Then, for all positive rational numbers $\delta$ small enough, there exists a positive integer $d_n$ such that the following non-vanishing holds:
$$
H^0(X_n,\mathcal O_{X_n}(m)\otimes\pi^*_{0,n}K_X^{-\delta m})=H^0(X,E_{n,m}T^*_X\otimes K_X^{-\delta m})\ne 0,
$$
provided $d\ge d_n$ and $m$ is large and divisible enough.
\end{thm}

\begin{rem}
Recall that for $X$ a smooth projective hypersurface of degree $d$ in $\mathbb P^{n+1}$, the canonical bundle has the following expression in terms of the hyperplane bundle:
$$
K_X\simeq \mathcal O_{X}(d-n-2).
$$
Thus, a non-zero section in $H^0(X,E_{n,m}T^*_X\otimes K_X^{-\delta m})$, has vanishing order at least $\delta m(d-n-2)$, seen as a section of $H^0(X,E_{n,m}T^*_X)$.
\end{rem}

\begin{proof}[Proof of Theorem \ref{existenceKX}]
For each weight $\mathbf{a}\in\mathbb N^n$ which satisfies (\ref{c}), we first of all express $\mathcal O_{X_n}(\mathbf{a})\otimes\pi^{*}_{0,n}K_X^{-\delta|\mathbf{a}|}$ as the difference of two nef line bundles:
\begin{multline*}
\mathcal O_{X_n}(\mathbf{a})\otimes\pi^{*}_{0,n}K_X^{-\delta|\mathbf{a}|}\\ 
=\bigl(
\mathcal O_{X_n}(\mathbf{a})\otimes\pi^*_{0,n}\mathcal O_X(2|\mathbf{a}|)\bigr)\otimes\bigl(\pi^*_{0,n}\mathcal O_X(2|\mathbf{a}|)\otimes\pi^{*}_{0,n}K_X^{\delta|\mathbf{a}|})\bigr)^{-1}.
\end{multline*}
This leads to evaluate, in order to apply holomorphic Morse inequalities, the following intersection product:
\begin{multline*}
\bigl(\mathcal O_{X_n}(\mathbf{a})\otimes\pi^*_{0,n}\mathcal O_X(2|\mathbf{a}|)\bigr)^{n^2}\\
-n^2\bigl(\mathcal O_{X_n}(\mathbf{a})\otimes\pi^*_{0,n}\mathcal O_X(2|\mathbf{a}|)\bigr)^{n^2-1}\cdot\bigl(\pi^*_{0,n}\mathcal O_X(2|\mathbf{a}|)\otimes\pi^{*}_{0,n}K_X^{\delta|\mathbf{a}|})\bigr).
\end{multline*}
After elimination, this intersection product gives back a polynomial in $d$ of degree less than or equal to $n+1$, whose coefficients are polynomial in $\mathbf{a}$ and $\delta$ of bidegree $(n^2,1)$, homogeneous in $\mathbf a$. Notice that, for $\delta=0$, this gives back (\ref{*}) for which we know that there exists a weight $\mathbf{a}$ satisfying condition (\ref{c}) such that  this polynomial has exactly degree $n+1$ and positive leading coefficient, which is what we need to achieve to proof. Thus, by continuity, for the same weight $\mathbf{a}$, for all $\delta>0$ small enough, we get the same conclusion.
\end{proof}

\section{Meromorphic vector fields}

The other ingredient in the proof is, as we have seen in the previous chapter, the existence of enough global meromorphic vector fields with controlled pole order on the space of vertical jets of the universal hypersurface: this result has been obtained in the case of arbitrary dimension in \cite{Mer09} and generalizes the corresponding versions in dimension $2$ and $3$ contained respectively in \cite{Pau08} and \cite{Rou07}.

Before stating the theorem, we fix once again the notations. Let $\mathcal X\subset\mathbb P^{n+1}\times\mathbb P^{N_d}$ be the universal projective hypersurface of degree $d$ in $\mathbb P^{n+1}$, whose moduli space is the projectivization 
$$
\mathbb P\bigl(H^0(\mathbb P^{n+1},\mathcal O(d))\bigr)=\mathbb P^{N_d-1},\quad N_d=\binom{n+d+1}{d};
$$
then we have two natural projections
$$
\xymatrix{
& \mathcal X \ar[dl]_{\text{pr}_1} \ar[dr]^{\text{pr}_2} &  \\
\mathbb P^{n+1} & & \mathbb P^{N_d-1}.}
$$
Consider the relative tangent bundle $\mathcal V\subset T_{\mathcal X}$ with respect to the second projection $\mathcal V=\ker(\text{pr}_2)_*$, and form the corresponding directed manifold $(\mathcal X,\mathcal V)$.

Now, let $p\colon J_n\mathcal V\to\mathcal X$ be the bundle of $n$-jets of germs of holomorphic curves in $\mathcal X$ tangent to $\mathcal V$, \emph{i.e.} of vertical jets, and consider the subbundle $J_n\mathcal V^{\text{\rm reg}}$ of regular $n$-jets of maps $f\colon(\mathbb C,0)\to\mathcal X$ tangent to $\mathcal V$ such that $f'(0)\ne 0$.  

\begin{thm}[\cite{Mer09}]\label{globgen}
The twisted tangent space to vertical $n$-jets
$$
T_{J_n\mathcal V}\otimes p^*\text{pr}_1^*\mathcal O_{\mathbb P^{n+1}}(n^2+2n)\otimes p^*\text{pr}_2^*\mathcal O_{\mathbb P^{N_d-1}}(1)
$$
is generated over $J_n\mathcal V^{\text{\rm reg}}$ by its global holomorphic sections.

Moreover, one may choose such global generating vector fields to be invariant with respect to the action of $\mathbb G_n$ on $J_n\mathcal V$.
\end{thm}

The proof of this theorem is in essence the same of the one presented above for the two dimensional case. Nevertheless the computational and combinatorial aspects, as one can guess, are much more involved. We refer to the original paper of Merker for a complete proof.

Note that this statement is stronger than the one described in the previous chapter: the global generation is over a bigger open subset of $J_n\mathcal V$, no wronskian locus appears. The price for this is a bigger order of poles (for $n=2$, the pole order here is $8$ instead of $7$), but this permits a more precise localization of the entire curves. 

There is also a weaker version of this theorem, which is the precise generalization of the statement in \cite{Pau08} and \cite{Rou07}, which gives better pole order. This version gives global generation outside the wronskian locus $\Sigma$, with pole order $(n^2+5n)/2$, and it suffices for instance to treat the case of threefold in projective $4$-space (see Corollary \ref{kob3}).

Of course, the pole order is important as far as the effective aspects on the degree of the hypersurfaces are concerned.

\section{Effective aspects}

In this section we shall try to outline the idea of the proof for the effective part of Theorem \ref{mainthm}. The interested reader can find all the details in \cite{D-M-R10}: we shall skip this here, since the combinatorics and the complexity of the computations are very involved.

Recall that we have to show that if $d\ge 2^{n^5}$, then we have the existence of the proper subvariety $Y\subset X$ absorbing the images of non constant entire curves, for $X\subset\mathbb P^{n+1}$ a generic smooth projective hypersurface of degree $d$.

As a byproduct of the proof of the theorem we gave above, we deduce that the bound for the degree depends on

\begin{itemize}
\item[(i)] the lowest integer $d_1$ such that $\delta> (n^2+2n)/(d_2-n-2)$, whenever $\deg X=d_1$; 
\item[(ii)] the lowest integer $d_2$ such that $H^0(X,E_{n,m}T^*_X\otimes K_X^{-\delta m})\ne 0$ for $m\gg 1$, whenever $\deg X=d_2$.
\end{itemize}

Then, $d_0=\max\{d_1,d_2\}$ will do the job.

\subsection{The strategy of the effective estimate}

The starting point of our effective estimate is the following, elementary lemma.
 
\begin{lem}
\label{d}
Let $p(z)=z^d+a_1\,z^{d-1}+\cdots+a_d\in\mathbb C[z]$ be a monic polynomial of degree $d$ and let $z_0$ be a root of $p$. Then
$$
|z_0|\le 2\max_{j=1,\dots,d}|a_j|^{1/j}.
$$
\end{lem}

\begin{proof}
Otherwise $|z_0|>2|a_j|^{1/j}$ for every $j=1,\dots,d$ and, from $-1=a_1/z_0+\cdots+a_d/z_0^d$ we would obtain
$$
1\le 2^{-1}+\cdots+2^{-d},
$$
contradiction.
\end{proof}

Recall that, in order to produce a global invariant jet differential with controlled vanishing order, we had to ensure the positivity of a certain intersection product, namely
\begin{multline*}
\bigl(\mathcal O_{X_n}(\mathbf{a})\otimes\pi^*_{0,n}\mathcal O_X(2|\mathbf{a}|)\bigr)^{n^2}\\
-n^2\bigl(\mathcal O_{X_n}(\mathbf{a})\otimes\pi^*_{0,n}\mathcal O_X(2|\mathbf{a}|)\bigr)^{n^2-1}\cdot\bigl(\pi^*_{0,n}\mathcal O_X(2|\mathbf{a}|)\otimes\pi^{*}_{0,n}K_X^{\delta|\mathbf{a}|})\bigr),
\end{multline*}
for some $\mathbf a$ satisfying (\ref{c}) and $\delta>0$ (cf. Theorem \ref{existenceKX}). This intersection product, after elimination, gives back a polynomial in the degree of the hypersurfaces (seen as an indeterminate) of degree $n+1$, whose coefficients are polynomial in $\mathbf a$ and $\delta$ of bidegree $(n^2,1)$, homogeneous in $\mathbf a$. Call it
$$
{\sf P}_{\mathbf a,\delta}(d)={\sf P}_{\mathbf a}(d)+\delta\,{\sf P}'_{\mathbf a}(d).
$$  
Observe once again that the above polynomial with $\delta=0$ is the result of the intersection product (\ref{*}), the positivity of which gives the existence of global invariant jet differentials with no control on their vanishing order.

Now, suppose we are able to choose the weight $\mathbf a$ \emph{explicitly, depending on $n$}, in such a way that it satisfies (\ref{c}) and that ${\sf P}_{\mathbf a}$ has positive leading coefficient and we have an explicit control of its coefficients in terms of $n$ only. This will be possible, as anticipated in Remark \ref{effa}, by the following basic observation.

\begin{prop}
The cube in $\mathbb R^n_{(a_1,\dots,a_n)}$ with integral edges defined by
\begin{multline*}
1\le a_n\le 1+n^2,\quad 3n^2\le a_{n-1}\le(3+1)n^2, \\
(3^2+3)n^2\le a_{n-2}\le (3^2+3+1)n^2,\quad\dots \\
(3^{n-1}+\cdots+3)n^2\le a_1\le (3^{n-1}+\cdots+3+1)n^2 
\end{multline*}
is contained in the cone defined by (\ref{c}) and there exists at least one $n$-tuple $\mathbf a$ of integers belonging to this cube with the property that ${\sf p}_{n+1,\mathbf a}$ is non zero.

Moreover, for such an $\mathbf a$ we have the effective, explicit control
$$
\max_{1\le j\le n}a_j=a_1\le\frac{3^n-1}2n^2\le\frac{3^n}2n^2.
$$
\end{prop}

\begin{proof}
This is just an explicit rephrasing of Lemma \ref{Zar}.
\end{proof}

 In addition, suppose we have for this choice of the weight an explicit control of the coefficients of ${\sf P}_{\mathbf a}'$ and that its leading coefficient is negative (the latter trivially holds true since otherwise (\ref{*}) would be positive for $\delta\gg 1$, which is impossible, since $K_X$ is ample). Write
$$
{\sf P}_{\mathbf a}(d)=\sum_{k=0}^{n+1}{\sf p}_{k,\mathbf a}\,d^k,\quad {\sf P}_{\mathbf a}'(d)=\sum_{k=0}^{n+1}{\sf p}_{k,\mathbf a}'\,d^k,
$$ 
with $|{\sf p}_{j,\mathbf a}|\le{\sf E}_j$, $j=0,\dots,n$, ${\sf p}_{n+1,\mathbf a}\ge{\sf G}_{n+1}\ge 1$ and $|{\sf p}_{j,\mathbf a}'|\le{\sf E}_j'$, $j=0,\dots,n+1$, the aforesaid explicit, depending on $n$ only, bounds.

Next, choose $\delta$ to be one half of ${\sf G}_{n+1}/{\sf E}_{n+1}'$ so that the leading coefficient ${\sf p}_{n+1,\mathbf a}-\delta|{\sf p}_{n+1,\mathbf a}'|$ of ${\sf P}_{\mathbf a,\delta}(d)$ is bounded below by $1/2$. In order to apply Lemma \ref{d}, we divide ${\sf P}_{\mathbf a,\delta}(d)$ by its leading coefficient and thus we obtain that the biggest integral root of ${\sf P}_{\mathbf a,\delta}(d)$ is less than or equal to 
$$
d_2\overset{\text{def}}=2\max_{j=1,\dots,n+1}\left(2\,{\sf E}_j+\frac{{\sf G}_{n+1}}{{\sf E}_{n+1}'}{\sf E_{j}'} \right)^{\frac 1{n+1-j}}.
$$
In order to have not only the existence of global invariant jet differentials with controlled vanishing order, but also algebraic degeneracy, we have to insure condition (i) above to be satisfied, too. This means that $d$ must be greater than 
$$
d_1\overset{\text{def}}=n+2+2\,\frac {{\sf E}_{n+1}'}{{\sf G}_{n+1}}(n^2+2n).
$$
In conclusion, we would have the \emph{effective} estimate of Theorem \ref{mainthm} if 
\begin{equation}
\label{finalestimate}
\max\{d_1,d_2\} < 2^{n^{5}}.
\end{equation}

Thus, inequality (\ref{finalestimate}) amounts to bound the quantities ${\sf E}_j$, ${\sf E}_j'$, ${\sf G}_{n+1}$, $j=1,\dots,n+1$. For that, one has to control (\ref{*}) by estimating at each step the process of elimination of all the Chern classes living at each level of the jet tower, which happens to be of high algebraic complexity. The reason sits in the intertwining of four different combinatoric aspects: the presence of several relations shared by all the Chern classes of the lifted horizontal contact distributions $V_k$, the Newton expansion of large $n^2$-powers by the multinomial theorem, the differences of various multinomial coefficients and the appearance of many Jacobi-Trudy type determinants.

As said before, we refer to the original paper \cite{D-M-R10} for the details of these computations.

\backmatter

\bibliography{Bibliography}{}

\end{document}